\documentclass[10pt,reqno]{amsart} 
\usepackage[margin=1in]{geometry} 
\usepackage[utf8]{inputenc} 
\usepackage[T1]{fontenc}

\usepackage{amsmath,amsthm,amsfonts,amssymb,amscd}
\usepackage{enumitem}
\usepackage{appendix}
\usepackage{times}
\usepackage{esint,stackrel}
\usepackage{enumitem}
\usepackage{bbm}
\usepackage{listings}
\usepackage{cancel}

\usepackage{xfrac}
\usepackage{cases}
\usepackage{emptypage}

\def\p{\partial} 
\def\eps{\varepsilon}
\def\les{\lesssim}

\newcommand{\inorm}[1]{\| #1 \|}

\newcommand*{\sgn}{\ensuremath{\mathrm{sgn\,}}}
\newcommand{\RR}{\mathbb R}

\newcommand{\TT}{\mathbb T}
\newcommand{\OO}{\mathcal O}

\renewcommand*{\tilde}{\widetilde}
\renewcommand*{\hat}{\widehat}
\renewcommand*{\bar}{\overline}
\newcommand*{\Id}{\ensuremath{\mathrm{Id\,}}}

\renewcommand{\epsilon}{\varepsilon}

\newtheorem{theorem}{Theorem}[section]
\newtheorem{lemma}[theorem]{Lemma}
\newtheorem{proposition}[theorem]{Proposition}
\newtheorem{corollary}[theorem]{Corollary}

\newtheorem{remark}[theorem]{Remark}

\numberwithin{equation}{section}

\newcommand{\be}{\begin{equation}}
\newcommand{\ee}{\end{equation}}

\newcommand{\R}{\mathbb{R}}

\def\cir{ \! \circ \! }

\newcommand{\loc}{\text{loc}}

\newcommand{\W}{\mathring{W}}
\newcommand{\Z}{\mathring{Z}}
\newcommand{\K}{\mathring{K}}

\newcommand{\C}{\mathcal{C}}
\newcommand{\CC}{\mathbb{C}}
\newcommand{\E}{\mathcal E}
\newcommand{\ZZ}{\mathbb{Z}}

\newcommand{\ndash}{--}

\newcommand{\x}{\mathring{x}}
\newcommand{\T}{\mathring{T}}

\usepackage[colorlinks=true, pdfstartview=FitV, linkcolor=blue,citecolor=blue, urlcolor=blue]{hyperref}

\title[An infinite hierarchy of H\"older cusps for 1D Euler]{Gradient catastrophes and an infinite hierarchy of H\"older\\ cusp-singularities for 1D Euler}

\author{Isaac Neal}
\address{Department of Mathematics, University of California Davis, Davis, CA 95616.}
\email{\href{ineal@ucdavis.edu}{ineal@ucdavis.edu}}

\author{Steve  Shkoller}
\address{Department of Mathematics, University of California Davis, Davis, CA 95616.}
\email{\href{shkoller@math.ucdavis.edu}{shkoller@math.ucdavis.edu}}

\author{Vlad Vicol}
\address{Courant Institute of Mathematical Sciences, New York University, New York, NY 10012.}
\email{\href{vicol@cims.nyu.edu}{vicol@cims.nyu.edu}}

\begin{document}

\begin{abstract}
We establish an infinite hierarchy of finite-time gradient catastrophes for smooth solutions of the 1D Euler equations of gas dynamics with non-constant entropy.  Specifically, for all integers $n\geq 1$, we prove that there exist classical solutions, emanating from  smooth, compressive,  and non-vacuous initial data, which form a cusp-type gradient singularity in finite time,  in which the gradient of the solution has precisely  $C^{0,\frac{1}{2n+1}}$ H\"older-regularity. We show that such Euler solutions are codimension-$(2n-2)$  stable in the Sobolev space $W^{2n+2,\infty}$.
\end{abstract}

\maketitle

\setcounter{tocdepth}{1}
\tableofcontents

\allowdisplaybreaks

\section{Introduction}
One of the central problems in the mathematical theory of the compressible Euler equations, the canonical macroscopic model of gas dynamics, is  the {\it{formation}} of gradient catastrophes  from {\it{smooth initial data}}. The archetypal {\it singular solution} contains a {\it{shock}}, a codimension-1 subset of spacetime across which the solution experiences jump discontinuities.  Compression in the initial conditions ensures that smooth sound waves steepen 
further and further, until an infinite gradient is formed along a spacetime set of ``first gradient singularities'', which we call the {\it pre-shock}. For localized one-dimensional initial data, the pre-shock is a point in spacetime, along which the gradient of the Euler solution  forms a H\"{o}lder-type cusp singularity, in which the state variables remain H\"older continuous, but with gradients that  blow up. It is only {\it {after the pre-shock}}, instantaneously, that the state variables become discontinuous, resulting in an entropy-producing shock wave. As such, a detailed analysis of the formation of the gradient singularity at the pre-shock is crucial to our understanding of the transition from smooth to discontinuous solutions of the Euler system. 

The goal of this paper is to prove that the one-dimensional Euler dynamics, starting from  from smooth and non-vacuous initial data,  can attain an infinite hierarchy of finite-time  cusp-type singularities at the pre-shock,
with a H\"{o}lder exponent that is indexed by an integer $n\geq 1$.   By a slight abuse of terminology, we shall refer to the Euler solution at the time of the
first gradient blowup as the {\it pre-shock solution} or simply as the {\it pre-shock}.
 Then,  the $n^{\rm th}$  pre-shock solution  in this hierarchy corresponds to a cusp-singularity with H\"older exponent $\sfrac{1}{(2n+1)}$. We analyze the stability of these singular pre-shock solutions and establish that the set of initial conditions leading to a $C^{0,\frac{1}{2n+1}}$ cusp forms a $(2n-2)$-dimensional Banach submanifold of the Sobolev space $W^{2n+2, \infty}$; in particular, for all $n\geq 2$, only finite codimension stability holds.

The case $n=1$, corresponding to the fully-stable $C^{0,\frac 13}$ pre-shock, has been fully investigated in 1D and under symmetry in \cite{Le1994,ChDo2001,Ko2002,Yi2004,BuShVi2022,BuDrShVi2022,NeShVi2023}, and in multiple space dimensions in \cite{BuShVi2023a,BuShVi2023b,ShVi2024}. 
The {\em only} prior result concerning the formation of an unstable pre-shock was obtained in~\cite{BuIy2022}, which treats the special case $n=2$ in the setting of  azimuthal symmetry, using a modulated stability analysis of an explicit self-similar blowup profile. 

Herein, we simultaneously analyze all $C^{0,\frac{1}{2n+1}}$ pre-shock solutions in the cusp hierarchy (meaning,  all $n\geq 1$) without making reference to an self-similar blowup profiles. We achieve this by taking a {\it{purely characteristic perspective}}, in which the pre-shock singularity is characterized by the 
narrowing distance between the nearby fast-acoustic characteristics. In this Lagrangian-like characteristic framework, we show that carefully designed {\it{differentiated Riemann variables}} remain uniformly smooth up to the time of the first gradient singularity. In turn, this allows us to apply the  Implicit Function Theorem in a smooth setting,  to precisely characterize the finite-codimension stability.  Our method of analysis for stable and unstable shock formation is sufficiently robust so as
to easily generalize to more complicated hyperbolic systems of conservation laws.

\subsection{The 1D Euler equations}
The one-dimensional Euler equations are given as the system of conservation laws
\begin{subequations} 
\label{euler-weak}
\begin{align}
\partial_t \rho + \p_y (  \rho u ) &=0\,, \label{ee1}\\
\partial_t (\rho u)  + \p_y (  \rho u^2 + p  ) & =0\,,  \label{ee2}\\
\partial_t E + \p_y ( (p+ E) u ) &=0 \label{ee3}\, ,
\end{align}
\end{subequations}
where the unknowns $u, \rho, E,$ and $p$ are scalar functions defined for $(y,t) \in  \TT \times \R$. Here $u$ is the fluid velocity, $\rho$ is the (strictly positive) density, $E$ is the specific total energy, and $p$ is the pressure. To close the system, one introduces an equation of state that relates the internal energy $E-\tfrac{1}{2}\rho u^2$ to $p$ and $\rho$. For an  ideal gas, the equation of state is
\begin{equation}
\label{eq:ig}
p  = (\gamma-1)(E-\tfrac{1}{2}\rho u^2),
\end{equation}
where $\gamma >1$ is a fixed adiabatic exponent.

For the study of shock formation, it is convenient to express the internal energy and pressure as functions of $\rho$ and the specific entropy $S$, and to 
express the Euler evolution in terms of the variables $(u, \rho,S)$. For ideal gases the specific entropy $S$ is defined by the relation 
\begin{equation}
\label{idealgas}
p = \tfrac{1}{\gamma} \rho^\gamma e^S.
\end{equation}
Introducing the rescaled adiabtic exponent $\alpha : = \tfrac{\gamma-1}{2}$ and defining the rescaled sound speed\footnote{The actual sound speed in a compressible fluid is $c =\sqrt{\sfrac{\p p}{\p \rho} } = e^{\sfrac{S}{2}} \rho^\alpha = \alpha \sigma$.} by
\begin{equation}
\sigma = \tfrac{1}{\alpha} \sqrt{\tfrac{\p p}{\p \rho} } = \tfrac{1}{\alpha} e^{\sfrac{S}{2}} \rho^\alpha \,,
\end{equation}
we can rewrite the one-dimensional Euler equations \eqref{euler-weak}--\eqref{eq:ig} in terms of the variables $(u,\sigma,S)$ as 
\begin{subequations}
\label{eq:euler:2}
\begin{align}
\p_t \sigma +  u \p_y\sigma + \alpha \sigma \p_y  u & = 0 \,,
\label{eq:euler:2:a}  \\
\p_tu + u \p_y u+ \alpha \sigma \p_y \sigma & = \tfrac{\alpha}{2\gamma} \sigma^2 \p_y S \,, 
\label{eq:euler:2:b} \\
\p_t S + u \p_y S & = 0 \,.
\label{eq:euler:2:c}
\end{align}
\end{subequations}
Prior to the development of a discontintuous shock wave,  the systems \eqref{euler-weak}--\eqref{eq:ig} and~\eqref{idealgas}--\eqref{eq:euler:2} are equivalent.

\subsection{Prior results}
\label{sec:priorresults}
The theory of  shock wave solutions for systems of hyperbolic conservation laws in {\it{one space dimension}} is well-developed. For a detailed exposition we refer the reader to~\cite{CoFr1948,LaLi1987,ChWa2002,Da2005,Li2021} and the references therein. Herein, we  shall only summarize results concerning the finite-time formation of gradient singularities for the Euler equations, from smooth initial conditions. 

The fact that smooth Euler solutions generically have a finite lifespan is well-known~\cite{Ri1860,La1964,Jo1974,Li1979,Ma1984,Si1985}; such results rely upon a  {\it{proof by contradiction}} argument to show that a finite-time breakdown of the smooth solution occurs. 
While such arguments apply in great generality, they  generally do not  provide detailed information about the emerging gradient catastrophe. We emphasize that a precise characterization of the first gradient singularity is crucial for solving the physical shock development problem (see, for example, the discussion in  \cite{BuDrShVi2022b}).

Recently, a number of {\it{constructive singularity formation}} results  have 
been established for the Euler equations. Constructive results show that for an appropriate class of smooth initial data,  the {\em very first} gradient singularity can be precisely described, either as being of {\it{shock-type}}\footnote{For shocks the gradients of the fundamental variables (such as velocity and density) blow up, while the fundamental variables themselves remain bounded, and even retain limited H\"older regularity, which characterizes the type of cusp that forms.} or as being an {\it{implosion}}\footnote{For implosions the fundamental variables (such as velocity and density) blow up themselves, along with their gradients.}, and that no other singularity can occur prior. Since in one space dimension neither implosion\footnote{Smooth solutions that implode at the very first singularity have been recently constructed rigorously in {\it{multiple space dimensions}}, for the isentropic problem in radial symmetry~\cite{MeRaRoSz2022a, MeRaRoSz2022b, BuCLGS2022}. These results have been extended perturbatively to hold also outside of symmetry constraints~\cite{CLGSShSt2023}, and they have been used to establish the finite time blowup of vorticity~\cite{ChCiShVi2024,Ch2024} for the multi-dimensional Euler system. We note however that all known smooth implosions have only been proven to be finite codimensions stable, and that the dimension of the instability remains unquantified.} 
 nor vacuum formation  is possible as a first singularity from smooth solutions~\cite{ChYoZh2013,Chen2017}, we henceforth focus exclusively on constructive proofs of shock-type singularities. More precisely, when the fundamental variables $(u,\sigma,S)$ remain H\"older continuous at the time of the first gradient catastrophe, we refer to the emerging singularity as a {\it{pre-shock}}.

With the exception of the recent papers~\cite{BuIy2022} and~\cite{NeRiShVi2024}, all existing constructive blowup results for the Euler system establish the formation of a $C^{0,\frac{1}{3}}$ H\"{o}lder pre-shock as the first singularity. This is because the $C^{0,\frac{1}{3}}$ pre-shock emerges in a {\it{stable}} fashion from a large open set of smooth and generic initial data.  A compendium is listed as follows:
\begin{itemize}[leftmargin=*]
\item The paper~\cite{Le1994} established both shock formation (as a $C^{0,\frac{1}{3}}$ pre-shock) and shock development for the one-dimensional $p$-system. Further refinements were obtained  in~\cite{ChDo2001} and~\cite{Ko2002}. For the three-dimensional Euler equations in spherical symmetry, \cite{Yi2004} used similar ideas to construct solutions which form a $C^{0,\frac{1}{3}}$ cusp (with respect to the radial coordinate) in finite time, and to show that the resulting weak entropy solution instantaneously develops a radial shock (jump discontinuity). 

\item Results similar to~\cite{Yi2004} have been established in~\cite{ChLi2016}, by using techniques originating in mathematical general relativity. More precisely,~\cite{ChLi2016} appeals to part of the framework developed in~\cite{Ch2007}. The monograph~\cite{Ch2007} established shock formation for multi-dimensional irrotational flows (see also~\cite{ChMi2014}), by studying the second order wave equation satisfied by the velocity potential. This framework was also used in~\cite{Ch2019} to analyze the {\it{restricted}}\footnote{In the {\it{restricted}} shock development problem,  the flow is artificially constrained to remain irrotational and isentropic even after the shock forms, which violates the second law of thermodynamics and the Rankine-Hugoniot jump conditions.} shock development problem for multi-dimensional irrotational and isentropic flows, and it was used in~\cite{LuSp2018,LuSp2024} to prove shock formation for the two- and three-dimensional Euler equations for initial data close to a plane wave, in the presence of vorticity and entropy. We emphasize that the shock formation results obtained in~\cite{Ch2007,ChMi2014,ChLi2016,Ch2019,LuSp2018,LuSp2024} do not provide the precise H\"older regularity of the Euler solution at the time of the very first singularity.

\item An alternative approach to shock formation for the Euler equations was proposed in~\cite{BuShVi2022}. This approach employs modulated self-similar  analysis\footnote{Finite-time blowup results for $2$+$1$-dimensional second-order quasilinear wave equations, which do not satisfy the null-condition (akin to irrotational isentropic Euler), were previously obtained in~\cite{Al1999a,Al1999b}. The proofs in~\cite{Al1999a,Al1999b} use a ``blowup technique'' which resembles self-similar analysis, but without the modulation parameters corresponding to unstable modes arising from the Galilean symmetry of the equations.} to establish the stability of an explicit blowup profile. The paper~\cite{BuShVi2022} considers the two-dimensional Euler equations in azimuthal symmetry and establishes the stable formation of a $C^{0,\frac{1}{3}}$ pre-shock (with respect to the angular variable) in finite time, from generic and smooth initial data. A more refined self-similar approach was then used in~\cite{BuShVi2023a,BuShVi2023b} to prove stable gradient singularity formation for the full three-dimensional Euler system, in the absence of any symmetry assumptions. The papers~\cite{BuShVi2023a,BuShVi2023b} establish the formation of a {\it{stable point-shock}} from smooth and generic initial data. Point-shock solutions are multi-dimensional analogues of the one-dimensional pre-shock
solutions; they occur about a distinguished point in space-time and quantify how the solution forms a $C^{0,\frac{1}{3}}$ cusp singularity in one direction, while remaining smooth in the orthogonal directions. We emphasize that while proofs of shock formation based on self-similar analysis provide detailed 
information about the solution at the point of the very first singularity (the point-shock mentioned earlier), these methods are ill-suited for resolving the physical shock-development problem.
 
\item In~\cite{NeShVi2023} we have revisited the results of~\cite{BuShVi2022} from a   {\it{characteristics-based}} perspective, adapted to the three distinct wave-families present in the Euler equations. For the two-dimensional Euler equations in azimuthal symmetry, we proved the stable formation of $C^{0,\frac{1}{3}}$ cusps in finite time, from smooth and generic initial conditions, in a more general setting than~\cite{BuShVi2022}. A  three-wave-family characteristics-based approach, was also used in~\cite{BuDrShVi2022} to resolve the shock development problem for the two-dimensional Euler equations in azimuthal symmetry. By utilizing the  $C^{0,\frac{1}{3}}$ pre-shock as new Cauchy data for the Euler evolution, the paper~\cite{BuDrShVi2022} shows that simultaneously to the discontinuous shock, a weak-rarefaction and a weak-contact singularity emerge from the pre-shock. 

\item Recently, pre-shock formation {\it{past the  time of the first singularity}} has been established in the context of the maximal globally hyperbolic development (MGHD) of smooth and generic Cauchy data, for the multi-dimensional Euler equations. In~\cite{ShVi2024}, the last two authors of this paper have constructed a fundamental piece of the boundary of the MGHD, which is necessary for the local shock development
problem. Specifically,~\cite{ShVi2024} used a precise geometric analysis of fast-acoustic characteristic surfaces and a new type of differentiated Riemann variables to smoothly evolve the Euler solution up to the future temporal boundary of the MGHD, which is a  singular hypersurface in spacetime containing the union of three sets: (i) a codimension-$2$ set of first singularities all of which are $C^{0,\frac{1}{3}}$ pre-shocks; (ii) a codimension-$1$ hypersurface emanating from the pre-shock set in the downstream direction, on which the Euler solution experiences a continuum of gradient catastrophes; and (iii) a codimension-$1$ Cauchy horizon emanating upstream from the pre-shock set, which the Euler solution cannot reach. We also mention the preprint~\cite{AbSp2022} which evolves the Euler solution up to a portion of the boundary of the MGHD, containing only the sets in (i) and (ii) above, but not the Cauchy horizon. We note that the proof framework employed in this manuscript closely resembles the one in~\cite{ShVi2024}, except that, in this case, we are not confronted with the geometric complexities inherent in multiple space dimensions.

\end{itemize}

To reiterate, every constructive singularity formation result cited in the list above demonstrates the formation of a stable $C^{0,\frac{1}{3}}$ cusp as the initial gradient singularity of a smooth Euler evolution. So far, only two the papers~\cite{BuIy2022} and~\cite{NeRiShVi2024} have established the formation of gradient singularity which is {\it{not}} a $C^{0,\frac{1}{3}}$ cusp:
\begin{itemize}[leftmargin=*] 
\item  For the compressible Euler equations in one space dimension, in~\cite{NeRiShVi2024} we have exhibited an open set of initial data, such that the classical solution of~\eqref{eq:euler:2}  forms a $C^{0,\alpha}$ pre-shock as a first singularity, for any $\alpha \in [\sfrac 12, 1)$. Note that these pre-shock singularities correspond to H\"older exponents {\it{strictly larger}} than $\sfrac 13$. Moreover, their formation is stable with respect to small perturbations of the initial data in the topology of $C^{1, \frac{1-\alpha}{\alpha}-\delta}$, for any small $\delta>0$. We emphasize however that these exotic cusps (with H\"older exponent larger that $\sfrac 13$) do not arise from smooth initial conditions; in fact, such initial data cannot lie in $C^{2,0}$.

\item For smooth initial conditions, the only paper which constructed a non-$C^{0,\frac{1}{3}}$ pre-shock for the Euler equations is~\cite{BuIy2022}. The authors of~\cite{BuIy2022} have used modulated self-similar stability analysis of an explicit blowup profile to construct solutions of the two-dimensional isentropic Euler equations in azimuthal symmetry, which form a $C^{0,\frac{1}{5}}$ cusp as a first singularity, and have proven that such a solution is codimension-$2$ stable with respect to $C^8$ perturbations. The paper~\cite{BuIy2022} builds on the self-similar analysis developed in~\cite{BuShVi2022}, and uses information about the explicit smooth unstable blowup profiles for the 1D Burgers equation (see also~\cite{CoGhMa2022} and the discussion in \S~\ref{sec:burgers} below).\footnote{We mention here also the paper~\cite{OhPa2024}, which does not treat the Euler system, but instead considers dispersive and dissipative perturbations of the Burgers equation such as fractional KdV, fractal Burgers, and the Whitham equation. Theorem 1.1 in~\cite{OhPa2024} shows that for a codimension $2n-2$ subset of $H^{2n+3}(\TT)$, the aforementioned models form a $C^{0,\frac{1}{2n+1}}$ cusp as the first gradient singularity. As with~\cite{BuIy2022}, the proof in~\cite{OhPa2024} is based on modulation theory, where the well-known smooth self-similar solutions to the inviscid Burgers equation (see~\cite{CoGhMa2022}) are used as ``profiles'' in weighted $L^2$-based stability estimates. The proof approach taken in our paper avoids self-similar analysis, or knowledge of any explicit blowup profiles.} The result in~\cite{BuIy2022} corresponds to the {\it{special case $n=2$ and $S_0 \equiv 0$ in our main result}}, Theorem~\ref{thm:HR:a} below. We emphasize that the proof presented in this manuscript is fundamentally different than the one in~\cite{BuIy2022}: instead of modulated self-similar stability analysis, we appeal to the analysis of differentiated Riemann-type variables in the coordinates of the fast-acoustic characteristics. We believe that this perspective is more robust (as seen also in~\cite{NeRiShVi2024,ShVi2024}), and generalizes in a straightforward way to general classes of hyperbolic systems of conservation laws.
\end{itemize}

\subsection{Main results}
\label{sec:main}

Our main result concerns the construction and the stability analysis of an infinite hierarchy of pre-shock solutions which arise as the first gradient singularity for velocity and sound speed in smooth one-dimensional Euler dynamics. The $n^{\rm th}$ pre-shock in this hierarchy ($n\geq 1$ is an integer) corresponds 
to a $C^{0,\frac{1}{2n+1}}$ cusp-singularity, and forms from initial data lying on a $(2n-2)$-dimensional Banach submanifold of $W^{2n+2, \infty}$. A precise statement of our main result is given in Theorem~\ref{thm:HR} below, while an abbreviated statement is given next in Theorem~\ref{thm:HR:a}.

Fix $n \geq 1$ an integer. Let us introduce a function $\bar w_0 \colon \TT \to \RR$ which has $2n+2$ bounded derivatives, whose derivative $\bar{w}_0^\prime$ attains a global minimum  at a distinguished point in $\TT$, and such that at this global minimum the function $\bar{w}_0^\prime$ exhibits flatness of order $2n$,  meaning that the functions $\{\p_y^j \bar{w}_0^\prime \}_{j=1}^{2n-1}$ all vanish here, but $\p_y^{2n} \bar{w}_0^\prime$ does not. Using the translation invariance of the Euler equations we may assume w.l.o.g. that $\bar{w}_0^\prime$ attains its global minimum at $y=0$. Moreover, using the hyperbolic scaling invariance of the Euler equations (see, for example, Section 1.3 in \cite{BuShVi2023a}), we may assume w.l.o.g. that $\bar{w}_0^\prime(0) = -1$. The prototypical such function is given as
\begin{equation*}
\bar{w}_0^\prime(y)  = -1+ y^{2n}\,,
\qquad \mbox{for } y \mbox{ close to } 0\,,
\end{equation*} 
suitably smooth for $y$ away from $0$, $\TT$-periodic, and with zero mean on $\TT$; this last condition ensures that the function $w_0(y) = w_0(0) + \int_0^y \bar{w}_0^\prime(\bar y) d\bar y$ is $\TT$-periodic. The constant $w_0(0)$ is chosen such that $w_0 >0$ uniformly on $\TT$, which is a  ``no-vacuum'' assumption on the initial data.
A precise definition of all permissible functions $\bar w_0$ satisfying the above properties is given in~\eqref{def:w_0:bar} below. We henceforth fix such a function $\bar w_0$.

\begin{theorem}[\bf Finite codimension-stable shock formation, abbreviated]
\label{thm:HR:a}
Let $n \geq 1$ be an integer, and fix the adiabatic exponent $\gamma > 1$. 
Then there exists a codimension $(2n-2)$ Banach submanifold  $\mathcal{M}_n$ of $(W^{2n+2, \infty}(\TT))^3$ which contains the point $(\bar u_0, \bar \sigma_0,\bar S_0):= \frac 12 (\bar{w}_0, \bar{w}_0, 0)$, and an $\eps_0 = \eps_0(n,\gamma)>0$ such that for all $0<\eps \leq \eps_0$ the following holds. For any initial data   $(u_0, \sigma_0, S_0) \in \mathcal{M}_n \cap B_\eps(\bar u_0, \bar \sigma_0,\bar S_0)$ (the ball of radius $\eps$ is taken with respect to the topology of $W^{2n+2, \infty}(\TT)$), the unique classical solution $(u,\sigma, S)$ of the one-dimensional Euler equations \eqref{eq:euler:2} with initial data $(u,\sigma, S) \vert_{t=0} = (u_0,\sigma_0, S_0)$ forms a first gradient singularity for $u$ and $\sigma$ at time $T_* =  \frac{2}{1+\alpha} + \OO_\gamma(\eps)$.\footnote{Here and throughout the paper, the notation $A \les B$ is used to say that there exists a constant $C > 0$, which is independent of $\gamma, n, \eps$, such that $A \leq C B$. We write $A = \OO(B)$ when $|A| \les B$. If the implicit constant $C$ in the $\les$ symbol also depends on a parameter $F$, we write $A \les_F B$, and respectively $A = \OO_F(B)$. See \S~\ref{sec:notation} for details.} The function $S$ remains $C^{1,\frac{1}{2n+1}}$ smooth uniformly up to time $T_*$. Furthermore, at time $T_*$ there exists a distinguished blowup point $y_* \in \TT$, such that away from $y=y_*$, the variables $(u,\sigma, S)(\cdot,T_*)$ remain $W^{2n+2, \infty}$ smooth. Finally, the gradient singularity experienced by the velocity and sound speed correspond to a {\it{$C^{0,\frac{1}{2n+1}}$ pre-shock}}, and the corresponding cusp profile is given by
\begin{subequations}
\label{eq:power:series:u:sigma:S}
\begin{align}
u(y,T_*) & = u(y_*,T_*) - \mathsf{b}\,  (y-y_*)^{\frac{1}{2n+1}} + \OO_{\gamma,n}( |y-y_*|^{\frac{2}{2n+1}}), 
\label{eq:power:series:u:sigma:S:a}\\
\sigma(y,T_*) & = \sigma(y_*,T_*)- \mathsf{b}\, (y-y_*)^{\frac{1}{2n+1}} + \OO_{\gamma,n}( |y-y_*|^{\frac{2}{2n+1}}), \label{eq:power:series:u:sigma:S:b}\\
S(y,T_*) & = S(y_*,T_*) + \p_y S(y_*,T_*) (y-y_*) +  \OO_{\gamma,n}(\eps |y-y_*|^{1+ \frac{1}{2n+1}}), 
\label{eq:power:series:u:sigma:S:c}
\end{align}
\end{subequations}
for all $y$ such that $|y-y_*|^{\frac{1}{2n+1}} \les 1$  and a constant $\mathsf{b}$ with the asymptotic $\mathsf{b}  = (2n+1)^{\frac{1}{2n+1}} + \OO_{\gamma,n}(\eps)$.
\end{theorem}

The precise statement of our main result is given in Theorem~\ref{thm:HR} below, where both the assumptions and the conclusions are stated in terms of the Riemann variables $(w,z,k):= (u+\sigma,u-\sigma,S)$ defined in~\eqref{riemann}. At this stage, we make a few comments concerning the statement of Theorem~\ref{thm:HR:a}: 
\begin{enumerate}[leftmargin=*] 

\item In the case $n=1$,  $\mathcal M_1$ is a codimension-$0$ submanifold  of $(W^{4, \infty}(\TT))^3$ containing the point $(\bar u_0, \bar \sigma_0,\bar S_0)$. A different way to state the {\it{codimension-$0$}} property of $\mathcal M_1$ is to say that it is an {\it{open subset}} of $(W^{4, \infty}(\TT))^3$ containing $(\bar u_0, \bar \sigma_0,\bar S_0)$. In this case the cusp structure described by~\eqref{eq:power:series:u:sigma:S} is that of a stable $C^{0,\frac 13}$ pre-shock, which recovers the known results (see e.g.~\cite{BuShVi2022,BuDrShVi2022,NeShVi2023}). For $n\geq 2$, $\mathcal M_n$ is the graph of a Lipschitz function from a codimension-$(2n-2)$ subspace to $\R^{2n-2}$.

\item In the case $n=2$,  $\mathcal M_2$ is a codimension-$2$ submanifold of $(W^{6, \infty}(\TT))^3$ containing the point $(\bar u_0, \bar \sigma_0,\bar S_0)$, and the expansion~\eqref{eq:power:series:u:sigma:S} describes the cusp structure of a $C^{0,\frac 15}$ pre-shock. This recovers (with a different proof) the result in~\cite{BuIy2022}, under milder regularity and smallness assumptions on the initial data (e.g.~$W^{6,\infty}$ instead of $C^8$). 

\item The intuition behind the cusp-expansion~\eqref{eq:power:series:u:sigma:S} stems from an exact computation performed on the 1D Burgers equation (presented in \S~\ref{sec:burgers} below) and the transport-structure of the evolution equation satisfied by the Riemann variables (see \S~\ref{sec:DRV} below). At this stage we just note that if the initial data of the one-dimensional Euler equation is precisely the smooth function $(\bar u_0, \bar \sigma_0,\bar S_0) := \frac 12 (\bar{w}_0, \bar{w}_0, 0)$, which corresponds to the limiting case $\eps\to0$ at fixed $n\geq 1$, then for all $t\in (0,T_*)$ the solution $(u,\sigma,S)(\cdot,t)$ of~\eqref{eq:euler:2} is given by $ \frac 12 (\bar w, \bar w, 0)(\cdot,t)$, where $\bar w$ solves the 1D Burgers equation $\partial_t \bar w + \frac{1+\alpha}{2} \bar w \p_y \bar w = 0$, with initial data $\bar{w}_0$. Now since $\bar{w}_0^\prime(y) = -1 + y^{2n}$ for $y$ close to $0$,  and $\bar{w}_0^\prime(y) \geq -1 + C$ for  $y$ far from $0$ (for some $C>0$), the 1D Burgers equation develops a gradient singularity  at exactly time $T_* = \frac{2}{1+\alpha}$,  at exactly the location $y_* = \frac{2}{1+\alpha} \bar{w}_0(0)$, and the Puiseux expansion $\bar w(y,T_*)  = \bar w_0(0) - 2 (2n+1)^{\frac{1}{2n+1}}(y-y_*)^{\frac{1}{2n+1}} + \OO_n(|y-y_*|^{\frac{1}{2n+1}})$ holds for all $y$ sufficiently close to $y_*$. For 1D Burgers these facts  were established earlier in~\cite[Proposition 9]{CoGhMa2022} using self-similar techniques.

\item If further regularity assumptions are placed on the initial conditions, such as $W^{2n+2+L,\infty}$ for some $L\geq 1$, then the $\OO_{\gamma,n}( |y-y_*|^{\frac{2}{2n+1}})$ term present in~\eqref{eq:power:series:u:sigma:S:a}--\eqref{eq:power:series:u:sigma:S:b} and the $\OO_{\gamma,n}(\eps |y-y_*|^{1+ \frac{1}{2n+1}})$ term present in~\eqref{eq:power:series:u:sigma:S:c} may be further expanded as 
$(y-y_*)^{\frac{2}{2n+1}}$, respectively $(y-y_*)^{1+\frac{1}{2n+1}}$ multiplied by an 
$(L-1)^{\rm th}$ order Taylor polynomial in the variable $(y-y_*)^{\frac{1}{2n+1}}$ (a truncated Puiseux series). In fact, our proof (see~\S~\ref{sec:thm} and Appendix~\ref{sec:appendix:poly}) gives an algorithm for computing the coefficients of this truncated Puiseux series, resulting in explicit expressions of these coefficients similar to that for $\mathsf{b}$, up to an $\OO_{\gamma,n}(\eps)$ correction. To sum up, if  the initial data is assumed to have  $L$ more bounded derivatives, then~\eqref{eq:power:series:u:sigma:S}  may be expanded to $L$ orders higher, with controlled coefficients and error bounds. For conciseness,  in this paper we choose not to pursue these higher order expansions.

\item \label{item:v} We have chosen to pay great attention to how the parameter $n$ enters the description about the Euler solution on the pre-shock; this is the main reason for the extended length of the paper. For instance, the estimate for the blowup time $T_* = \sfrac{2}{(1+\alpha)} + \OO_\gamma(\eps)$ is independent of $n$. A much more subtle and delicate issue is the description  of the parameter $\mathsf{b}$ appearing in \eqref{eq:power:series:u:sigma:S} --- to leading order in $\eps$ it is given by an explicit formula, $(2n+1)^{\frac{1}{2n+1}}$. Equally subtle is the description of the range of $y$ (close to $y_*$) for which \eqref{eq:power:series:u:sigma:S} holds --- the implicit constant in $|y-y_*|^{\frac{1}{2n+1}} \les 1$ is independent of $n$ (the precise bound is $\leq \sfrac{1}{C_0}$, where $C_0$ depends only on the choice of $\bar w_0$). These $n$-independent bounds are important for Remark~\ref{rem:Riemann} below.

\end{enumerate}

\begin{remark}[\bf The Riemann data and the infinitely unstable limit $n\to \infty$]
\label{rem:Riemann} 
Consider the limit as $n\to \infty$ of the function $\mathsf{E}_n(y):= \mathsf{a} - \mathsf{b}_n (y-y_*)^{\frac{1}{2n+1}}$, where $\mathsf{b}_n := (2n+1)^{\frac{1}{2n+1}}$ and $\mathsf{a} >0$ are constants. In view of item~\ref{item:v} in the above remarks, the function $\mathsf{E}_n$ precisely matches the cusp expressions~\eqref{eq:power:series:u:sigma:S:a} and~\eqref{eq:power:series:u:sigma:S:b} appearing in Theorem~\ref{thm:HR:a}. While the pointwise limit may be easily computed as $\lim_{n\to \infty}\mathsf{E}_n(y) := \mathsf{a} - \sgn( y- y_*)$, which we recognize as the standard {\it{Riemann data}} of 1D gas dynamics (with left state $\mathsf{a}+1$ and right state $\mathsf{a}-1$), the constraint $|y-y_*|^{\frac{1}{2n+1}} \leq \sfrac{1}{C_0}$ present in \eqref{eq:power:series:u:sigma:S} (see also item~\ref{item:v} above) means that the range of validity of this limit is trivial ($y=y_*$). We thus highlight the fact that while formally the Puiseux series expansion at the pre-shock converges to the Riemann data in the infinitely unstable limit $n\to \infty$, drawing a direct connection between the solutions of 1D Euler constructed in Theorem~\ref{thm:HR:a} (as $n\to \infty$) and the solution of the classical Riemann problem for 1D Euler, remains to date elusive.
\end{remark}

\subsection{Outline and new ideas}
\label{sec:outline}

As we have already noted, the proof of Theorem~\ref{thm:HR:a} relies on the analysis of differentiated Riemann variables in the coordinates of the fast-acoustic characteristics. When coupled to energy estimates and geometric considerations, this perspective has proven itself to be fundamental in analyzing shock formation for  multi-dimensional Euler in~\cite{ShVi2024}. 
The proof of Theorem~\ref{thm:HR:a} proceeds as follows:
\begin{itemize}[leftmargin=*] 
\item First, we change variables and study the 1D Euler system in {\it{differentiated Riemann variables}} (see \S~\ref{sec:DRV}) which allows us to include non-trivial entropy evolution without any derivative loss. Then, we change coordinates by analyzing the system when the functions are composed with the {\it{fast acoustic characteristics}} $\eta$ of the system (see \S~\ref{sec:system}). In these new {\it{Lagrangian-type coordinates}}, the dominant Riemann variable is effectively ``frozen'', analogous to how the solution of Burgers equation is frozen along its corresponding characteristics (see \S~\ref{sec:burgers}). By making this change to Lagrangian-type coordinates, we are able to reduce the task of controlling the derivatives of the solution to that of controlling the derivatives of the subdominant Riemann variable, and the entropy (see e.g.~\S~\ref{sec:apriori} and \S~\ref{sec:ineq:dx}).

\item Second, we show in \S~\ref{chp-estimates} that for a broad class of initial data the solution forms a singularity in finite time, and that this singularity is characterized by the flow $\eta$ of the fast acoustic characteristics becoming degenerate; that is $\eta_x \to 0$ at a distinguished point in spacetime.

\item Third, in \S~\ref{sec:HOE} we prove that {\it{in Lagrangian coordinates the solution stays as smooth as the initial data}}, uniformly up to the time of the first singularity. To do this, we prove novel $L^q$ energy estimates for the time derivatives of the system which are uniform in $1<q<\infty$, and then send $q \rightarrow \infty$ to obtain $L^\infty$ bounds. After we have obtained $L^\infty$ bounds on the time derivatives, the fact that the wave speeds are uniformly transverse to one another (since we are bounded away from vacuum) allows us extract bounds on mixed space-time derivatives from the time derivatives.

\item Next, in \S~\ref{sec:stab} we establish stability estimates for the Euler solution  with respect to perturbations of the initial data. For instance, if the initial data for the dominant Riemann variable is given by $w_0 = \bar{w}_0 + \lambda \tilde{w}_0$, for some parameter $\lambda \in \mathbb{R}$ and with $\bar{w}_0$, $\tilde{w}_0$ fixed, we compute $\p_\lambda$ of the resulting one-parameter family of solutions and establish uniform-in-$\lambda$ bounds. In \S~\ref{sec:cont} we bound the difference of two solutions in terms of the size of the difference of their initial data in $(W^{2n+2, \infty}(\TT))^3$.  
\item After this, in \S~\ref{FCBM} we use the implicit function theorem to show that for initial data in a given open subset of $(W^{2n+2, \infty}(\TT))^3$ there exists a solution $(x_*, T_*)$ of the system
\begin{equation*}
\eta_x(x_*, T_*) = 0 \,, \qquad
\p_x \eta_{x}(x_*,T_*) = 0 \,, \qquad \ldots \qquad 
\p^{2n-1}_x \eta_x(x_*, T_*) = 0\,,
\end{equation*}
if and only if this data lies on a codimension-$(2n-2)$ Banach submanifold of $(W^{2n+2, \infty}(\TT))^3$. Because this system has $2n$ equations and only $2$ unknowns, this is precisely what one would expect. For all such data where a solution $(x_*, T_*)$ exists, the solution is then shown to be unique.
\item Finally, in \S~\ref{sec:thm} we  invert the fast-acoustic flow map $\eta$ at the site of the pre-shock using a {\it{Puiseux series}},  to obtain a description of the solution in the original Eulerian coordinates. Because in coordinates adapted to the fast acoustic characteristic,  the solution remains as smooth as the initial data up to the  time of the first gradient blowup, the regularity of the solution $(w,z,k)$ in Eulerian variables is determined by the regularity of the inverse flow map $\eta^{-1}$.
\end{itemize}

\subsection{Notation}
\label{sec:notation}
Throughout the paper we shall appeal to the following notational conventions: 
\begin{itemize}[leftmargin=*] 
\item The notation $A \les B$ is used to say that there exists a constant $C > 0$,  such that $A \leq C B$. Here, the implicit constant $C$ is assumed to be {\it{independent}} of the parameters $\alpha, n, \eps$, {\it{independent}} of $(x, t) \in \TT \times \R$, {\it{independent}} of the constant $C_0 \geq 3$ (appearing in~\eqref{data:all}), and {\it{independent}} of our choice of $(w_0,z_0, k_0)  \in \mathcal A_n(\eps, C_0)$ (defined in~\eqref{data:all}).
\item  The notation $A \les_F B$ is used to indicate that there exists a constant $C = C(F)> 0$, such that $A \leq C B$; that is, the implicit constant depends on the parameter $F$.

\item We write $A = \OO(B)$ when $|A| \les B$, and $A = \OO_F(B)$ when $|A| \les_F B$.

\item We write ``If $A \ll B$ then $X$'' to mean that there exists a constant $C \geq 1$, independent of $\alpha,  n,  \eps, x,t, C_0$ and independent of our choice of $(w_0,z_0,k_0) \in\mathcal A_n(\eps, C_0)$, such that ``if $C \cdot A \leq B$, then $X$ is true''.

\item If $f$ is a function defined on $\TT \times \R$, and $\eta$ is a map defined on $\TT \times \R$, we will write ``$f\circ \eta$'' to denote the function
$
f\cir \eta (x,t) : = f(\eta(x,t),t).
$

\item We use the notation $\mathbbm{1}_A$ to denote the characteristic function of a set $A$.
\item We write $\p_y f, \p_x f, \p_t f$ to denote the partial derivative of a function $f$ with respect to $y,x,$ and $t$ respectively. If $f$ is a function of $(x,t)$, and $\beta = (\beta_x, \beta_t)$ is a multi-index, we write $\p^\beta f$ to denote $\p^{\beta_x}_x \p^{\beta_t}_t f$. There are a few exceptions where we will use subscript notation to denote derivatives: $\eta_x : = \p_x \eta, \eta_{xt} : = \p_x \p_t \eta, \Sigma_t : = \p_t \Sigma, \Sigma_x : = \p_x\Sigma, (\eta_x \W)_t : = \p_t (\eta_x \W), (\Sigma^{-1})_t : = \p_t (\Sigma^{-1}), (\Sigma^{-1})_x : = \p_x(\Sigma^{-1})$. In \S~\ref{sec:stab} we will also use subscript notation to denote a derivative with respect to a parameter $\lambda$ introduced in \S~\ref{sec:stab}, except when $x$ or $t$ is already being used as a subscript; so, for example, we will write $\p^\beta \K_\lambda$ to mean $\p^{\beta_x}_x \p^{\beta_t}_t \p_\lambda \K$ but we will write $\p_\lambda \eta_x$ to denote $\p_x \p_\lambda \eta$.
\item We write $w'_0, z'_0,$ and $k'_0$ to denote the first derivatives of $w_0, z_0,$ and $k_0$ respectively.
\end{itemize}


\section{Intuition from the Burgers equation}
\label{sec:burgers}
This section aims to revisit some well-known aspects of the 1D Burgers equation from a specific perspective that will inform our proof for the 1D Euler equation. Though the examples presented are not novel, they offer valuable intuition;  moreover, the 1D Burgers equation is embedded in the 1D Euler system. Indeed, upon inspecting~\eqref{eq:euler:2}, we see that if $S|_{t=0} \equiv 0$, then $S(\cdot,t) \equiv 0$ for all times $t$ prior to the first gradient singularity. In this case, the remaining equations~\eqref{eq:euler:2:a}--\eqref{eq:euler:2:b} have a particular solution: if $u|_{t=0} = \sigma|_{t=0}$ and $S|_{t=0} \equiv 0$, then the relation $u(\cdot,t) \equiv \sigma(\cdot,t)$ holds for all times $t$ prior to the first gradient singularity. Denoting $u+\sigma = 2 u = 2\sigma$ by $w$, we see that in this special case, \eqref{eq:euler:RV} reduces to:
\begin{equation}\label{eq:euler:burgers}
\p_tw + \tfrac{1+\alpha}{2}w\p_yw = 0,
\end{equation}
Rescaling time with $t \rightarrow \tfrac{1+\alpha}{2} t$, this becomes the classical {\it{inviscid Burgers equation}}
\begin{equation}
\label{eq:burgers}
\p_tw + w\p_yw = 0,
\end{equation}
the archetypal equation for shock formation and development.

Classically,~\eqref{eq:burgers} is solved by the method of characteristics. If $w$ is a solution of \eqref{eq:burgers} on $\R \times [0,T]$ and we define the {\it{characteristics}} $\eta$ to be the flow of of the scalar field $w$, namely as the solution of
$$
\tfrac{d}{dt} \eta(x,t) = w(\eta(x,t),t),
\qquad
\eta(x,0) = x,
$$
then $w$ must remain constant along the flow lines of $\eta$. From this it follows that if $w_0(x):= w(x,0)$ then
\begin{equation}\label{eq:eta_x:burgers}
\eta(x,t) = x+ tw_0(x).
\end{equation} 
Thus, characteristics are straight lines in spacetime, and the line originating at the spacetime point $(x,0)$ has slope $w_0(x)$. Conversely, if we are given initial data $w_0$ and we define $\eta$ via \eqref{eq:eta_x:burgers}, then one can construct a unique solution $w$ via the implicit equation
\begin{equation*}
w(\eta(x,t), t) = w_0(x)
\end{equation*}
at all points $(y,t)$ such that the equation
\begin{equation}
\label{eq:y:burgers}
(y,t) = (\eta(x,t), t)
\end{equation}
has a unique solution $x \in \R$. 
For initial data $w_0$ that is not monotone increasing in $x$, it is clear that \eqref{eq:y:burgers} will not be uniquely solvable for all $(y,t) \in \R \times [0, \infty)$, resulting in a gradient singularity at some finite time $T_* \in (0,\infty)$. In particular, when $w_0 \in C^1_\loc(\R)$,  the maximal time of existence and uniqueness of the $C^1$ solution $w$ of~\eqref{eq:burgers} is
\begin{equation*}
T_*  : = 
\begin{cases} -\frac{1}{\inf w'_0} & \inf w'_0 < 0\,,
\\ 
+\infty & \inf w'_0 \geq 0 \,.
\end{cases}
\end{equation*}
Note that for all prior times $0 \leq t < T_*$ we have
\begin{equation*}
\eta_x(x,t) = 1 + tw'_0(x) > 0,
\end{equation*}
and hence $\eta(\cdot, t) : \R \rightarrow \R$ is a $C^1$ diffeomorphism, but at the point $(x_*, T_*)$ where $\eta_x(x_*, t) \rightarrow 0$ as $t \rightarrow T_*$, we have $\p_y w(\eta(x_*, t), t) = w_0^\prime(x_*) (\eta_x(x_*,t))^{-1} \rightarrow -\infty$. Therefore, the first singularity in a $C^1$ solution of Burgers equation is characterized by $\eta$ ceasing to be a $C^1$ diffeomorphism, i.e.~the function $\eta_x$ attaining a zero.

Now suppose that $w_0$ is smooth, $w'_0$ attains its global minimum at a point $x_*$, and $w'_0(x_*) < 0$.  Our objective is to give a precise description
of the solution $w$ near the pre-shock $(y_*, T_*) = (\eta(x_*,T_*), T_*)$. Generically, $x_*$ is  a {\it{non-degenerate}} local minimum of $w_0'$, i.e.~the second derivative test of calculus applies: 
\begin{equation*}
\p_x w'_0(x_*) = 0, \qquad \mbox{and} \qquad 
\p^2_x w'_0(x_*) > 0 \,.
\end{equation*}
In this case, from~\eqref{eq:eta_x:burgers} we deduce 
\begin{equation*}
\eta_x(x_*, T_*)= 0,
\qquad
\p_x \eta_{x}(x_*, T_*) = T_* \p_xw'_0(x_*) = 0, 
\qquad \mbox{and} \qquad
\p^2_x \eta_{x}(x_*,T_*) = T_* \p^2_xw'_0(x_*) =: 3! a_3 > 0.
\end{equation*}
Therefore, assuming that $w_0 \in C^4$, Taylor's theorem yields
\begin{equation*}
\eta(x, T_*) = \eta(x_*, T_*) +  a_3(x-x_*)^3[1 + \OO(|x-x_*|)] \,.
\end{equation*}
If $y = \eta(x,T_*)$ and $y_* = \eta(x_*,T_*)$, then for $y$ near $y_*$ we have
\begin{equation*}
x-x_* = \big(\tfrac{y-y_*}{a_3} \big)^{1/3} + \OO(|y-y_*|^{2/3}),
\end{equation*}
and thus
\begin{align}
w(y,T_*) = w(\eta(x,T_*), T_*)
= w_0(x)
& = w_0(x_*) + \tfrac{w'_0(x_*)}{1!}(x-x_*) +  \OO((x-x_*)^{2}) \notag \\
& = w(y_*, T_*) + w'_0(x_*) \big(\tfrac{y-y_*}{a_3} \big)^{1/3} + \OO(|y-y_*|^{2/3}).
\label{eq:puke}
\end{align}
Thus, the Taylor series expansion of the smooth function $(w\circ \eta)(\cdot, T_*) = w_0(\cdot)$ with respect to $x-x_*$, becomes a {\it{fractional series expansion}} of the H\"{o}lder continuous function $w(\cdot, T_*)$ with respect to $(y-y_*)^{1/3}$. Since by assumption we have that $w_0^\prime(x_*)<0$, the expansion~\eqref{eq:puke} shows that 
$w(\cdot, T_*)$ forms a $C^{0,\frac 13}$-cusp at $(y_*, T_*)$, resembling
$
- y^{1/3}.
$
The above described behavior is moreover {\it{stable}} under small $W^{4,\infty}$ perturbations to $w_0$, because all functions $w_0$ close enough to our original function in $W^{4,\infty}$ are going to be such that $w'_0$ attains its minimum at a point nearby $x_*$, the minimum will still be negative, and it will still be a non-degenerate critical point of $w'_0$.

Consider next the case when $w_0^\prime$ attains its global minimum at a {\it{degenerate}} critical point $x_*$, and the derivatives of $w_0^\prime$ do not vanish to infinite order at $x_*$. Then, necessarily there would exist some integer $n \geq 2$ such that
\begin{equation*}
\p^i_x w'_0(x_*) = 0  \; \mbox{ for all } \; i \in \{ 1, \ldots, 2n-1\}, 
\qquad
\mbox{and}
\qquad
\p^{2n}_x w'_0(x_*) > 0 \,.
\end{equation*}
The integer $n$ is sometimes referred to as the {\it{order of flatness}} of $w_0^\prime$ at the critical point $x_*$.
In this case, the same arguments as before would imply that
\begin{equation*}
x-x_* =  
\big(\tfrac{y-y_*}{a_{2n+1}} \big)^{\frac{1}{2n+1}} + \OO(|y-y_*|^{\frac{2}{2n+1}})\,,
\end{equation*}
where
\begin{equation*}
1 + T_* w_0^\prime(x_*)= 0\,,
\quad
T_* \p_x w_0^\prime(x_*) = 0\,,
\quad
\ldots
\quad
T_* \p_x^{2n-1} w_0^\prime(x_*) = 0\,,
\quad
T_* \p_x^{2n} w_0^\prime(x_*) =: (2n+1)! a_{2n+1} > 0 
\,.
\end{equation*}
The result is that $w(\cdot, T_*)$ forms a $C^{0,\frac{1}{2n+1}}$-cusp at $(y_*,T_*)$,  resembling
$
- y^{\frac{1}{2n+1}}.
$
Such a behavior, however, is {\it{unstable}} under smooth perturbations. 

For example, the  prototypical function which attains a global minimum at a degenerate critical point $x_*$, and whose derivative  exhibits flatness of order $2n$ at $x_*$, is given by 
\begin{equation}
\label{data:keyexample}
\bar{w}_0(x) = - x + \tfrac{1}{2n+1} x^{2n+1} \,,
\end{equation} 
for which with $x_*=0$. When $n\geq 2$, the  $C^\infty$-smooth and locally small (take $0 < \eps \leq 1$ small) perturbation $w_0(x)
:= \bar{w}_0(x) + \frac 13 \eps  x^3$ represents a function which attains its global minimum at $x_*(\eps) := 0$, but now this global minimum is non-degenerate since $\p_x^2 w_0'(x_*(\eps)) = 2\eps >0$. A more interesting example is given by the  $C^\infty$-smooth and locally small perturbation $w_0(x)
:= \bar{w}_0(x) - \frac 13  \eps  x^3$, which attains its global minimum simultaneously at two points $x_*(\eps) := \pm (\sfrac{\eps}{n})^{\frac{1}{2n-2}}$, and each of these global minima are non-degenerate since $\p_x^2 w_0'(x_*(\eps)) = 2 \eps (2n-2) >0$. 
These examples provide the simple reason behind the instability of $C^{0,\frac{1}{2n+1}}$-cusp singularities for 1D Burgers.

Next, we discuss the finite codimension stability of $C^{0,\frac{1}{2n+1}}$-cusp singularities for 1D Burgers dynamics, from a perspective which will turn out to be useful when discussing 1D Euler. Recall the function $\bar{w}_0$ defined in~\eqref{data:keyexample}, 
and let $\bar w$ be the solution of 1D Burgers~\eqref{eq:burgers} on $\R$ with initial data $\bar{w}_0$. Since $\bar{w}'_0$ has a unique global minimum at $x_*=0$, and at this point $\bar{w}_0^\prime(0) = -1$, the solution $\bar w$ develops a singularity at time $T_* =1$. At the blowup time, we have
\begin{equation*}
\eta(x,1) = \tfrac{1}{2n+1} x^{2n+1} \,,
\end{equation*}
and therefore if $y = \eta(x,1)$, so that $y_* = \eta(x_*,1) = 0$, then
\begin{equation*}
x =( (2n+1) y)^{\frac{1}{2n+1}} \,, 
\end{equation*}
resulting in the precise and global description of the $C^{0,\frac{1}{2n+1}}$-cusp at $y=0$:
\begin{equation}
\label{data:keyexample:sol}
\bar w(y,1) = \bar{w}_0(x) = -  (2n+1)^{\frac{1}{2n+1}} y^{\frac{1}{2n+1}} + y. 
\end{equation}
The above scenario is in fact stable under suitable perturbations.

\begin{proposition}[Stability and finite codimension-stability]
Let $n \geq 1$ be an integer, let  $\bar{w}_0$ be defined as in~\eqref{data:keyexample}, and define for $\eps \in (0,1]$ the ball
\begin{equation*}
U^n_\eps : = \{ w_0 : \|w_0-\bar{w}_0\|_{C^{2n+1}} < \eps \}
\,.
\end{equation*}
Then there is a codimension-$(2n-2)$ Banach manifold $\mathcal M_n$ of initial data in $U^n_\eps$ such that for all data $w_0 \in \mathcal M_n$ the unique solution $w$ has a first singularity at a unique point $(y_*, T_*) = (0,1) +  \OO_n(\eps)$ and at that point $w$ of~\eqref{eq:burgers} forms a cusp resembling~\eqref{data:keyexample:sol}.
\end{proposition}

\begin{proof}
Since $\p^{2n}_x w'_0(x) \geq (2n)! - \eps$ uniformly in $x$ for all $w_0 \in U^n_\eps$, there is a unique $\x$ in $\R$ such that
$\p^{2n-1}_x w'_0(\x) = 0$,  and $\x$ satisfies $|\x| \leq \frac{\eps}{(2n)!-\eps}$.  In the case $n =1$, $\x$ is the unique critical point of $w'_0$ and is its global minimum. In the case $n \geq 2$, define the function $f : U^n_\eps \rightarrow \R^{2n-2}$,
\begin{equation*}
f(w_0) : = \begin{bmatrix} \p_x w'_0(\x) \\ \vdots \\ \p^{2n-2}_x w'_0(\x) \end{bmatrix}
\end{equation*}
If $f(w_0) = (0,0 \hdots, 0)$, then Taylor expanding $w'_0$ about $x=\x$ gives us
$
w'_0(x) \geq w'_0(\x) + [(2n)!-\eps](x-\x)^{2n}
$
for all $x \in \R$, so $\x$ is the unique global minimum of $w'_0$. If we define $\mathcal M_1 = U^1_\eps$, and $\mathcal M_n = \{ f = 0\}$ for $n \geq 2$, then it is straightforward to show that at time $-1/ w_0^\prime(\x)$ the solution $w$ of~\eqref{eq:burgers} with initial data $w_0 \in \mathcal M_n$, develops a cusp resembling $-y^{\frac{1}{2n+1}}$.

When $n \geq 2$, we can observe that $\mathcal M_n$ is a Banach manifold using the Implicit Function Theorem. All $w_0 \in U^n_\eps$ can be decomposed as
\begin{equation*}
w_0(x) = \bar{w}_0(x) +  \big( \tfrac{\lambda_1}{2!} x^2 + \tfrac{\lambda_2}{3!} x^3 + \hdots + \tfrac{\lambda_{2n-2}}{(2n-1)!} x^{2n-1}\big) \chi (x)+ \tilde{w}_0(x)
\end{equation*}
where $\chi$ is a $C^\infty$ bump function with 
$
\mathbbm{1}_{[-1,1]} \leq \chi \leq \mathbbm{1}_{[-3,3]}
$,
the $\{\lambda_j\}_{j=1}^{2n-2}$ are sufficiently small (with respect to $\eps$ and $n$) real numbers,
and $\tilde{w}_0$ is a sufficiently small (with respect to $\eps$ and $n$) $C^{2n+1}$ function satisfying $\p^{j}_x \tilde{w}_0^\prime (0) = 0$ for $j =1, \hdots, 2n-2$. 
To see this, simply define $\lambda_j := \p_x^j w_0^\prime(0)$, and then declare $\tilde{w}_0$ to be the remainder in the above formula; since $\p_x^j \bar{w}_0^\prime(0) = 0$ for all $j =1, \hdots, 2n-2$, it is clear that $\tilde{w}_0$ will then satisfy $\p^{j}_x \tilde{w}_0^\prime (0) = 0$ for $j =1, \hdots, 2n-2$. Viewing the critical point $\x$ as a function of $\{\lambda_j\}_{j=1}^{2n-2}$ and $\tilde{w}_0$, and differentiating the expression $\p^{2n-1}_x w'_0(\x) = 0$ respect to $\lambda_j$, taking into account that both $w_0^\prime$ and $\x$ are implicitly functions of $\lambda_j$, and using that $\chi^\prime(\x) = 0$, we deduce
\begin{equation*}
\p^{2n}_x w'_0(\x) \tfrac{\p \x}{\p \lambda_j} = 0\,, \qquad \mbox{for all} \qquad j = 1, \hdots, 2n-2\,.
\end{equation*}
From the positive lower bound $\p^{2n}_x w'_0(\x) \geq (2n)! - \eps$ it follows that $\frac{\p \x}{\p \lambda_j} = 0$ for all $j = 1, \hdots, 2n-2$. Therefore, writing the $i^{\rm th}$ component of the function $f$ as $f^i = \p_x^i w_0^\prime(\x)$, and using that $\frac{\p \x}{\p \lambda_j} = 0$, we may similarly obtain
\begin{equation*}
\tfrac{\p f_i}{\p \lambda_j} 
= \p^{i+1}_x \bar{w}'_0(\x) \tfrac{\p \x}{\p \lambda_j}  + \tfrac{\p}{\p \lambda_j} \Bigl( \displaystyle{\sum}_{k=i}^{2n-2} \tfrac{\lambda_k}{(k-i)!} \x^{k-i} \Bigr) + \p^{i+1}_x \tilde{w}'_0(\x) \tfrac{\p \x}{\p \lambda_j} 
= \tfrac{\x^{j-i}}{(j-i)!}\mathbbm{1}_{j \geq i} \,,
\end{equation*} 
for all $i,j = 1, \hdots, 2n-2$. Since $|\x| \leq \frac{\eps}{(2n)!-\eps}$, it follows that
\begin{equation*}
D_\lambda f = \Id + \OO_n(\eps)
\end{equation*}
everywhere in $U^n_\eps$. Therefore, for each choice of $\tilde{w}_0$ (lying in a sufficiently small ball, of radius related to $\eps$ and $n$, in $C^{2n+1}$) there is a unique choice of $(\lambda_1, \hdots, \lambda_{2n-2})$ such that $f(\lambda_1,\hdots,\lambda_{2n-2},\tilde w_0) = 0$, i.e., such that the corresponding $w_0$ lies in $\mathcal M_n$. The Implicit Function Theorem yields that his choice of $\{\lambda_j\}_{j=1}^{2n-2}$ is given by a $C^1$ function of $\tilde{w}_0$ (in the sense of Fr\'{e}chet derivatives).
\end{proof}

\section{The differentiated Riemann variables, the initial data, and a blowup criterion}
\label{chp-preliminaries}

Returning to the 1D Euler equations~\eqref{eq:euler:2}, in this section we introduce the 
classical and the differentiated Riemann-type variables which are used in our analysis, we precisely characterize the set of initial conditions with which we work with, and we give a characterization of the time $T_*$ of the first gradient singularity.

\subsection{Classical and differentiated Riemann variables}
\label{sec:DRV}

For many first-order hyperbolic systems in one space dimension, it can be fruitful to analyze the PDE in terms of {\it{Riemann variables}}, which in a sense diagonalize the system. For the 1D Euler equations, the Riemann variables $w$ and $z$ are defined as follows:
\begin{align}\label{riemann}
w : = u + \sigma \,,
\qquad
z := u-\sigma \,,
\qquad 
k: = S.
\end{align}
Above we have also introduced the notation $k$ for the specific entropy $S$ in order to distinguish the Riemann set of variables $(w,z,k)$ from the classical ones, $(u,\sigma,S)$.
With respect to the $(w,z,k)$ variables  the 1D Euler equations \eqref{eq:euler:2} become
\begin{subequations}
\label{eq:euler:RV}
\begin{align}
\p_t w + \lambda_3 \p_y w & = \tfrac{\alpha}{2\gamma} \sigma^2 \p_y k,  \\
\p_t z + \lambda_1 \p_y z & = \tfrac{\alpha}{2\gamma} \sigma^2 \p_y k,  \\
\p_t k + \lambda_2 \p_y k & = 0,
\end{align}
\end{subequations}
where
\begin{align*}
\lambda_1 & : = u-\alpha \sigma = \tfrac{1-\alpha}{2} w + \tfrac{1+\alpha}{2} z, \\
\lambda_2 & : = u = \tfrac{1}{2} w + \tfrac{1}{2} z, \\
\lambda_3 & : = u+ \alpha \sigma = \tfrac{1+\alpha}{2} w + \tfrac{1-\alpha}{2} z.
\end{align*}
The  functions $\lambda_i$ are the {\it{wave speeds}} of the system. Most important to us will be $\lambda_3 = u+ \alpha \sigma = u + c$, which is the {\it{fast acoustic wave speed}}. The {\it{fast acoustic characteristics}} are denoted by $\eta$ and are as the solutions of the ODE via
\begin{equation}\label{def:eta}
\p_t \eta(x,t) = \lambda_3 (\eta(x,t),t), \qquad \eta(x,0 ) = 0\,.
\end{equation}
That is, $\eta$ is the flow of the fast acoustic wave speed.

Because the 1D Euler equations have  both Galilean transformations and additions of a constant entropy as symmetries, all of the essential dynamics of \eqref{eq:euler:2} are captured by studying $\sigma, \p_y u,\p_y \sigma$, and $\p_y S$. Therefore, it makes sense to study \eqref{eq:euler:RV} in terms of $\sigma, \p_y w, \p_y z,$ and $\p_y k$. If one takes $\p_y$ of the system \eqref{eq:euler:RV}, and diagonalizes the system (omitting terms which are quadratic or cubic in $\p_yw, \p_y z,$ and $\p_y k$), one arrives at the {\it{differentiated Riemann variables}} $(\mathring{w}, \mathring{z}, \mathring{k})$,  defined via
\begin{subequations}
\label{def:q^w}
\begin{align}
\mathring{w} & : = \p_y w - \tfrac{1}{2\gamma} \sigma \p_y k, \\
\mathring{z} & : = \p_y z + \tfrac{1}{2\gamma} \sigma \p_y k,\\
\mathring{k} & : = \p_y k,
\end{align}
\end{subequations}
and which satisfy the evolution equations
\begin{subequations}
\label{eq:euler:RV:diff}
\begin{align}
\p_t \mathring{w} + \lambda_3 \p_y \mathring{w} 
&= -   \mathring{w} \p_y \lambda_3    + \tfrac{\alpha}{4\gamma} \sigma \mathring{k}  ( \mathring{w} + \mathring{z}   )
= -   \mathring{w} (\tfrac{1+\alpha}{2} \mathring{w} + \tfrac{1-\alpha}{2} \mathring{z} +  \tfrac{\alpha}{2\gamma}\sigma \mathring{k} )   + \tfrac{\alpha}{4\gamma} \sigma \mathring{k}  ( \mathring{w} + \mathring{z}   )
\\
\p_t \mathring{z} + \lambda_1 \p_y \mathring{z}  
&= - \mathring{z} \p_y \lambda_1 - \tfrac{\alpha}{4\gamma} \sigma \mathring{k}  ( \mathring{w} + \mathring{z}  ) 
= - \mathring{z}  (\tfrac{1-\alpha}{2} \mathring{w} + \tfrac{1+\alpha}{2} \mathring{z} - \tfrac{\alpha}{2\gamma}\sigma \mathring{k} ) - \tfrac{\alpha}{4\gamma} \sigma \mathring{k}  ( \mathring{w} + \mathring{z}  )\\
\p_t \mathring{k} + \lambda_2 \p_y \mathring{k}  
& = - \mathring{k} \p_y \lambda_2
= - \mathring{k} (\tfrac{1}{2} \mathring{w} + \tfrac{1}{2} \mathring{z} ) \,.
\end{align} 
\end{subequations}
The differentiated Riemann variables will allow us to study the evolution of the system without derivative loss; indeed, it is evident from~\eqref{eq:euler:RV:diff} that these are three coupled transport equations with right hand sides which are polynomial expressions of the unknowns (see also~\eqref{system:lagrange} below). The ``ring super-index'' in the notation of~\eqref{def:q^w} is present to signify that differentiation has taken place; this notation is chosen for consistency with~\cite{ShVi2024}. 

For the remainder of this paper, we will study the Cauchy problem for the system  \eqref{eq:euler:RV} and~\eqref{eq:euler:RV:diff}. Next, we describe the initial data for these systems, by specifying the set $\mathcal A_n(\eps, C_0)$ where $(w_0,z_0,k_0)$ lies. The functions $(\mathring{w}_0,\mathring{z}_0,\mathring{k}_0) := (w_0^\prime - \frac{1}{2\gamma} \sigma_0 k_0^\prime, z_0^\prime + \frac{1}{2\gamma} \sigma_0 k_0^\prime,k_0^\prime)$ and $\sigma_0 = \frac 12 (w_0 + z_0)$ are then defined as a consequence.


\subsection{Assumptions on the initial data}

The domain of our functions will be $\TT \times \R$, where $\TT : = \R/\ZZ$ is the torus. We will use the notation
$|x-x'|$
to denote the distance between two points $x,x'$ on the torus. We will often implicitly identify the torus with the interval $[-\frac{1}{2}, \frac{1}{2})$ when describing functions in a neighborhood of the point $0 = 0+ \ZZ \in \TT$.

Fix a constant $C_0 \geq 3$ and an integer $n\geq 1$. Assume that the initial data $(w_0,z_0,k_0)$ of \eqref{eq:euler:RV} satisfies
\begin{subequations}
\label{data:all}
\begin{align}
\label{data:1}
\tfrac{1}{2}  < \sigma_0 = \tfrac 12 (w_0 + z_0) & < 2,  & \\
\label{data:2}
-1-\eps < \min_\TT w'_0 & < -1 + \eps, & \\
\frac{\|\p^{i+1}_x w_0\|_{L^\infty_x}}{i!} & < C_0^i(1+\eps) & i=0, \hdots, 2n+1, \\
\frac{\|\p^{i+1}_x z_0\|_{L^\infty_x}}{i!} & < C_0^i \eps & i=0, \hdots, 2n+1, \\
\label{data:9}
\frac{\|\p^{i+1}_x k_0\|_{L^\infty_x}}{i!} & < C_0^i \eps & i=0, \hdots, 2n+1.
\end{align}
\end{subequations}
The constraints listed in \eqref{data:1}--\eqref{data:9} {\it{define}} an open set 
\begin{equation}
\mathcal A_n(\eps, C_0) \subset (W^{2n+2,\infty}(\TT))^3
\,.
\label{eq:cal:A:n:def}
\end{equation} We will show that the solution $(w,z,k)$ emanating from data in this open set all form a gradient singularity at roughly the same time, and we will prove uniform derivative estimates for all solutions with data in $\mathcal A_n(\eps, C_0)$.

The definition of $\mathcal A_n(\eps, C_0)$ does not require any zeroth order constraints other than \eqref{data:1} because the Euler equations have Galilean transformations and additions of constant entropy as symmetries (though they are not invariant under constant perturbations to the mass density). The constraint \eqref{data:2} is not strictly necessary for the formation of the shock to occur, but is rather used to pin down the blowup time $T_*$ (see \S~ \ref{sec:LWP} for a definition of $T_*$); this follows from the time-rescaling symmetry of the Euler equations. In \S~\ref{sec:FCBM} we will impose additional constraints on the initial data in order to say more about how the spatial location of the pre-shock is approximately determined by the initial data; since the Euler equations are invariant under spatial translations, we will not need such specifications until then.


\subsection{Local well-posedness and the Eulerian blowup criterion}
\label{sec:LWP}

The Euler equations \eqref{eq:euler:2} are locally well-posed in $H^s(\TT)$ for $s > \frac{3}{2}$. In particular, for $u_0, \sigma_0, S_0 \in H^s(\TT) $ with $\sigma_0 > 0$ everywhere,  there exists a positive {\it{maximal time of existence}} $T_*  \in (0, +\infty]$ such that \eqref{eq:euler:2} has a unique  $C^{1}_{y,t}$-smooth solution $(u,\sigma, S)$ on $\TT \times [0,T_*)$, and furthermore 
\begin{equation*}
(u, \sigma, S) \in C([0,T_*) ; (H^s(\TT))^3) \cap C^1([0,T_*) ; (H^{s-1}(\TT))^3).
\end{equation*}
Additionally, $T_*$ can only be finite if: (a) the solution $(u, \sigma, S)$ develops vacuum or implodes in finite time, meaning that either $\min_{y\in \TT} \sigma(\cdot,t) \to 0$ as $t\to T_*$  or $\|(u,\sigma, S)(\cdot,t)\|_{L^\infty_y} \to \infty$ as $t\to T_*$; (b) the solutions' gradients blow up in finite time, at a non-integrable rate, meaning that it satisfies the {\it{Eulerian blowup criterion}}
\begin{equation}
\label{eq:LCC:1}
\int^{T_*}_0 
\Bigl(\|\p_y u(\cdot, t)\|_{L^\infty} + \|\p_y \sigma(\cdot, t)\|_{L^\infty} + \|\p_y S(\cdot, t)\|_{L^\infty}\Bigr)    dt = +\infty 
\,.
\end{equation}
These facts follow  from the classical local well-posedness theory of symmetrizable quasilinear hyperbolic systems (see e.g.~\cite[Chapter V]{Da2005}, or \cite[Chapter 2]{Ma1984}).

For the 1D Euler equations~\eqref{eq:euler:2}, it is known that neither vacuum formation nor finite-time implosion is not possible on $\TT \times [0, T_*)$ (see~\cite{ChYoZh2013,Chen2017} and Proposition~\ref{prop:0thorder}  for the initial data studied in this paper),  for any choice of $u_0, \sigma_0, S_0 \in H^s(\TT)$ with $\sigma_0 > 0$ everywhere and $s > \frac{3}{2}$; using this uniform upper and lower bound on the sound speed, one can show that $(u,\sigma, S)$ remain uniformly bounded on $\TT \times [0, T_*)$. That is, option (a) above is not available to the 1D Euler\footnote{In multiple space dimensions, this is not the case: smooth finite-time implosions may be constructed~\cite{MeRaRoSz2022a, MeRaRoSz2022b, BuCLGS2022,CLGSShSt2023,ChCiShVi2024,Ch2024}. The formation of vacuum in finite time, for smooth multi-dimensional Euler dynamics remains however to date open.} dynamics with smooth and non-vacuous initial conditions. Thus,  the Eulerian blowup criterion~\eqref{eq:LCC:1} determines whether a finite-time singularity occurs at a time $T_* < \infty$.  

The Eulerian blowup criterion~\eqref{eq:LCC:1} simultaneously serves as a continuation criterion for smooth solutions, and as a criterion for the propagation of higher regularity. Of relevance to this manuscript is the fact that if the initial data $u_0, \sigma_0, S_0 \in W^{2n+2,\infty}(\TT) \subset H^2(\TT)$, $\sigma_0$ is positive, and $T_*$ is defined via~\eqref{eq:LCC:1} as the maximal time of existence of a $C^{1}_{y,t}$-smooth solution, then for all $t \in (0,T_*)$ we have that $\|(u,\sigma,S)(\cdot,t)\|_{W^{2n+2,\infty}(\TT)} < \infty$ (naturally, the upper bound blows up as $t\to T_*$).

The Eulerian blowup criterion may also be stated in terms of Riemann variables. If $n \geq 1$, then for all initial data $w_0,z_0, k_0 \in W^{2n+2, \infty}(\TT)$ with $w_0 -z_0 > 0$ there exists a positive maximal time of existence $T_* \in (0, +\infty]$ for which \eqref{eq:euler:RV} has a unique $C^1_{y,t}$ solution $(w,z,k)$ on $\TT \times [0, T_*)$ and if $T_* < \infty$ then  
\begin{equation}
\label{eq:LCC:2}
\int^{T_*}_0 
\Bigl(\|\p_y w(\cdot, t)\|_{L^\infty} + \|\p_y z(\cdot, t)\|_{L^\infty} + \|\p_y k(\cdot, t)\|_{L^\infty}\Bigr)    dt = +\infty 
\,.
\end{equation}
For the initial data $(w_0, z_0, k_0) \in \mathcal A_n(\eps, C_0)$, we will prove that $T_*$ is finite, so we will refer to $T_*$ in this context as the {\it{blowup time}}.




\section{Initial estimates along the fast acoustic characteristic}
\label{chp-estimates}

Let $\eta$ be the fast acoustic characteristics defined in \eqref{def:eta}. We will show that for all initial data $(w_0,z_0, k_0) \in \mathcal A_n(\eps, C_0)$, the corresponding solution $(w,z,k)$ of \eqref{eq:euler:RV} stays as smooth as the initial data when precomposed with $\eta$, i.e. the functions $w\circ \eta, z\circ \eta,$ and $k\circ \eta$ stay as smooth as $w_0,z_0,$ and $k_0$ {\it{ all the way up until its maximal time of existence, $T_*$}}. For all data in $\mathcal A_n(\eps, C_0)$, the maximal time of existence will be shown to be finite, and will be approximately determined as
\begin{equation*}
T_* = \tfrac{2}{1+\alpha} + \OO_\alpha(\eps) \,.
\end{equation*}
Furthermore, the time $T_*$ will be characterized as the first time when $\eta$ {\it{ceases being a diffeomorphism}}. That is, $T_*$ will be equivalently characterized in terms of the {\it{Lagrangian blowup criterion}} 
\begin{equation*}
\lim_{t\to T_*} \Bigl(\min_{x\in \TT} \eta_x (x,t) \Bigr) = 0\,.
\end{equation*}

If $(y,t) \in \TT \times \R$ are the {\it{Eulerian variables}} which the functions $w,z,k$ take as inputs, we will define the {\it{Lagrangian variables}} $(x,t) \in \TT \times \R$ via the equation
\begin{align}
(y,t) = (\eta(x,t), t).
\end{align}
We will spend the rest of the paper studying the Euler system in terms of functions of these Lagrangian variables defined in \S~ \ref{sec:system}. In \S~ \ref{sec:initialE}, we prove that for all data in $\mathcal A_n(\eps, C_0)$ the solution exists up until a time $T_* = \frac{2}{1+\alpha} + \OO_\alpha(\eps)$ at which point $\eta_x$ must have a zero, which corresponds to $\p_y w$ diverging to $-\infty$. Then in \S~ \ref{sec:HOE} we will perform energy estimates in $L^p$ for the higher order derivatives with respect to $x$ and $t$. In \S~\ref{sec:stab} we utilize similar energy estimates to quantify the stability of solutions  with data in $\mathcal A_n(\eps, C_0)$ with respect to small perturbations in $w_0$.


\subsection{The system in Lagrangian variables}
\label{sec:system}

With the   differentiated Riemann variables $\mathring{w}, \mathring{z},  \mathring{k}$ introduced   in~\eqref{def:q^w}, we  define 
\begin{equation*}
\Sigma : = \sigma \cir \eta \, , \qquad
\W : = \mathring{w}\cir \eta \, , \qquad
\Z : = \mathring{z} \cir \eta \, , \qquad
\K : = \mathring{k} \cir \eta\,.
\end{equation*}
These are functions of the Lagrangian variables $(x,t)$. 

It is straightforward to compute that
\begin{subequations}
\begin{align}
\label{id:Sigma_t}
\Sigma_t & = -\alpha \Sigma \Z + \tfrac{\alpha}{2\gamma} \Sigma^2 \K, \\
\label{id:Sigma_x}
\Sigma_x & = \tfrac{1}{2}\eta_x \W -\tfrac{1}{2}\eta_x \Z + \tfrac{1}{2\gamma} \eta_x \Sigma \K, \\
\label{id:Sigma^-1}
(\Sigma^{-1})_t & = \alpha \Sigma^{-1} \Z - \tfrac{\alpha}{2\gamma} \K, \\
\label{id:eta_xt}
\eta_{xt} & = \tfrac{1+\alpha}{2} \eta_x \W + \tfrac{1-\alpha}{2} \eta_x \Z + \tfrac{\alpha}{2\gamma} \eta_x \Sigma \K \\
\label{id:W_t}
(\eta_x \W)_t & = \tfrac{\alpha}{4\gamma} \Sigma \K (\eta_x \W + \eta_x \Z).
\end{align}
\end{subequations}

One can deduce from \eqref{eq:euler:RV} that
$$ \p_t(k\cir \eta) = \alpha \Sigma \K. $$
Letting $\p_x$ act on both sides  of this equation, we have that
\begin{equation*}
\p_t(\eta_x \K) = \alpha \Sigma \p_x \K + \alpha \Sigma_x \K,
\quad 
\Rightarrow 
\quad 
\eta_x \p_t \K = \alpha \Sigma \p_x \K + (\alpha \Sigma_x - \eta_{xt}) \K.
\end{equation*}
Plugging in \eqref{id:Sigma_x} and \eqref{id:eta_xt} gives us
\begin{equation}
\label{id:K_t}
\eta_x \p_t \K = \alpha \Sigma \p_x \K - \tfrac{1}{2} \K \eta_x \W - \tfrac{1}{2} \eta_x \K \Z.
\end{equation}

We can also derive a similar equation for $\Z$. Since
\begin{equation*}
 \p_t u + \lambda_3 \p_y u = \alpha \sigma \mathring{z}, 
 \end{equation*}
we conclude that
\begin{equation*}
 \p_t(u\cir \eta) = \alpha \Sigma \Z. 
 \end{equation*}
Letting $\p_x$ act on both sides of this equation and rearranging shows that
\begin{equation*}
\p_t( \eta_x \Z) = 2\alpha \p_x(\Sigma \Z) - (\eta_x \W)_t.
\end{equation*}
Expanding this equation and using \eqref{id:Sigma_x} - \eqref{id:W_t} gives us the identity
\begin{equation}
\label{id:Z_t}
\eta_x \p_t \Z = 2\alpha \Sigma \p_x \Z - \tfrac{1-\alpha}{2} \eta_x \W \Z - \tfrac{\alpha}{4\gamma} \Sigma \K \eta_x \W - \tfrac{1+\alpha}{2} \eta_x \Z^2 +\tfrac{\alpha}{4\gamma} \eta_x \Sigma \K \Z.
\end{equation}

For the rest of the paper, we will work with the variables $\eta_x, \Sigma, \eta_x \W, \Z,$ and $\K$, which solve the following system:
\begin{subequations}
\label{system:lagrange}
\begin{align}
\Sigma_t & = -\alpha \Sigma \Z + \tfrac{\alpha}{2\gamma} \Sigma^2 \K, \\
\Sigma_x & = \tfrac{1}{2}\eta_x \W -\tfrac{1}{2}\eta_x \Z + \tfrac{1}{2\gamma} \eta_x \Sigma \K, \\
(\Sigma^{-1})_t & = \alpha \Sigma^{-1} \Z - \tfrac{\alpha}{2\gamma} \K, \\
\eta_{xt} & = \tfrac{1+\alpha}{2} \eta_x \W + \tfrac{1-\alpha}{2} \eta_x \Z + \tfrac{\alpha}{2\gamma} \eta_x \Sigma \K \\
(\eta_x \W)_t & = \tfrac{\alpha}{4\gamma} \Sigma \K (\eta_x \W + \eta_x \Z) \\
\eta_x \p_t \K & = \alpha \Sigma \p_x \K - \tfrac{1}{2} \K \eta_x \W - \tfrac{1}{2} \eta_x \K \Z \\
\eta_x \p_t \Z & = 2\alpha \Sigma \p_x \Z - \tfrac{1-\alpha}{2} \eta_x \W \Z - \tfrac{\alpha}{4\gamma} \Sigma \K \eta_x \W - \tfrac{1+\alpha}{2} \eta_x \Z^2 +\tfrac{\alpha}{4\gamma} \eta_x \Sigma \K \Z.
\end{align}
\end{subequations}
Notice that we are treating the product $\eta_x \W$ as one of the basic variables rather than analyzing  $\W$ on its own. This is because the product $\eta_x \W$ will remain as smooth as the initial data while $\W$ does not.


\subsection{Initial estimates and determination of the blowup time}
\label{sec:initialE}

We will show that $\Sigma, \Sigma^{-1}, \eta_x, \eta_x \W, \Z$ and $\K$ all remain bounded as long as the solution exists, that the solution exists up until a time $T_* = \frac{2}{1+\alpha} + \OO_\alpha(\eps)$, and that at time $T_*$ $\eta_x$ must have a zero.

\begin{proposition}\label{prop:0thorder} Define the constants
\begin{equation}
B_z := 6^{\frac{2}{\min(1,\alpha)}}(2+\tfrac{1}{\gamma})e^{21}, 
\qquad 
\mbox{and}
\qquad
B_k  := 6^{\frac{1}{\alpha}}.
\end{equation}
Note that $B_k < B_z$. If $\eps$ is chosen small enough such that
\begin{equation*}
(1+ B_z) \eps \ll 1,
\end{equation*}
then for all $(x,t) \in \TT\times [0, T_* \wedge \tfrac{2}{1+\alpha}(1+\eps^{\frac{1}{2}}))$ we have that
\begin{subequations}
\label{ineq:0thorder}
\begin{align}
\tfrac{1}{3} \leq \Sigma & \leq 3, \\
|\eta_x \W| & \leq \tfrac{4}{3}, \\
\eta_x & \leq 3 , \\
|\K| & \leq B_k \eps, \\ 
|\Z| & \leq B_z \eps ,
\end{align} 
\end{subequations}
and
\begin{align}
\label{ineq:W}
\eta_x \W & \leq -\tfrac{1}{2} + 4 \eta_x.
\end{align}
\end{proposition}

\begin{proof}
We will proceed with a bootstrap argument. Fix $T \in [0, T_* \wedge \tfrac{2}{1+\alpha}(1+\eps^{\frac{1}{2}}))$ and pick a constant
\begin{equation*}
B_z < A < 2B_z.
\end{equation*}
We make the following bootstrap assumptions:
\begin{align*}
\tfrac{1}{4} \leq \Sigma & \leq 4 & \forall \: (x,t) \in \TT \times [0,T], \\
|\eta_x \W| & \leq 2 & \forall \: (x,t) \in \TT \times [0,T], \\
\eta_x & \leq 4 & \forall \: (x,t) \in \TT \times [0,T],\\
|\K| & \leq A \eps & \forall \: (x,t) \in \TT \times [0,T], \\ 
|\Z| & \leq A\eps & \forall \: (x,t) \in \TT \times [0,T]. 
\end{align*}
It follows from \eqref{id:Sigma_t} that
\begin{equation*}
\Sigma(x,t) = \sigma_0(x) \exp\bigg( \int^t_0 (-\alpha \Z + \tfrac{\alpha}{2\gamma} \Sigma\K)(x,s) \: ds \bigg).
\end{equation*}
Since $T \leq \frac{2}{1+\alpha}(1+\eps^{\frac{1}{2}})$ and $\frac{2\alpha}{1+\alpha}(1+ \frac{2}{\gamma}) < 2$, it follows that
\begin{equation*}
\Sigma, \Sigma^{-1} 
\leq 2e^{\alpha t A(1 + \frac{2}{\gamma}) \eps} 
\leq 2e^{2A(1+\eps^{\frac{1}{2}})\eps} \,.
\end{equation*}
It follows that if $\eps \ll B_z$ and $\eps \ll 1$ then
\begin{align*}
\tfrac{1}{3} \leq \Sigma & \leq 3 & \forall \: (x,t) \in \TT \times [0,T].
\end{align*}

To improve upon the second bootstrap assumption, plugging $\Sigma \leq 3$ and our bootstrap assumptions into \eqref{id:W_t} gives us the bound
\begin{align*}
|(\eta_x \W)_t| & \leq\tfrac{3}{8} A \eps (2+ 4A\eps).
\end{align*}
Since
\begin{equation*}
(\eta_x \W)(x,t) = w'_0(x) - \tfrac{1}{2\gamma} \sigma_0(x) k'_0(x) + \int^t_0 (\eta_x \W)_t (x,s) \: ds
\end{equation*}
it follows that
\begin{align}
\label{ineq:W-w_0'}
|(\eta_x \W)(x,t)-w'_0(x)| & \leq \tfrac{1}{\gamma} \eps + \tfrac{2}{1+\alpha}(1+ \eps^{\frac{1}{2}})\tfrac{3}{8} A \eps (2+ 4A\eps) \,.
\end{align}
Therefore,
\begin{equation*}
|\eta_x \W|
\leq 1 + \eps \bigl[ 1+ \tfrac{1}{\gamma} + (1+ \eps^{\frac{1}{2}}) \tfrac{3}{4} A (2+ 4A\eps) \bigr] 
\leq \tfrac{4}{3}
\end{equation*}
if $B_z\eps \ll 1$ and $\eps \ll 1$. This proves our second inequality.

To prove the third inequality in \eqref{ineq:0thorder}, our identity \eqref{id:eta_xt} for $\eta_{xt}$ paired with \eqref{ineq:W-w_0'} gives us
\begin{align*}
\eta_x = 1 + \int^t_0 \eta_{xt} \: ds 
& = 1 + \tfrac{1+\alpha}{2}t w'_0 + \tfrac{1+\alpha}{2} \int^t_0 (\eta_x \W-w'_0) \: ds +  \tfrac{1-\alpha}{2} \int^t_0 \eta_x \Z \: ds + \tfrac{\alpha}{2\gamma} \int^t_0 \eta_x \Sigma \K \: ds \\
& = 1 + \tfrac{1+\alpha}{2}t w'_0 + \OO ( \tfrac{1+\alpha}{2} t(1+ A\eps)(1+A) \eps) \\
& =  1 + \tfrac{1+\alpha}{2}t w'_0 + \OO ( (1+ A\eps)(1+A) \eps)
\end{align*}
provided that $\eps \ll 1$. Therefore
\begin{align}
\label{ineq:eta_xt-stuff}
|\eta_{xt}(x,t) -1 - \tfrac{1+\alpha}{2}tw'_0(x)| & \leq \OO ( (1+ A\eps)(1+A) \eps) \,.
\end{align}
From here it follows that
\begin{equation*}
\eta_x 
\leq 1 + (1+\eps^{\frac{1}{2}})(1+\eps) + \OO ( (1+ A\eps)(1+A) \eps) 
\leq 3
\end{equation*}
for all $(x,t) \in \TT \times [0, T]$.

Before improving upon the last two bootstrap assumptions, we will prove \eqref{ineq:W} for $t \leq T$. When $0 \leq t \leq T \wedge \frac{1}{1+\alpha}$, \eqref{ineq:eta_xt-stuff} gives us
\begin{align*}
\eta_x & \geq 1+ \tfrac{1+\alpha}{2} t w'_0 + \OO ( (1+ A\eps)(1+A) \eps) \\
& \geq 1 -\tfrac{1}{2}(1+\eps) + \OO ( (1+ A\eps)(1+A) \eps).
\end{align*}
Using this lower bound in conjunction with \eqref{ineq:W-w_0'} results in
\begin{align*}
 \eta_x \W & \leq 1 + \OO ( (1+ A\eps)(1+A) \eps) \\
& = 1- 4 \eta_x + 4\eta_x + \OO ( (1+ A\eps)(1+A) \eps) \\
& \leq -1 + 4\eta_x + \OO ( (1+ A\eps)(1+A) \eps) \\
& \leq -\tfrac{1}{2} + 4\eta_x.
\end{align*}
The last inequality here is true provided that $(1+B_z)\eps \ll 1$. For $\frac{1}{1+\alpha} \leq t \leq T$  we have $1+\alpha \geq \frac{1}{t} \geq \frac{1+\alpha}{2(1+\eps^{\frac{1}{2}})}$, so the inequalities \eqref{ineq:W-w_0'} and \eqref{ineq:eta_xt-stuff} imply that
\begin{align*}
\eta_x \W & = \tfrac{2}{1+\alpha} \tfrac{1}{t} (\eta_x -1) + \tfrac{2}{1+\alpha}\tfrac{1}{t}(1+ \tfrac{1+\alpha}{2}t w'_0 -\eta_x) + (\eta_x \W-w'_0) \\
& \leq 2 \eta_x-\tfrac{1}{1+\eps^{\frac{1}{2}}} + \tfrac{2}{1+\alpha}\tfrac{1}{t}(1+ \tfrac{1+\alpha}{2}t w'_0 -\eta_x) + (\eta_x \W-w'_0) \\
& = 2 \eta_x-\tfrac{1}{1+\eps^{\frac{1}{2}}} + \OO ( (1+ A\eps)(1+A) \eps) \\
& \leq 2\eta_x -\tfrac{1}{2}. 
\end{align*} 
This proves \eqref{ineq:W} for $(x,t) \in \TT \times [0,T]$. We will now use this inequality to prove the last two inequalities.

For $1< q< \infty$ and a constant $\delta \geq 0$ to be determined, define the quantities
\begin{align*}
\E^{0,q}_k(t) & : = \int_\TT \Sigma^{-\delta q} \eta_x |\K|^q \: dx, \\
\E^{0,q}_z(t) & : = \int_\TT \Sigma^{-\delta q} \eta_x |\Z|^q \: dx.
\end{align*}
Using \eqref{system:lagrange} we compute that
\begin{align*}
\dot{\E}^{0,q}_k 
& = -\delta q \int \tfrac{\Sigma_t}{\Sigma} \Sigma^{-\delta q} \eta_x |\K|^q + \int \Sigma^{-\delta q} \eta_{xt} |\K|^q + q\int \Sigma^{-\delta q}  |\K|^{q-1} \sgn(\K)  \eta_x \p_t \K \\
& = \int \Sigma^{-\delta q} \eta_x |\K|^q[ \alpha \delta q(\Z-\tfrac{1}{2\gamma} \Sigma \K) + ( \tfrac{1-\alpha}{2} \Z + \tfrac{\alpha}{2\gamma} \Sigma \K)] \\
&\quad 
+ \int \Sigma^{-\delta q} |\K|^q \tfrac{1+\alpha}{2} \eta_x \W 
+ q\int \Sigma^{-\delta q}  |\K|^{q-1} \sgn(\K)  \eta_x \p_t \K \\
& = \int \Sigma^{-\delta q} \eta_x |\K|^q[ \alpha \delta q(\Z-\tfrac{1}{2\gamma} \Sigma \K) + ( \tfrac{1-\alpha}{2} \Z + \tfrac{\alpha}{2\gamma} \Sigma \K)] 
+ \int \Sigma^{-\delta q}  |\K|^q \tfrac{1+\alpha}{2} \eta_x \W  
\\
&\quad 
+ \alpha \int \Sigma^{1-\delta q} \p_x \big( |\K|^q \big) 
+ \int \Sigma^{-\delta q}  |\K|^q [-\tfrac{q}{2} \eta_x \W] 
+ \int \Sigma^{-\delta q} \eta_x |\K|^q[-\tfrac{q}{2} \Z] \\
& = \int \Sigma^{-\delta q} \eta_x |\K|^q[ \alpha \delta q(\Z-\tfrac{1}{2\gamma} \Sigma \K) + ( \tfrac{1-\alpha}{2} \Z + \tfrac{\alpha}{2\gamma} \Sigma \K)] \\
&\quad + \int \Sigma^{-\delta q}  |\K|^q \tfrac{1+\alpha}{2} \eta_x \W
+ \int \Sigma^{-\delta q}  |\K|^q [\tfrac{1}{2}\alpha(\delta q-1) \eta_x \W] \\
&\quad + \alpha(\delta q-1)\int \Sigma^{-\delta q} \eta_x |\K|^q[-\tfrac{1}{2} \Z + \tfrac{1}{2\gamma} \Sigma \K] 
+ \int \Sigma^{-\delta q}  |\K|^q [-\tfrac{q}{2} \eta_x \W] + \int \Sigma^{-\delta q} \eta_x |\K|^q[-\tfrac{q}{2} \Z] \\
& = [\tfrac{1+\alpha}{2} + \tfrac{\alpha}{2}(\delta q-1) -\tfrac{q}{2}]  \int \Sigma^{-\delta q}  |\K|^q \eta_x \W \\
&\quad + [\alpha \delta q + \tfrac{1-\alpha}{2} - \tfrac{\alpha}{2} (\delta q-1) - \tfrac{q}{2}]\int \Sigma^{-\delta q} \eta_x |\K|^q \Z  + [-\alpha \delta q + \alpha + \alpha(\delta q-1)]\int \Sigma^{-\delta q} \eta_x |\K|^q \tfrac{1}{2\gamma} \Sigma \K \\
& = [ \tfrac{\alpha}{2}\delta q - \tfrac{q-1}{2} ]  \int \Sigma^{-\delta q}  |\K|^q \eta_x \W +  [ \tfrac{\alpha}{2}\delta q - \tfrac{q-1}{2} ] \int \Sigma^{-\delta q} \eta_x |\K|^q \Z.
\end{align*}
Performing analogous computations for $\E^{0,q}_z$, we find that
\begin{subequations}
\begin{align}
\label{eq:dtE:K:0}
\dot{\E}^{0,q}_k & = [ \tfrac{\alpha}{2}\delta q - \tfrac{q-1}{2} ]  \int \Sigma^{-\delta q}  |\K|^q \eta_x \W +  [ \tfrac{\alpha}{2}\delta q - \tfrac{q-1}{2} ] \int \Sigma^{-\delta q} \eta_x |\K|^q \Z, \\
\dot{\E}^{0,q}_z & =  [\alpha \delta q - \tfrac{1-\alpha}{2}(q-1)]\int \Sigma^{-\delta q}  |\Z|^q \eta_x \W  
- q\tfrac{\alpha}{4\gamma} \int \Sigma^{-\delta q} |\Z|^{q-1} \sgn(\Z) \Sigma \K \eta_x \W
\notag \\
&\quad  - \tfrac{1+\alpha}{2}(q-1) \int \Sigma^{-\delta q} \eta_x |\Z|^q \Z + [\alpha \delta q + \alpha (\tfrac{q}{2}-1)]  \int \Sigma^{-\delta q} \eta_x |\Z|^q \tfrac{1}{2\gamma} \Sigma \K 
. 
\label{eq:dtE:Z:0}
\end{align}
\end{subequations}

First, let us derive bounds on $\K$. Choosing $\delta = \frac{1}{\alpha}$ and using \eqref{eq:dtE:K:0}, \eqref{ineq:W}, and our bound $\eta_x \leq 3$ gives us
\begin{equation*}
\dot{\E}^{0,q}_k = \tfrac{1}{2} \int \Sigma^{-\frac{q}{\alpha}}  |\K|^q \eta_x \W +  \tfrac{1}{2}\int \Sigma^{-\frac{q}{\alpha}} \eta_x |\K|^q \Z 
 \leq -\tfrac{1}{4}\int \Sigma^{-\frac{q}{\alpha}}  |\K|^q + (2+ \tfrac{A}{2}\eps) \E^{0,q}_k 
\leq (\tfrac{23}{12} + \tfrac{A}{2}\eps) \E^{0,q}_k.
\end{equation*}
and thus
\begin{equation*}
 \E^{0,q}_k(t)^{1/q} 
 \leq \E^{0,q}_k(0)^{1/q} e^{\frac{1}{q}(\frac{23}{12} + \frac{A}{2}\eps) t}  
 \leq 2^{\frac{1}{\alpha}} \|k'_0\|_{L^q_x}  e^{\frac{1}{q}(\frac{23}{12} + \frac{A}{2}\eps) t} .
\end{equation*}
Therefore, for any $t \in [0,T]$ we have that
\begin{equation*}
\| \Sigma^{-\frac{1}{\alpha}} \K \|_{L^\infty_x} =
\lim_{q \rightarrow \infty} \E^{0,q}_k(t)^{1/q}
\leq \limsup_{q \rightarrow \infty}  2^{\frac{1}{\alpha}} \|k'_0\|_{L^\infty_x} e^{\frac{1}{q}(\frac{23}{12} + \frac{A}{2}\eps) t}
= 2^{\frac{1}{\alpha}}  \|k'_0\|_{L^\infty_x}.
\end{equation*}
So
\begin{equation*}
|\K| \leq 6^{\frac{1}{\alpha}} \|k'_0\|_{L^\infty_x} , \qquad \forall \: (x,t) \in \TT \times [0,T].
\end{equation*}

To bound $\Z$, it is best to bound $\E^{0,q}_k + \E^{0,q}_z$. If $\delta$ is large enough that $\tfrac{\alpha}{2}\delta q - \tfrac{q-1}{2} > 0$ and $\alpha \delta q - \tfrac{1-\alpha}{2}(q-1) > 0$, then adding \eqref{eq:dtE:K:0} and \eqref{eq:dtE:Z:0} together and using \eqref{ineq:W} gives us 
\begin{align*}
& \dot{\E}^{0,q}_k + \dot{\E}^{0,q}_z 
+ \tfrac{1}{2}[\tfrac{\alpha}{2}\delta q - \tfrac{q-1}{2}] \int \Sigma^{-\delta q} |\K|^q + \tfrac{1}{2}[ \alpha \delta q - \tfrac{1-\alpha}{2}(q-1)] \int \Sigma^{-\delta q} |\Z|^q \\
&  \leq [\tfrac{\alpha}{2}\delta q - \tfrac{q-1}{2}](4+ A\eps) \E^{0,q}_k  + \big( 4 [\alpha \delta q - \tfrac{1-\alpha}{2}(q-1)] + \tfrac{3}{2\gamma} A\eps [\alpha \delta q + \alpha (\tfrac{q}{2}-1)] + \tfrac{1+\alpha}{2}(q-1)A \eps \big) \E^{0,q}_z \\
&\quad - q\tfrac{\alpha}{4\gamma} \int \Sigma^{-\delta q} |\Z|^{q-1} \sgn(\Z) \Sigma \K \eta_x \W.
\end{align*}
Since $|\Sigma \eta_x \W| \leq 3 \cdot \frac{4}{3} = 4$ we have
\begin{align*}
-q\tfrac{\alpha}{4\gamma} \int \Sigma^{-\delta q} |\Z|^{q-1} \sgn(\Z) \Sigma \K \eta_x \W & \leq \tfrac{\alpha}{\gamma}(q-1) \int \Sigma^{-\delta q} |\Z|^q + \tfrac{\alpha}{\gamma} \int \Sigma^{-\delta q} |\K|^q,
\end{align*}
and our bound on $\dot{\E}^{0,q}_k + \dot{\E}^{0,q}_z$ becomes
\begin{align*}
& \dot{\E}^{0,q}_k + \dot{\E}^{0,q}_z  
+ \tfrac{1}{2}[\tfrac{\alpha}{2}\delta q - \tfrac{q-1}{2}-\tfrac{2\alpha}{\gamma}] \int \Sigma^{-\delta q} |\K|^q + \tfrac{1}{2}[ \alpha \delta q - \tfrac{1-\alpha}{2}(q-1) - \tfrac{2\alpha}{\gamma}(q-1)] \int \Sigma^{-\delta q} |\Z|^q \\
&  \leq [\tfrac{\alpha}{2}\delta q - \tfrac{q-1}{2}](4+ A\eps) \E^{0,q}_k 
+ \big( 4 [\alpha \delta q - \tfrac{1-\alpha}{2}(q-1)] + \tfrac{3}{2\gamma} A\eps [\alpha \delta q + \alpha (\tfrac{q}{2}-1)] + \tfrac{1+\alpha}{2}(q-1)A \eps \big) \E^{0,q}_z.
\end{align*}
Now choose $\delta = \frac{2}{\min(1,\alpha)}$.  This choice of $\delta$ is large enough that the factors in front of the damping terms are positive, and we get
\begin{align*}
 \dot{\E}^{0,q}_k + \dot{\E}^{0,q}_z  & \leq 10 \max(1, \alpha) q [1 + \tfrac{1}{2}A \eps] ( \E^{0,q}_k + \E^{0,q}_z).
 \end{align*}
Since $T \leq \frac{2}{1+\alpha} (1+ \eps^{\frac{1}{2}})$, we conclude that
\begin{align*}
(\E^{0,q}_k(t) + \E^{0,q}_z(t))^{1/q} & \leq (\E^{0,q}_k(0) + \E^{0,q}_z(0))^{1/q} e^{10 \max(1, \alpha)  [1 + \frac{1}{2}A \eps] t} \\
& \leq 2^{\frac{2}{\min(1,\alpha)}}(2+ \tfrac{1}{\gamma}) \max(\|k'_0\|_{L^\infty_x}, \|z'_0\|_{L^\infty_x}) e^{20(1+\eps^{\frac{1}{2}})[1+ \frac{1}{2}A\eps]}
\end{align*}
for all $t \in [0,T]$. If $\eps$ is small enough that $20(1+\eps^{\frac{1}{2}})(1+B_z \eps) < 21$, sending $q \rightarrow \infty$ now gives us the bound $|\Z| \leq B_z \max(\|k'_0\|_{L^\infty_x}, \|z'_0\|_{L^\infty_x})$ in the same manner as before, which finishes our bootstrap argument.
\end{proof}

\begin{proposition}\label{prop:T_*}
If $\eps$ is chosen small enough so that
\begin{equation*}
(1+B_z)^2 \eps \ll 1
\end{equation*}
then the blowup time $T_*$ satisfies
\begin{subequations}\label{approx:T_*}
\begin{equation}
\label{ineq:T_*}
 (1-\eps^{\frac{1}{2}}) \tfrac{2}{1+\alpha} < T_* < (1+\eps^{\frac{1}{2}})\tfrac{2}{1+\alpha},
\end{equation}
and in particular
\begin{equation}
\label{eq:T_*}
T_* = \tfrac{2}{1+\alpha}[1 + \OO( (1+B_z)\eps) ].
\end{equation}
\end{subequations}
It follows immediately from \eqref{ineq:T_*} that \eqref{ineq:0thorder} holds for all $(x,t) \in \TT \times [0, T_*)$.
\end{proposition}

\begin{proof}
Recall from the proof of the previous proposition that
\begin{align*}
\eta_x \W & = w'_0 + \OO( (1+B_z)\eps), \\
\eta_x & = 1 + \tfrac{1+\alpha}{2}t [w'_0 + \OO( (1+B_z)\eps)],
\end{align*}
for all $(x,t) \in \TT \times [0, T_* \wedge \frac{2}{1+\alpha}(1+ \eps^{\frac{1}{2}}) )$.
Since $\eta_x > 0$ for all $t < T_*$, it follows from \eqref{data:2} that 
\begin{align*}
0 < \eta_x < 1 - \tfrac{1+\alpha}{2} t [(1-\eps) + \OO( (1+B_z)\eps)]
\end{align*}
for all $(x,t) \in \TT \times [0, T_* \wedge \frac{2}{1+\alpha}(1+ \eps^{\frac{1}{2}}) )$. Therefore,
\begin{equation*}
\tfrac{1+\alpha}{2} T_* \wedge (1+\eps^{\frac{1}{2}}) 
\leq \tfrac{1}{1-\eps + \OO( (1+B_z)\eps)} 
= 1 + \OO( (1+B_z)\eps).
\end{equation*}
Since $(1+B_z)^2 \eps \ll 1$, we can conclude that
\begin{align*}
 T_* \leq \tfrac{2}{1+\alpha}[1 + \OO( (1+B_z)\eps)] < \tfrac{2}{1+\alpha}(1+ \eps^{\frac{1}{2}}).
\end{align*}
Therefore, $T_* = T_* \wedge \frac{2}{1+\alpha}(1+ \eps^{\frac{1}{2}})$,  and the inequalities \eqref{ineq:0thorder} are true on $\TT \times [0, T_*)$.

It follows that
\begin{align*}
\int^{T_*}_0 \Bigl(\|\p_y z(\cdot, t)\|_{L^\infty_y} + \|\p_y k(\cdot, t)\|_{L^\infty_y}\Bigr)    dt & \leq T_*[ B_z + (1+ \tfrac{3}{2\gamma}) B_k] \eps,
\end{align*}
and
\begin{align*}
\int^{T_*}_0 \|\p_y w(\cdot, t)\|_{L^\infty_y} \: dt & \leq \int^{T_*}_0 \tfrac{4}{3} \tfrac{1}{\min_{x \in \TT} \eta_x(x, t)} + \tfrac{3}{2\gamma} B_k \eps \: dt.
\end{align*}
Since $T_*$ is finite, it follows from \eqref{eq:LCC:2} that
\begin{equation*}
\lim_{t \rightarrow T_*} \min_{x \in \TT} \eta_x(x, t) = 0.
\end{equation*}
For all $t < T_*$, our initial data assumption \eqref{data:2} implies that
\begin{align*}
\eta_x & > 1 - \tfrac{1+\alpha}{2} t [1+ \eps + \OO( (1+B_z)\eps)].
\end{align*}
Sending $t \rightarrow T_*$ and letting $(1+B_z)^2 \eps$ be sufficiently small gives us
\begin{equation*}
\tfrac{1+\alpha}{2} T_* 
\geq \tfrac{1}{1+ \OO((1+B_z) \eps)}
= 1+ \OO((1+B_z)\eps)
> 1-\eps^{\frac{1}{2}}.
\end{equation*}

\end{proof}


\section{Higher order estimates}
\label{sec:HOE}

In this section, we will prove estimates on the derivatives of $\Sigma, \eta_x, \eta_x\W, \Z,$ and $\K$  with constants that do not depend on our choice of $(w_0,z_0, k_0) \in \mathcal A_n(\eps, C_0)$, and whose dependence on $\eps, \alpha, n,$ and $C_0$ is explicit. For the entirety of this section, we assume that $\eps$ is chosen small enough such that the results of the previous section hold.


\subsection{Estimates at time zero}
\label{sec:ineq:t=0}

\begin{proposition}\label{prop:t=0}
 Let $\C_x$ and $\C_t$ and be constants satisfying
\begin{itemize}[leftmargin=*]
\item $\C_x \geq 2e^3 C_0$,
\item $\C_x \gg 1$,
\item $\C_t \gg (1+\alpha)$,
\item $\alpha \C_x \ll \C_t$.
\end{itemize}
Then if $(1+\alpha)\eps \ll 1$ the following inequalities hold at time $t=0$ for all $|\beta| \leq 2n+1$:
\begin{subequations}\label{ineq:prop:t=0}
\begin{align}
\label{ineq:prop:t=0:1}
\frac{(|\beta|+1)^2\inorm{\p^\beta \K}_{L^\infty_x}}{|\beta|! \C^{\beta_x}_x \C^{\beta_t}_t} & \leq \eps \\
\label{ineq:prop:t=0:2}
\frac{(|\beta|+1)^2\inorm{\p^\beta \Z}_{L^\infty_x}}{|\beta|! \C^{\beta_x}_x \C^{\beta_t}_t} & \leq (1+\tfrac{1}{\gamma}) \eps  \\
\label{ineq:prop:t=0:3}
\frac{(|\beta|+1)^2\inorm{\p^\beta \Sigma_t}_{L^\infty_x}}{|\beta|! \C^{\beta_x}_x \C^{\beta_t}_t} & \leq 7\alpha \eps   \\
\label{ineq:prop:t=0:4}
\frac{(|\beta|+1)^2\inorm{\p^\beta \eta_{xt}}_{L^\infty_x}}{|\beta|! \C^{\beta_x}_x \C^{\beta_t}_t} & \leq \tfrac{1+\alpha}{2}(1+2\eps) \\
\label{ineq:prop:t=0:5}
\frac{(|\beta|+1)^2\inorm{\p^\beta (\eta_x \W)_t}_{L^\infty_x}}{|\beta|! \C^{\beta_x}_x \C^{\beta_t}_t} & \leq 8\tfrac{\alpha}{\gamma}\eps.
\end{align}
\end{subequations}
\end{proposition}

\begin{proof}
We will prove this via induction on $\beta_t$.

{\it{Base Case:}} At time $t =0$
\begin{equation*}
\frac{(m+1)^2\inorm{ \p^m_x \K}_{L^\infty_x}}{m! \C^{m}_x} 
= \frac{(m+1)^2\inorm{ \p^{m+1}_x k_0}_{L^\infty_x}}{m! \C^{m}_x} 
\leq \big( \tfrac{e^2 C_0}{\C_x} \big)^m \eps 
\leq \eps.
\end{equation*}

Since
\begin{align*}
\p^m_x \Z 
= \p^{m+1}_x z_0 + \tfrac{1}{2\gamma} \sigma_0 \p^{m+1}_x k_0 + \tfrac{1}{2\gamma} \sum^{m-1}_{j=0} {m \choose{j}} \p^{j+1}_x k_0 \p^{m-1-j+1}_x \sigma_0,
\end{align*}
at time zero, our assumptions $C_0 \geq 3$ and $\C_x \geq 2e^3 C_0$ imply that
\begin{align*}
\frac{(m+1)^2 \inorm{\p^m_x \Z}_{L^\infty_x}}{m! \C^m_x \eps} & \leq \big( \tfrac{e^2 C_0}{\C_x} \big)^m(1+\tfrac{1}{\gamma}) + \tfrac{1}{2\gamma} ( \tfrac{1}{2} + \eps) \tfrac{1}{\C_x} \big( \tfrac{C_0}{\C_x} \big)^{m-1} (m+1)^2 \sum^{m-1}_{j=0} {m\choose{j}} \tfrac{j! (m-1-j)!}{m!} \\
& \leq \big( \tfrac{e^2 C_0}{\C_x} \big)^m(1+\tfrac{1}{\gamma}) + \tfrac{1}{2\gamma} ( \tfrac{1}{2} + \eps) \tfrac{1}{C_0} \big( \tfrac{e^2C_0}{\C_x} \big)^{m}  \sum^{m-1}_{j=0} \tfrac{1}{m-j} \\
& \leq (\tfrac{1}{2e})^m(1+ \tfrac{1}{\gamma}) + \tfrac{1}{2\gamma} ( \tfrac{1}{2} + \eps) \mathbbm{1}_{\{m \geq 1\}} (\tfrac{1}{2e})^{m}(1+\log m) \\
& \leq (\tfrac{1}{2e})^m(1+ \tfrac{1}{\gamma}) + \tfrac{1}{2\gamma} ( \tfrac{1}{2} + \eps)\tfrac{1}{2e}\mathbbm{1}_{\{m \geq 1\}}  
\leq 1 + \tfrac{1}{\gamma}.
\end{align*}
Similar computations applied to $\eta_x\mathring{W}$ at time zero also give us the bound
\begin{equation}
\label{ineq:W_x:t=0}
\frac{(m+1)^2\|\p^m_x (\eta_x \W)\|_{L^\infty_x}}{m! \C^m_x} \leq 1 + (1+ \tfrac{1}{\gamma}) \eps < 1+ 2\eps
\end{equation}
for all $m= 0, \hdots, 2n+1$.

Since $\Sigma_t = -\alpha \Sigma \p_y z\circ \eta$, at time zero we have
\begin{align*}
\Sigma_t = - \alpha \sigma_0 z'_0.
\end{align*}
Therefore
\begin{align*}
-\frac{\p^m_x \Sigma_t}{\alpha} & = \sigma_0 \p^{m+1}_x z_0 + \sum^{m-1}_{j=0} {m \choose{j}} \p^{j+1}_x z_0 \p^{m-1-j+1}_x \sigma_0,
\end{align*}
and computations analogous to those applied to $\p^m_x \Z$ and $\p^m_x (\eta_x \W)$ give us
\begin{equation*}
\frac{(m+1)^2\inorm{\p^m_x \Sigma_t}_{L^\infty_x}}{m! \C^m_x \alpha \eps} 
\leq 2(\tfrac{1}{2e})^m + ( \tfrac{1}{2} + \eps)\tfrac{1}{2e}\mathbbm{1}_{\{m \geq 1\}} 
\leq 2.
\end{equation*}

Let $0 \leq m \leq 2n+1$. At time zero
\begin{align*}
(\eta_x \W)_t & = \tfrac{\alpha}{4\gamma} \sigma_0 k'_0 (w'_0 + z'_0), \\
\tfrac{4\gamma}{\alpha}\p^m_x (\eta_x \W)_t & = \sigma_0 \sum^m_{j=0} {m \choose{j}} \p^{j+1}_x k_0 \p^{m-j+1}_x (w_0 + z_0)  \\& + \sum_{\substack{j_1+j_2+j_3 = m \\ j_1 \geq 1}} {m\choose{j_1 j_2 j_3}} \p^{j_1-1+1}_x \sigma_0 \p^{j_2+1}_x k_0 \p^{j_3+1}_x(w_0+z_0). \\
\implies \tfrac{4\gamma}{\alpha} \frac{(m+1)^2 \| \p^m_x (\eta_x \W)_t\|_{L^\infty_x}}{m! \C^m_x \eps} & \leq 2(1+2\eps)\big( \tfrac{ C_0}{\C_x}\big)^m (m+1)^3 + \tfrac{(1+2\eps)^2}{2} \tfrac{1}{C_0} \big( \tfrac{C_0}{\C_x} \big)^{m} (m+1)^2 \sum_{\substack{j_1+j_2+j_3 = m \\ j_1 \geq 1}} \tfrac{1}{j_1}.
\end{align*}
Since
\begin{equation*}
\sum_{\substack{j_1+j_2+j_3 = m \\ j_1 \geq 1}} \tfrac{1}{j_1} = (m+1) \sum^m_{j=1} \tfrac{1}{j} - m,
\end{equation*}
our hypothesis that $\C_x \geq 2e^3 C_0$ and $C_0 \geq 3$ now gives us
\begin{align*}
\tfrac{4\gamma}{\alpha} \frac{(m+1)^2 \| \p^m_x (\eta_x \W)_t\|_{L^\infty_x}}{m! \C^m_x \eps} & \leq 2(1+2\eps)\big( \tfrac{ e^3C_0}{\C_x}\big)^m + \tfrac{(1+2\eps)^2}{2}\tfrac{1}{C_0} \big( \tfrac{e^3C_0}{\C_x} \big)^m \sum^m_{j=1} \tfrac{1}{j} \\
& \leq 2(1+2\eps)(\tfrac{1}{2})^m + \mathbbm{1}_{\{m \geq 1\}} \tfrac{(1+2\eps)^2}{4}\tfrac{1}{C_0} 
\leq 3.
\end{align*}

Lastly, since $\eta_{xt} = \tfrac{1+\alpha}{2} w'_0 + \tfrac{1-\alpha}{2} z_0'$ at time zero, we compute that
\begin{align*}
\frac{(m+1)^2 \|\p^m_x \eta_{xt}\|_{L^\infty_x}}{m! \C^m_x} 
&\leq \big( \tfrac{e^2 C_0}{\C_x} \big)^m( \tfrac{1+\alpha}{2}(1+\eps) + \tfrac{|1-\alpha|}{2} \eps) \\
&= (\tfrac{1+\alpha}{2} + \max(1, \alpha) \eps) \big( \tfrac{e^2 C_0}{\C_x} \big)^m 
\leq \tfrac{1+\alpha}{2} (1 + 2\eps) (\tfrac{1}{2e})^m
\end{align*}
for all $m = 0, \hdots, 2n+1$. This concludes the base case.

{\it{Inductive Step:}} Now fix $0 \leq m \leq 2n$ and suppose our result is true for all multi indices $\beta$ with $|\beta| \leq 2n+1$ and $\beta_t \leq m$.

Let $\beta$ have $|\beta| \leq 2n$ and $\beta_t = m$. Taking $\p^\beta$ of \eqref{id:K_t}, setting $t=0$, and simplifying using the fact that
\begin{equation}\label{eq:eta_x:dx:t=0}
\p^i_x \eta_x(x, 0) = \delta_{i0}
\end{equation}
for all nonnegative integers $i$, we arrive at the identity
\begin{align*}
\p^{\beta + e_t} \K 
& = - \sum_{e_t \leq \gamma \leq \beta} {\beta\choose{\gamma}} \p^{\gamma-e_t} \eta_{xt} \p^{\beta+e_t-\gamma} \K \\
&\quad  + \alpha \sum_{ e_t \leq \gamma \leq \beta} {\beta\choose{\gamma}} \p^{\gamma-e_t} \Sigma_t \p^{\beta+e_x-\gamma} \K + \alpha \sigma_0 \p^{\beta+e_x} \K \\
&\quad + \alpha \sum^{\beta_x}_{j=1} {\beta_x\choose{j}} \p^j_x \sigma_0 \p^{\beta-(j-1)e_x} \K - \tfrac{1}{2} \sum^{\beta_x}_{j=0} {\beta_x\choose{j}} \p^{j}_x (\eta_x \W) \p^{\beta-je_x} \K \\
&\quad - \tfrac{1}{2} \sum_{e_t \leq \gamma \leq \beta} {\beta\choose{\gamma}} \p^{\beta-\gamma} \K \p^{\gamma-e_t}(\eta_x \W)_t - \tfrac{1}{2} \sum_{0 \leq \gamma \leq \beta} {\beta \choose{\gamma}} \p^\gamma \K \p^{\beta-\gamma} \Z \\
&\quad  - \tfrac{1}{2} \sum_{\substack{\gamma_1 + \gamma_2 + \gamma_3 = \beta \\ \gamma_1 \geq e_t}} {\beta \choose{\gamma_1 \gamma_2 \gamma_3}} \p^{\gamma_1 -e_t}\eta_{xt} \p^{\gamma_2} \K \p^{\gamma_3} \Z.
\end{align*}
Applying our inductive hypotheses, the inequality  \eqref{ineq:W_x:t=0},  Lemma \ref{lem:id:multi},  and \eqref{ineq:sum:nk}--\eqref{ineq:appendix:B} yields
\begin{align*}
&\frac{(|\beta+e_t|+1)^2 \inorm{\p^{\beta+e_t}\K}_{L^\infty_x}}{|\beta+e_t|! \C^{\beta_x}_x \C^{\beta_t+1}_t \eps} \notag\\
& \leq \tfrac{1}{\C_t}(\tfrac{1+\alpha}{2}(1+2\eps) +2 \alpha^2 \tfrac{\C_x}{\C_t}  \eps) \sum^{|\beta|}_{j=1} {|\beta|\choose{j}} \tfrac{(|\beta|+2)^2 (j-1)! (|\beta|+1-j)!}{(|\beta|+1)! j^2 (|\beta| + 2-j)^2} \\
&\quad + 2\alpha \tfrac{\C_x}{\C_t} 
+ \alpha(\tfrac{1}{2} + \eps)\tfrac{1}{\C_t} \sum^{\beta_x}_{j=1} {\beta_x\choose{j}} \tfrac{(|\beta|+2)^2 (j-1)! (|\beta|+1-j)!}{(|\beta|+1)! j^2 (|\beta| + 2-j)^2} \big( \tfrac{C_0}{\C_x} \big)^{j-1} \\
&\quad + \tfrac{1}{2}(1+2\eps)\tfrac{1}{\C_t} \sum^{\beta_x}_{j=0} {\beta_x\choose{j}} \tfrac{(|\beta|+2)^2 j! (|\beta|-j)!}{(|\beta|+1)! (j+1)^2 (|\beta| + 1-j)^2} \big( \tfrac{C_0}{\C_x} \big)^{j} \\
&\quad + \tfrac{1}{2\C^2_t} \tfrac{3\alpha}{4\gamma} \eps \sum^{|\beta|}_{j=1} {|\beta|\choose{j}} \tfrac{(|\beta|+2)^2 (|\beta|-j)! (j-1)!}{(|\beta|+1)! (|\beta|+1-j)^2 j^2} 
+ \tfrac{1}{2\C_t}(1+ \tfrac{1}{\gamma}) \eps \sum^{|\beta|}_{j=0} {|\beta| \choose{j}} \tfrac{(|\beta|+2)^2 j! (|\beta|-j)!}{(|\beta|+1)! (j+1)^2 (|\beta|+1-j)^2} \\
&\quad +  \tfrac{1+\alpha}{4}(1+2\eps)(1+\tfrac{1}{\gamma}) \eps \tfrac{1}{\C^2_t}  \sum_{\substack{j_1 + j_2 + j_3 = |\beta| \\ j_1 \geq 1}} {|\beta| \choose{j_1 j_2 j_3}} \tfrac{(|\beta|+2)^2 (j_1-1)! j_2! j_3!}{(|\beta|+1)! j_1^2 (j_2+1)^2 (j_3+1)^2} \\
&\leq \tfrac{1}{\C_t}(\tfrac{1+\alpha}{2}(1+2\eps) + 2\alpha^2 \tfrac{\C_x}{\C_t} \eps)  \sum^{|\beta|}_{j=1}  \tfrac{|\beta|+2}{j^3 (|\beta|+2-j)}
+ 2\alpha \tfrac{\C_x}{\C_t} 
+ \alpha(1 + 2\eps)\tfrac{1}{\C_t} \sum^{|\beta|}_{j=1} \tfrac{|\beta|+2}{j^3(|\beta|+2-j)}\\
&\quad + (1+\eps)\tfrac{1}{\C_t} \sum^{|\beta|}_{j=0} \tfrac{|\beta|+2}{(j+1)^2(|\beta|+1-j)^2}
+ \tfrac{1}{\C^2_t} \tfrac{3\alpha}{4\gamma}  \eps \sum^{|\beta|}_{j=1} \tfrac{|\beta|+2}{ j^3(|\beta|+1-j)^2 } 
+ \tfrac{1}{\C_t}(1+ \tfrac{1}{\gamma}) \eps \sum^{|\beta|}_{j=0} \tfrac{|\beta|+2}{(j+1)^2 (|\beta|+1-j)^2} \\
&\quad +  \tfrac{1+\alpha}{2}(1+2\eps)(1+\tfrac{1}{\gamma}) \eps \tfrac{1}{\C^2_t}  \sum_{\substack{j_1 + j_2 + j_3 = |\beta| \\ j_1 \geq 1}} \tfrac{|\beta|+2}{ j_1^3 (j_2+1)^2 (j_3+1)^2} \\
& \lesssim \tfrac{1}{\C_t} [ \alpha \C_x(1+\alpha \eps) + (1+\alpha)(1+\eps)(1+ \tfrac{1}{\C_t}) ].
\end{align*}
It follows that the bound \eqref{ineq:prop:t=0:1} holds for $\p^{\beta+e_t} \K$ provided that $\C_t \gg \alpha \C_x$, $\C_t \gg 1+\alpha$, and $(1+\alpha)\eps \ll 1$.

The inductive step for \eqref{ineq:prop:t=0:2} works out in the same way: pick $\beta$ with $|\beta| \leq 2n$ and $\beta_t = m$; take $\p^\beta$ of \eqref{id:Z_t}; set $t=0$ and simplify; carry out the same type of computations as done above; if $\C_t$ is large enough with respect to $\alpha \C_x$ and $1+\alpha$, then the right hand side of the inequalities  can be made smaller than $1+\tfrac{1}{\gamma}$ provided that $(1+\alpha)\eps \ll 1$.

The inductive steps for proving \eqref{ineq:prop:t=0:3}--\eqref{ineq:prop:t=0:5} will  now utilize the fact that we have already proven \eqref{ineq:prop:t=0:1} and \eqref{ineq:prop:t=0:2} for $\beta$ with $|\beta| \leq 2n+1$ and $\beta_t \leq m+1$. Pick $\beta$ with $|\beta| \leq 2n+1$ and $\beta_t = m+1$. Note that since $\beta_t \geq 1$, we must have $\beta_x \leq 2n$. Taking $\p^\beta$ of \eqref{id:Sigma_t}, setting $t=0$, and simplifying yields
\begin{align*}
\tfrac{1}{\alpha}\p^\beta \Sigma_t 
& = -\sigma_0\p^\beta \Z + \tfrac{1}{2\gamma} \sigma^2_0 \p^\beta \K 
 + \sum^{\beta_x}_{j=1} {\beta_x\choose{j}} [ -\p^{j-1+1}_x \sigma_0 \p^{\beta-je_x} \Z + \tfrac{1}{\gamma} \sigma_0 \p^{j-1+1}_x \sigma_0 \p^{\beta-je_x} \K] \\
&\quad + \sum_{e_t \leq \gamma \leq \beta} {\beta\choose{\gamma}} [- \p^{\gamma-e_t} \Sigma_t \p^{\beta-\gamma} \Z + \tfrac{1}{\gamma} \sigma_0 \p^{\gamma-e_t} \Sigma_t \p^{\beta-\gamma} \K] \\
&\quad + \tfrac{1}{2\gamma}\sum_{\substack{j_1+j_2 +j_3 = \beta_x \\ j_1, j_2 \geq 1}} {\beta_x\choose{j_1 j_2 j_3}} \p^{j_1-1+1}_x \sigma_0 \p^{j_2-1+1}_x \sigma_0 \p^{\beta_t e_t + j_3 e_x} \K \\
&\quad + \tfrac{1}{2\gamma} \sum_{\substack{je_x + \gamma \leq \beta \\ \gamma \geq e_t \\ j \geq 1}} {\beta \choose{je_x \:  \gamma \: \:  \beta-je_x-\gamma}} \p^{j-1+1}_x \sigma_0 \p^{\gamma-e_t}_t \Sigma \p^{\beta - \gamma - je_x} \K \\
&\quad + \tfrac{1}{2\gamma} \sum_{\substack{\gamma_1+\gamma_2 + \gamma_3 = \beta \\ \gamma_1, \gamma_2 \geq e_t}} {\beta \choose{\gamma_1 \gamma_2 \gamma_3}} \p^{\gamma_1-e_t} \Sigma_t \p^{\gamma_2-e_t} \Sigma_t \p^{\gamma_3} \K.
\end{align*}
Therefore, applying our inductive hypotheses, the fact that \eqref{ineq:prop:t=0:1} and \eqref{ineq:prop:t=0:2} have been proven for multiindices with $\beta_t \leq m+1$, and the inequalities in \S~\ref{sec:appendix:freqcomp} gives us
\begin{align*}
\frac{(|\beta|+1)^2 \inorm{ \p^\beta \Sigma_t}_{L^\infty_x}}{|\beta|! \C^{\beta_x}_x \C^{\beta_t}_t \alpha \eps} & \leq 2+ \tfrac{4}{\gamma} + \OO( \tfrac{1}{\C_x} + \tfrac{\alpha \eps}{\C_t} + \tfrac{1}{\C^2_x} + \tfrac{\alpha \eps}{\C_x \C_t} + \tfrac{\alpha^2 \eps^2}{\C^2_t} ).
\end{align*}
Since $\C_x \gg 1, \C_t \gg (1+\alpha),$ and $\eps \ll 1$, we have proven \eqref{ineq:prop:t=0:3} for our choice of $\beta$.

The proof for \eqref{ineq:prop:t=0:4} is analogous: if $|\beta| \leq 2n+1$ and $\beta_t = m+1$, then taking $\p^\beta$ of \eqref{id:eta_xt}, setting $t=0$, and simplifying produces
\begin{align*}
\p^\beta \eta_{xt} & =  \tfrac{1-\alpha}{2} \p^\beta \Z + \tfrac{\alpha}{2\gamma} \sigma_0 \p^\beta \K \\
&\quad + \tfrac{1+\alpha}{2} \p^{\beta-e_t} (\eta_x \W)_t   + \sum_{e_t \leq \gamma \leq \beta} {\beta\choose{\gamma}}\big[ \tfrac{1-\alpha}{2}  \p^{\gamma-e_t} \eta_{xt} \p^{\beta-\gamma} \Z +  \tfrac{\alpha}{2\gamma} \sigma_0\p^{\gamma-e_t} \eta_{xt} \p^{\beta-\gamma} \K\big] \\
&\quad + \tfrac{\alpha}{2\gamma} \sum^{\beta_x}_{j=1} {\beta_x\choose{j}} \p^{j-1+1}_x \sigma_0 \p^{\beta-je_x} \K 
+ \tfrac{\alpha}{2\gamma} \sum_{\substack{\gamma_1 + je_x + \gamma_3 = \beta \\ \gamma_1 \geq e_t, j \geq 1}} {\beta\choose{\gamma_1 je_x \gamma_3}} \p^{\gamma_1-e_t} \eta_{xt} \p^{j-1+1}_x \sigma_0 \p^{\gamma_3} \K \\
&\quad  + \tfrac{\alpha}{2\gamma} \sum_{\substack{\gamma_1 + \gamma_2 + \gamma_3 = \beta \\ \gamma_1,\gamma_2 \geq e_t, }} {\beta\choose{\gamma_1 \gamma_2 \gamma_3}} \p^{\gamma_1-e_t} \eta_{xt} \p^{\gamma_2-e_t}\Sigma_t \p^{\gamma_3} \K.
\end{align*}
Therefore,
\begin{align*}
\frac{(|\beta|+1)^2 \inorm{\p^\beta \eta_{xt}}_{L^\infty_x}}{|\beta|! \C^{\beta_x}_x \C^{\beta_t}_t (1+\alpha) \eps} & \leq \tfrac{1}{1+\alpha}(\tfrac{|1-\alpha|}{2}(1+ \tfrac{1}{\gamma}) + \tfrac{\alpha}{\gamma}) + \OO ( \tfrac{1+|1-\alpha|}{\C_t} + \tfrac{1}{(1+\alpha)\C_x} + \tfrac{1}{\C_x \C_t} + \tfrac{\alpha \eps}{\C^2_t} ) \\
& \leq 1 + \OO ( \tfrac{1+|1-\alpha|}{\C_t} + \tfrac{1}{(1+\alpha)\C_x} + \tfrac{1}{\C_x \C_t} + \tfrac{\alpha \eps}{\C^2_t}) .
\end{align*}
Since $\eps \ll 1$, $\C_x \gg 1$, and $\C_t \gg (1+\alpha)$, this proves \eqref{ineq:prop:t=0:4}.

The inductive step for \eqref{ineq:prop:t=0:5} works out in the same way as the inductive steps for \eqref{ineq:prop:t=0:3} and \eqref{ineq:prop:t=0:4}. If $|\beta| \leq 2n+1$ and $\beta_t = m+1$, then using \eqref{ineq:W_x:t=0} in addition to the same types of bounds as in the previous computations gives us
\begin{align*}
\frac{(|\beta|+1)^2\inorm{\p^\beta (\eta_x \W)_t}_{L^\infty_x}}{|\beta|! \C^{\beta_x}_x \C^{\beta_t}_t \tfrac{\alpha}{4\gamma} \eps} & \leq 2(1+2\eps) \sum^{|\beta|}_{j=0} \tfrac{(|\beta|+1)^2}{(j+1)^2(|\beta|+1-j)^2} + \OO( \eps + \tfrac{(1+\alpha)\eps}{\C_t} + \tfrac{1}{\C_x} )\\
&  \leq \tfrac{8\pi^2}{3}(1+2\eps) +  \OO( \eps + \tfrac{(1+\alpha)\eps}{\C_t} + \tfrac{1}{\C_x} ).
\end{align*}
Our hypotheses on the relative sizes of $\C_t, \C_x, \eps,$ and $\alpha$ now imply our desired bound on $\p^\beta(\eta_x \W)_t$.
\end{proof}

\subsection{ A Priori estimates for time derivatives}
\label{sec:apriori}

In this section, we will fix $T \in (0, T_*]$, a function $C_t : [0,T_*) \rightarrow \R^+$, and a constant $\delta \geq 0$. Suppose that $A_k \geq B_k$ and $A_z \geq B_z$ are constants such that $A_z \geq 10 A_k$ and
\begin{subequations}
\begin{align}
\label{ineq:apriori:K_t}
\frac{(m+1)^2 \|\Sigma^{-\delta m} \p^m_t \K\|_{L^\infty_x}}{m! C^m_t} & \leq A_k \eps, \\
\label{ineq:apriori:Z_t}
\frac{(m+1)^2 \|\Sigma^{-\delta m} \p^m_t \Z\|_{L^\infty_x}}{m! C^m_t} & \leq A_z \eps,
\end{align}
\end{subequations}
for all  $t \in [0,T], \: 0 \leq m \leq 2n+1.$

\begin{proposition}\label{prop:apriori:est}
If 
\begin{itemize}[leftmargin=*]
\item $(1+\alpha)3^\delta \ll C_t$ for all $t \in [0,T]$, and
\item $(1+A_z) \eps \ll 1$,
\end{itemize}
then
\begin{subequations}\label{ineq:apriori:est}
\begin{align}\label{apriori:Sigma_t}
\frac{(m+1)^2 \|\Sigma^{-\delta m} \p^{m}_t (\Sigma^{-1})_t\|_{L^\infty_x}}{m! C^m_t} & \leq 4\alpha A_z \eps, &
\frac{(m+1)^2 \|\Sigma^{-\delta m} \p^m_t \Sigma_t\|_{L^\infty_x}}{m! C^m_t} & \leq 4\alpha A_z \eps, \\
\label{apriori:W_t}
\frac{(m+1)^2 \|\Sigma^{-\delta m} \p^m_t \eta_{xt} \|_{L^\infty_x}}{m! C^m_t} & \leq (1+\alpha)(\tfrac{2}{3}\delta_{0m} + 3A_z\eps), &
\frac{(m+1)^2 \|\Sigma^{-\delta m} \p^m_t (\eta_x\W)_t\|_{L^\infty_x}}{m! C^m_t} & \leq \tfrac{5}{8} A_k \eps,
\end{align}
\end{subequations}
for all  $t \in [0,T], \: 0 \leq m \leq 2n+1.$
\end{proposition}

\begin{proof}
We prove our result via induction on $m$. The base case, $m=0$, is immediate from \eqref{system:lagrange}, \eqref{ineq:0thorder}, and the fact that $A_k \geq B_k$ and $A_z \geq B_z$.

For the inductive step, suppose that our inequalities are true up to $m-1$ for some $1 \leq m \leq 2n+1$. 

Applying $\p^m_t$ to both sides of \eqref{id:Sigma^-1} yields
\begin{align*}
\tfrac{1}{\alpha} \p^{m}_t (\Sigma^{-1})_t & = \Sigma^{-1} \p^m_t \Z - \tfrac{1}{2\gamma} \p^m_t \K + \sum^m_{j=1} {m\choose{j}} \p^{j-1}_t (\Sigma^{-1})_t \p^{m-j}_t \Z.
\end{align*}
The bounds from \S~\ref{sec:appendix:freqcomp} now give us
\begin{align*}
\frac{(m+1)^2 \|\Sigma^{-\delta m} \p^{m}_t (\Sigma^{-1})_t\|_{L^\infty_x}}{m! C^m_t \alpha \eps} & \leq 3A_z + \tfrac{1}{2\gamma} A_k  + 4\alpha  A^2_z \eps \tfrac{3^\delta}{C_t}\sum^m_{j=1} \tfrac{(m+1)^2}{j^3 (m+1-j)^2} \\
& \leq [3 + \tfrac{1}{20\gamma}+ 100 A_z\eps \tfrac{3^\delta \alpha}{C_t}] A_z  \\
& \leq 4A_z .
\end{align*}
Here the last inequality holds because of our hypotheses $(1+A_z) \eps \ll 1$ and $C_t \gg 3^\delta(1+\alpha)$.

Taking $\p^m_t$ of \eqref{id:Sigma_t} gives us
\begin{align*}
\tfrac{1}{\alpha} \p^m_t \Sigma_t & = - \Sigma \p^m_t \Z + \tfrac{1}{2\gamma} \Sigma^2 \p^m_t \K 
- \sum^m_{j=1} {m\choose{j}} \p^{j-1}_t \Sigma_t \p^{m-j}_t \Z + \tfrac{1}{\gamma}\Sigma \sum^m_{j=1} {m\choose{j}} \p^{j-1}_t \Sigma_t \p^{m-j}_t \K \\
&\quad + \tfrac{1}{2\gamma} \sum_{\substack{j_1 + j_2 + j_3 = m \\ j_1, j_2 \geq 1}} {m\choose{j_1 j_2 j_3}} \p^{j_1-t}_t \Sigma_t \p^{j_2-1}_t \Sigma_t \p^{j_3}_t \K.
\end{align*}
Therefore,
\begin{align*}
\frac{(m+1)^2 \|\Sigma^{-m\delta} \p^m_t \Sigma_t \|_{L^\infty_x}}{m! C^m_t \alpha \eps} 
& \leq 3A_z + \tfrac{9}{2\gamma} A_k 
 +   \tfrac{3^\delta 4\alpha A_z \eps}{C_t}  (   A_z + \tfrac{3}{\gamma} A_k) \sum^m_{j=1} \tfrac{(m+1)^2}{j^3 (m+1-j)^2} \\
&\quad + \tfrac{1}{2\gamma} \big(\tfrac{3^\delta 4\alpha A_z \eps}{C_t}\big)^2 A_k \sum_{\substack{j_1 + j_2 + j_3 = m \\ j_1, j_2 \geq 1}} \tfrac{(m+1)^2}{j_1^3 j^3_2 (j_3+1)^2} \\
& \leq 3A_z + \tfrac{9}{2\gamma} A_k 
+   25\tfrac{3^\delta 4\alpha A_z  \eps}{C_t}  (   A_z + \tfrac{3}{\gamma} A_k)  
+ 66\tfrac{1}{2\gamma} \big(\tfrac{3^\delta 4\alpha A_z\eps}{C_t}\big)^2 A_k \\
& \leq 4A_z.
\end{align*}

Taking $\p^m_t$ of \eqref{id:eta_xt} gives us
\begin{align*}
\p^m_t \eta_{xt} 
& = \tfrac{1-\alpha}{2} \eta_x \p^m_t \Z + \tfrac{\alpha}{2\gamma} \eta_x \Sigma \p^m_t \K \\
&\quad + \tfrac{1+\alpha}{2} \p^{m-1}_t (\eta_x \W)_t + \tfrac{1-\alpha}{2} m \eta_{xt} \p^{m-1}_t \Z + \tfrac{\alpha}{2\gamma} m \eta_{xt}  \Sigma \p^{m-1}_t \K \\
&\quad + \tfrac{1-\alpha}{2} \sum^m_{j=2} {m\choose{j}} \p^{j-1}_t \eta_{xt} \p^{m-j}_t \Z + \tfrac{\alpha}{2\gamma} \sum^m_{j=2} {m\choose{j}} \p^{j-1}_t \eta_{xt} \Sigma \p^{m-j}_t \K \\
&\quad +\tfrac{\alpha}{2\gamma} \sum^m_{j=1} {m\choose{j}} \eta_x \p^{j-1}_t \Sigma_t \p^{m-j}_t \K 
 + \tfrac{\alpha}{2\gamma} \sum_{\substack{j_1 + j_2 + j_3 = m \\ j_1, j_2 \geq 1}} {m\choose{j_1 j_2 j_3}} \p^{j_1-t}_t \eta_{xt} \p^{j_2-1}_t \Sigma_t \p^{j_3}_t \K.
\end{align*}
Therefore,
\begin{align*}
\frac{(m+1)^2 \|\Sigma^{-m\delta} \p^m_t \eta_{xt}\|_{L^\infty_x}}{m! C^m_t \eps} 
& \leq \tfrac{3|1-\alpha|}{2}A_z + \tfrac{9 \alpha}{2\gamma} A_k
+ \tfrac{(1+\alpha)3^\delta}{C_t} (\tfrac{m+1}{m})^2\bigg[ \tfrac{1}{2m} \tfrac{5}{8} A_k + \tfrac{|1-\alpha|}{2} A_z + \tfrac{3\alpha}{2\gamma} A_k \bigg] \\
&\quad + \tfrac{3^\delta \eps}{C_t}\bigg[ \tfrac{3|1-\alpha|}{2} (1+\alpha)A_z^2 + \tfrac{9\alpha}{2\gamma} (1+\alpha)A_k A_z + 6\tfrac{\alpha^2}{\gamma} A_k A_z \bigg] \sum^m_{j=1} \tfrac{(m+1)^2}{j^3 (m+1-j)^2} \\
&\quad + \tfrac{3}{2\gamma} \big( \tfrac{3^\delta\alpha \eps}{ C_t}\big)^2 (1+\alpha)A_kA^2_z \sum_{\substack{j_1 + j_2 + j_3 = m \\ j_1, j_2 \geq 1}} \tfrac{(m+1)^2}{j_1^3 j^3_2 (j_3+1)^2}.
\end{align*}
Applying our hypotheses on $C_t$ and $\eps$, as well as our bounds from  \S~\ref{sec:appendix:freqcomp}, we obtain our inequality for $\p^m_t \eta_{xt}$.

Lastly, taking $\p^m_t$ of \eqref{id:W_t} gives us
\begin{align*}
\tfrac{4\gamma}{\alpha} \p^m_t(\eta_x \W)_t 
& = \Sigma \eta_x \W \p^m_t \K 
 + \eta_x \Sigma \sum^m_{j=0} {m\choose{j}} \p^j_t \K \p^{m-j}_t \Z \\
&\quad + \sum^m_{j=1} {m\choose{j}} (\Sigma \p^{j-1}_t (\eta_x \W)_t \p^{m-j}_t \K + \p^{j-1}_t \Sigma_t \eta_x \W \p^{m-j}_t \K) \\
&\quad + \sum_{\substack{j_1 + j_2 +j_3 = m \\ j_1 \geq 1}} {m\choose{j_1 j_2 j_3}} (\eta_x \p^{j_1-1}_t \Sigma_t \p^{j_2}_t \K \p^{j_3}_t \Z + \p^{j_1-1}_t \eta_{xt} \Sigma \p^{j_2}_t \K \p^{j_3}_t \Z) \\
&\quad + \sum_{\substack{j_1+j_2+j_3=m \\ j_1, j_3 \geq 1}} {m\choose{j_1 j_2 j_3}} \p^{j_1-1}_t \Sigma_t \p^{j_2}_t \K \p^{j_3-1}_t (\eta_x \W)_t \\
&\quad + \sum_{\substack{j_1+j_2+j_3+j_4=m \\ j_1, j_2 \geq 1}} {m\choose{j_1 j_2 j_3 j_4}} \p^{j_1-1}_t \eta_{xt} \p^{j_2-1}_t \Sigma_t \p^{j_3}_t \K \p^{j_4}_t \Z.
\end{align*}
The same types of bounds as before give us
\begin{align*}
&\tfrac{4\gamma}{\alpha} \frac{(m+1)^2 \|\Sigma^{-\delta m} \p^m_t (\eta_x \W)_t \|_{L^\infty_x}}{m! C^m_t A_k \eps} \\
&\leq  4 
+ 9 A_z \eps \sum^m_{j=0} \tfrac{(m+1)^2}{(j+1)^2(m+1-j)^2}
+ \tfrac{3^\delta \eps}{C_t}(\tfrac{15 }{8}A_k + \tfrac{16}{3}\alpha A_z) \sum^m_{j=1} \tfrac{(m+1)^2}{j^3(m+1-j)^2} \\
&\quad 
+ 2(1+\alpha)\tfrac{3^\delta \eps}{C_t} A_z \sum^{m-1}_{j=0} \tfrac{(m+1)^2}{(j+1)^2(m-j)^2}  
+ 3(4\alpha  + 3(1+\alpha))\tfrac{3^\delta \eps}{C_t}A^2_z \eps\sum_{\substack{j_1+j_2+j_3=m \\ j_1 \geq 1}} \tfrac{(m+1)^2}{j^3_1(j_2+1)^2(j_3+1)^2} \\
&\quad 
+ \tfrac{5\alpha}{2} A_kA_z \big( \tfrac{3^\delta \eps}{C_t}\big)^2\sum_{\substack{j_1+j_2+j_3=m \\ j_1, j_3 \geq 1}} \tfrac{(m+1)^2}{j^3_1(j_2+1)^2j^3_3}
+ \tfrac{8}{3}\alpha (1+\alpha) A^2_z \big( \tfrac{3^\delta \eps}{C_t}\big)^2 \sum_{\substack{j_2+j_3+j_4 = m-1\\ j_2 \geq 1}} \tfrac{(m+1)^2}{j^3_2 (j_3+1)^2(j_4+1)^2} \\
&\quad + 12\alpha (1+\alpha)A^3_z\eps \big( \tfrac{3^\delta \eps}{C_t}\big)^2 \sum_{\substack{j_1+j_2+j_3+j_4=m \\ j_1, j_2 \geq 1}} \tfrac{(m+1)^2}{j^3_1 j^3_2 (j_3+1)^2 (j_4+1)^2} \\
& \leq 5.
\end{align*}
\end{proof}

\begin{corollary}\label{cor:apriori:est}
Under the same hypotheses as the previous proposition, we have
\begin{subequations}
\begin{align}\label{cor:apriori:W}
\frac{(m+1)^2 \|\Sigma^{-\delta m} \p^m_t \eta_x\|_{L^\infty_x}}{m! C^m_t} & \leq 3,  & 
\frac{(m+1)^2 \|\Sigma^{-\delta m} \p^m_t( \eta_x \W)\|_{L^\infty_x}}{m! C^m_t} & \leq \tfrac{4}{3}, \\
\label{cor:apriori:Sigma_x}
\frac{(m+1)^2 \|\Sigma^{-\delta m} \p^m_t \Sigma_x\|_{L^\infty_x}}{m! C^m_t} & \leq 1,  & 
\frac{(m+1)^2 \|\Sigma^{-\delta m} \p^m_t( \Sigma^{-1})_x\|_{L^\infty_x}}{m! C^m_t} & \leq 10, 
\end{align}
\end{subequations}
for all  $t \in [0,T], \: 0 \leq m \leq 2n+1$ and
\begin{align}\label{cor:apriori:KZ:dx}
\frac{(m+1)^2 \|\Sigma^{-\delta m} \p^{m-1}_t \p_x \K\|_{L^\infty_x}}{m! C^m_t} & \leq \tfrac{10}{\alpha} A_k \eps,  &
\frac{(m+1)^2 \|\Sigma^{-\delta m} \p^{m-1}_t \p_x \Z \|_{L^\infty_x}}{m! C^m_t} & \leq \tfrac{5}{\alpha}A_z\eps , 
\end{align}
for all  $t \in [0,T], \: 1 \leq m \leq 2n+1$.
\end{corollary}

\begin{proof}
The bounds \eqref{cor:apriori:W} follow immediately from \eqref{ineq:0thorder} in the case where $m=0$, and the $1 \leq m \leq 2n+1$ case follows immediately from \eqref{apriori:W_t} and our assumption that $(1+\alpha) 3^\delta \ll C_t$. The bounds on $\p^m_t \Sigma_x$ are proven by taking $\p^m_t$ of \eqref{id:Sigma_x},  applying \eqref{ineq:apriori:est}, and using the bounds from \S~\ref{sec:appendix:freqcomp} along with our assumption that $(1+A_z) \eps \ll 1$. To get our bounds on $\p^m_t (\Sigma^{-1})_x$, take $\p^m_t$ of the equation
\begin{equation}
\label{id:Sigma-1:dx}
(\Sigma^{-1})_x = -\Sigma^{-2} \Sigma_x
\end{equation}
and apply \eqref{ineq:apriori:est} together with our bounds on $\p^m_t \Sigma_x$.

The estimates \eqref{cor:apriori:KZ:dx} can be obtained by taking $\p^{m-1}_t$ of the identities
\begin{subequations}
\begin{align}\label{id:K_x}
\alpha \p_x \K & = \Sigma^{-1} \eta_x \p_t \K + \tfrac{1}{2} \Sigma^{-1} \K \eta_x \W + \tfrac{1}{2} \Sigma^{-1} \eta_x \K \Z, \\
\label{id:Z_x}
\alpha \p_x \Z & = \tfrac{1}{2} \Sigma^{-1} \eta_x \p_t \Z + \tfrac{1-\alpha}{4} \Sigma^{-1}(\eta_x \W) \Z + \tfrac{\alpha}{8\gamma} \K \eta_x \W + \tfrac{1+\alpha}{4} \Sigma^{-1}\eta_x \Z^2 -\tfrac{\alpha}{8\gamma} \eta_x \K \Z.
\end{align}
\end{subequations}
and then using \eqref{ineq:apriori:est}, our inequalities from \S~\ref{sec:appendix:freqcomp}, and our assumption that $(1+A_z)\eps \ll 1$.
\end{proof}

\subsection{Energy estimates for time derivatives}
\label{sec:EE:dt}

Pick constants $\delta, \kappa$ with
\begin{subequations}
\label{def:delta}
\begin{align}
\delta  & \geq \tfrac{6}{\min(1,\alpha)}, \\
\kappa & :=  \max(1,\alpha)(2+ 5\delta)
\end{align}
\end{subequations}
and define the function $C_t : [0,T_*) \rightarrow \RR^+$,
\begin{equation}\label{def:C_t}
C_t(t) : = C_t(0) e^{\kappa t}
\end{equation}
where $C_t(0)$ is some positive constant to be determined.

\begin{proposition}\label{prop:EE:infty}
Suppose $C_t(0)$ satisfies
\begin{align*}
C_t(0) & \gg (1+\alpha) 3^\delta, & 
C_t(0) & \gg \max(1,\alpha) 2^\delta C_0,
\end{align*}
and $\eps$ satisfies
\begin{equation*}
 (1+\alpha)^2(1+B_z)^2 \eps \ll 1.
\end{equation*}
Then
\begin{subequations}\label{ineq:prop:EE}
\begin{align}
\label{ineq:prop:EE:K_t}
\frac{(m+1)^2 \|\Sigma^{-\delta m} \p^m_t \K \|_{L^\infty_x}}{m! C^m_t} & \leq B_k \eps, \\
\label{ineq:prop:EE:Z_t}
\frac{(m+1)^2 \|\Sigma^{-\delta m} \p^m_t \Z \|_{L^\infty_x}}{m! C^m_t} & \leq B_z \eps, 
\end{align}
\end{subequations}
for all $t \in [0,T_*)$ and $0 \leq m \leq 2n+1$.
\end{proposition}

\begin{proof}
These estimates will be proven via $L^q$ energy estimates similar to those performed in  the proof of Proposition \ref{prop:0thorder}. Fix $T \in [0, T_*)$ and pick a constants $A_k, A_z$ satisfying
\begin{equation*}
B_k < A_k < 2B_k < B_z <  A_z < 2B_z.
\end{equation*}
Our bootstrap hypothesis will be that 
\begin{align*}
\frac{(m+1)^2 \|\Sigma^{-\delta m} \p^m_t \K\|_{L^\infty_x}}{m! C^m_t} & \leq A_k \eps, &
\frac{(m+1)^2 \|\Sigma^{-\delta m} \p^m_t \Z\|_{L^\infty_x}}{m! C^m_t} & \leq A_z \eps,
\end{align*}
for all  $t \in [0,T], \: 0 \leq m \leq 2n+1.$ Our assumptions on $C_t(0)$ paired with the fact that $A_z < 2 B_z$ and $(1+B_z) \eps \ll 1$ imply that the hypotheses of Proposition \ref{prop:apriori:est} and Corollary \ref{cor:apriori:est} are satisfied with our choice of $\delta$, $T$, $A_k$, and $A_z$. Additionally, if we define $\C_t : = \frac{1}{2}2^{-\delta} C_t(0)$, then our hypotheses on $C_t(0)$ imply that $\C_t \gg \alpha C_0$ and $\C_t \gg 1+ \alpha$ so that Proposition \ref{prop:t=0} gives us
\begin{equation}\label{ineq:bsh:t=0}
\begin{aligned}
\frac{(m+1)^2 \|\Sigma^{-\delta m} \p^m_t \K\|_{L^\infty_x}}{m! C^m_t} & \leq \tfrac{\eps}{2^m} , &
\frac{(m+1)^2 \|\Sigma^{-\delta m} \p^m_t \Z\|_{L^\infty_x}}{m! C^m_t} & \leq \tfrac{\eps}{2^{m-1}} ,
\end{aligned}
\end{equation}
at time $t = 0$. Therefore, our bootstrap hypothesis is true for $T = 0$.

For $1< q < \infty$ and $1 \leq m \leq 2n+1$ define the energies
\begin{align*}
\E^{m,q}_k(t) & : = \int_\TT \Sigma^{-\delta m q} \eta_x |\p^m_t \K|^q \: dx, &
\E^{m,q}_z(t) & : = \int_\TT \Sigma^{-\delta m q} \eta_x |\p^m_t \Z|^q \: dx, \\
E^{m,q}_k(t) & : = \frac{(m+1)^{2q} \E^{m,q}_k(t)}{(m!)^q C_t(t)^{mq}}, &
E^{m,q}_z(t) & : = \frac{(m+1)^{2q} \E^{m,q}_z(t)}{(m!)^q C_t(t)^{mq}}.
\end{align*}
Taking $\p_t$ of $\E^{m,q}_k$ yields
\begin{align*}
\dot{\E}^{m,q}_k 
& = -\delta m q \int \Sigma^{-\delta m q} \tfrac{\Sigma_t}{\Sigma} \eta_x |\p^m_t \K|^q 
+ \int \Sigma^{-\delta m q} \eta_{xt} | \p^m_t \K|^q \\
&\quad + \int \Sigma^{-\delta m q} |\p^m_t \K|^{q-1} \sgn(\p^m_t \K) q \eta_x \p^{m+1}_t \K \\
& = \int \Sigma^{-\delta m q} \eta_x |\p^m_t \K|^q [ ( \alpha \delta m q + \tfrac{1-\alpha}{2}) \Z + (-\alpha \delta m q + \alpha) \tfrac{1}{2\gamma} \Sigma \K ] \\
&\quad + \tfrac{1+\alpha}{2} \int \Sigma^{-\delta m q} |\p^m_t \K|^q \eta_x \W
+ \int \Sigma^{-\delta m q} |\p^m_t \K|^{q-1} \sgn(\p^m_t \K) q \eta_x \p^{m+1}_t \K. 
\end{align*}

Taking $\p^m_t$ of \eqref{id:K_t} produces the identity
\begin{align} 
\eta_x \p^{m+1}_t \K 
& = \alpha \Sigma \p^m_t \p_x \K - [ \tfrac{1}{2} \eta_x \W + \tfrac{1}{2} \eta_x \Z + m \eta_{xt}] \p^m_t \K
\notag \\
&\quad - \sum^m_{j=2} {m\choose{j}} \p^{j-1}_t \eta_{xt} \p^{m+1-j}_t \K + \alpha \sum^m_{j=1} {m\choose{j}} \p^{j-1}_t \Sigma_t \p^{m-j}_t \p_x \K 
\notag \\
&\quad - \tfrac{1}{2} \eta_x \sum^m_{j=1} {m\choose{j}} \p^j_t \Z \p^{m-j}_t \K 
  - \tfrac{1}{2}  \sum^m_{j=1} {m\choose{j}} \p^{j-1}_t (\eta_x \W)_t \p^{m-j}_t \K 
  \notag \\
&\quad - \tfrac{1}{2} \sum_{\substack{j_1+j_2+j_3 = m \\ j_1 \geq 1}} {m\choose{j_1 j_2 j_3}} \p^{j_1-1}_t \eta_{xt} \p^{j_2}_t \K \p^{j_3}_t \Z.
\label{id:K_t:n+1}
\end{align}
This means that
\begin{align*}
&\int \Sigma^{-\delta m q} |\p^m_t \K|^{q-1} \sgn(\p^m_t \K) q \eta_x \p^{m+1}_t \K \\
& = \alpha \int \Sigma^{1-\delta m q} \p_x\big( |\p^m_t \K|^q \big)
- q \int \Sigma^{-\delta m q} |\p^m_t \K|^q [ \tfrac{1}{2} \eta_x \W + \tfrac{1}{2} \eta_x \Z + m \eta_{xt}]  \\
&\quad - q\sum^m_{j=2} {m\choose{j}}\int \Sigma^{-\delta m q} |\p^m_t \K|^{q-1} \sgn(\p^m_t \K) \p^{j-1}_t \eta_{xt} \p^{m+1-j}_t \K \\
&\quad + \alpha q \sum^m_{j=1} {m\choose{j}}\int \Sigma^{-\delta m q} |\p^m_t \K|^{q-1} \sgn(\p^m_t \K) \p^{j-1}_t \Sigma_t \p^{m-j}_t \p_x \K \\
&\quad - \tfrac{1}{2}q  \sum^m_{j=1} {m\choose{j}} \int \Sigma^{-\delta m q} |\p^m_t \K|^{q-1} \sgn(\p^m_t \K) \eta_x \p^j_t \Z \p^{m-j}_t \K \\
&\quad - \tfrac{1}{2}q  \sum^m_{j=1} {m\choose{j}} \int \Sigma^{-\delta m q} |\p^m_t \K|^{q-1} \sgn(\p^m_t \K)\p^{j-1}_t (\eta_x \W)_t \p^{m-j}_t \K \\
&\quad - \tfrac{1}{2}q \sum_{\substack{j_1+j_2+j_3 = m \\ j_1 \geq 1}} {m\choose{j_1 j_2 j_3}} \int \Sigma^{-\delta m q} |\p^m_t \K|^{q-1} \sgn(\p^m_t \K) \p^{j_1-1}_t \eta_{xt} \p^{j_2}_t \K \p^{j_3}_t \Z.
\end{align*}
Since
\begin{align*}
\alpha \int \Sigma^{1-\delta m q} \p_x\big( |\p^m_t \K|^q \big) = \tfrac{\alpha}{2}(\delta m q -1) \int \Sigma^{-\delta m q} |\p^m_t \K|^q [\eta_x \W-\eta_x \Z + \tfrac{1}{\gamma} \eta_x \Sigma \K],
\end{align*}
we get
\begin{align*}
\dot{\E}^{m,q}_k 
& = \int \Sigma^{-\delta m q} \eta_x |\p^m_t \K|^q [ ( \alpha \delta m q + \tfrac{1-\alpha}{2}) \Z + (-\alpha \delta m q + \alpha) \tfrac{1}{2\gamma} \Sigma \K ] \\
&\quad + \tfrac{1+\alpha}{2} \int \Sigma^{-\delta m q} |\p^m_t \K|^q \eta_x \W 
+ \tfrac{\alpha}{2}(\delta m q -1) \int \Sigma^{-\delta m q} |\p^m_t \K|^q [\eta_x \W-\eta_x \Z + \tfrac{1}{\gamma} \eta_x \Sigma \K] \\
&\quad  - q \int \Sigma^{-\delta m q} |\p^m_t \K|^q [ \tfrac{1}{2} \eta_x \W + \tfrac{1}{2} \eta_x \Z + m \eta_{xt}]  \\
&\quad  - q\sum^m_{j=2} {m\choose{j}}\int \Sigma^{-\delta m q} |\p^m_t \K|^{q-1} \sgn(\p^m_t \K) \p^{j-1}_t \eta_{xt} \p^{m+1-j}_t \K \\
&\quad + \alpha q \sum^m_{j=1} {m\choose{j}}\int \Sigma^{-\delta m q} |\p^m_t \K|^{q-1} \sgn(\p^m_t \K) \p^{j-1}_t \Sigma_t \p^{m-j}_t \p_x \K \\
&\quad  - \tfrac{1}{2}q  \sum^m_{j=1} {m\choose{j}} \int \Sigma^{-\delta m q} |\p^m_t \K|^{q-1} \sgn(\p^m_t \K) \eta_x \p^j_t \Z \p^{m-j}_t \K \\
&\quad - \tfrac{1}{2}q  \sum^m_{j=1} {m\choose{j}} \int \Sigma^{-\delta m q} |\p^m_t \K|^{q-1} \sgn(\p^m_t \K)\p^{j-1}_t (\eta_x \W)_t \p^{m-j}_t \K \\
&\quad  - \tfrac{1}{2}q \sum_{\substack{j_1+j_2+j_3 = m \\ j_1 \geq 1}} {m\choose{j_1 j_2 j_3}} \int \Sigma^{-\delta m q} |\p^m_t \K|^{q-1} \sgn(\p^m_t \K) \p^{j_1-1}_t \eta_{xt} \p^{j_2}_t \K \p^{j_3}_t \Z.
\end{align*}
Simplifying the first four terms on the righthand side of this equation gives us
\begin{align}
\dot{\E}^{m,q}_k 
& = D_k \int \Sigma^{-\delta m q} |\p^m_t \K|^q \eta_x \W 
+ (D_k + \alpha m q)\int \Sigma^{-\delta m q} \eta_x |\p^m_t \K|^q \Z 
\notag \\
&\quad - \tfrac{\alpha}{2\gamma} mq \int \Sigma^{-\delta m q} \eta_x |\p^m_t \K|^q \Sigma \K 
\notag \\
&\quad  - q\sum^m_{j=2} {m\choose{j}}\int \Sigma^{-\delta m q} |\p^m_t \K|^{q-1} \sgn(\p^m_t \K) \p^{j-1}_t \eta_{xt} \p^{m+1-j}_t \K 
\notag \\
&\quad + \alpha q \sum^m_{j=1} {m\choose{j}}\int \Sigma^{-\delta m q} |\p^m_t \K|^{q-1} \sgn(\p^m_t \K) \p^{j-1}_t \Sigma_t \p^{m-j}_t \p_x \K 
\notag \\
&\quad  - \tfrac{1}{2}q  \sum^m_{j=1} {m\choose{j}} \int \Sigma^{-\delta m q} |\p^m_t \K|^{q-1} \sgn(\p^m_t \K) \eta_x \p^j_t \Z \p^{m-j}_t \K 
\notag \\
&\quad - \tfrac{1}{2}q  \sum^m_{j=1} {m\choose{j}} \int \Sigma^{-\delta m q} |\p^m_t \K|^{q-1} \sgn(\p^m_t \K)\p^{j-1}_t (\eta_x \W)_t \p^{m-j}_t \K 
\notag \\
&\quad  - \tfrac{1}{2}q \sum_{\substack{j_1+j_2+j_3 = m \\ j_1 \geq 1}} {m\choose{j_1 j_2 j_3}} \int \Sigma^{-\delta m q} |\p^m_t \K|^{q-1} \sgn(\p^m_t \K) \p^{j_1-1}_t \eta_{xt} \p^{j_2}_t \K \p^{j_3}_t \Z,
\label{id:energy:K:dt}
\end{align}
where
\begin{equation}\label{def:D_k}
D_k  = D_k(m,q,\alpha) : = \tfrac{1}{2} + q[ (\tfrac{\alpha}{2}\delta - \tfrac{1+\alpha}{2}) m - \tfrac{1}{2}]
.
\end{equation}

Our lower bound \eqref{def:delta} and the fact that $q>1$ gives us the bounds
\begin{equation}\label{ineq:D_k}
\tfrac{1}{2} +\tfrac{3}{2}mq \max(1,\alpha) \leq D_k < \tfrac{1}{2} mq \alpha (\delta-1).
\end{equation}
Since $D_k > 0$, the inequality \eqref{ineq:W} gives us
\begin{align}  
&\dot{\E}^{m,q}_k + \tfrac{1}{2}D_k \int \Sigma^{-\delta m q} |\p^m_t \K|^q 
\\
& \leq 4 D_k \E^{m,q}_k 
+ (B_zD_k + \alpha m q(B_z +\tfrac{3}{2\gamma} B_k) )\eps \E^{m,q}_k  
\notag \\
&\quad - q\sum^m_{j=2} {m\choose{j}}\int \Sigma^{-\delta m q} |\p^m_t \K|^{q-1} \sgn(\p^m_t \K) \p^{j-1}_t \eta_{xt} \p^{m+1-j}_t \K 
\notag \\
&\quad + \alpha q \sum^m_{j=1} {m\choose{j}}\int \Sigma^{-\delta m q} |\p^m_t \K|^{q-1} \sgn(\p^m_t \K) \p^{j-1}_t \Sigma_t \p^{m-j}_t \p_x \K 
\notag \\
&\quad - \tfrac{1}{2}q  \sum^m_{j=1} {m\choose{j}} \int \Sigma^{-\delta m q} |\p^m_t \K|^{q-1} \sgn(\p^m_t \K) \eta_x \p^j_t \Z \p^{m-j}_t \K 
\notag \\
&\quad - \tfrac{1}{2}q  \sum^m_{j=1} {m\choose{j}} \int \Sigma^{-\delta m q} |\p^m_t \K|^{q-1} \sgn(\p^m_t \K)\p^{j-1}_t (\eta_x \W)_t \p^{m-j}_t \K 
\notag \\
&\quad  - \tfrac{1}{2}q \sum_{\substack{j_1+j_2+j_3 = m \\ j_1 \geq 1}} {m\choose{j_1 j_2 j_3}} \int \Sigma^{-\delta m q} |\p^m_t \K|^{q-1} \sgn(\p^m_t \K) \p^{j_1-1}_t \eta_{xt} \p^{j_2}_t \K \p^{j_3}_t \Z.
\label{ineq:energy:K:dt}
\end{align}
Because
\begin{equation*}
\dot{E}^{m,q}_k = -mq\kappa E^{m,q}_k + \tfrac{(m+1)^{2q} \dot{\E}^{m,q}_k}{(m!)^{q} C^{mq}_t},
\end{equation*}
multiplying \eqref{ineq:energy:K:dt} by $\big(\frac{(m+1)^2}{m! C^m_t}\big)^q$  and using our a priori estimates from \S~\ref{sec:apriori} gives us
\begin{align*}
\dot{E}^{m,q}_k + \tfrac{1}{2} D_k \bigg[ \tfrac{(m+1)^2 \|\Sigma^{-\delta m} \p^m_t \K\|_{L^q_x}}{m! C^m_t} \bigg]^q & \leq [-mq \kappa + (4+B_z\eps)D_k + \alpha m q(B_z +\tfrac{3}{2\gamma} B_k) )\eps ] E^{m,q}_k \\
&\quad + q \bigg[ \tfrac{(m+1)^2 \|\Sigma^{-\delta m} \p^m_t \K\|_{L^q_x}}{m! C^m_t} \bigg]^{q-1} A_kA_z\eps^2 \big( \cdots \big),
\end{align*}
where
\begin{align*}
\big( \cdots \big) 
& \leq (2|1-\alpha| + \tfrac{5\alpha}{\gamma})\sum^m_{j=2} \tfrac{(m+1)^2(m+1-j)}{j^3(m+2-j)^2}
+ 50\alpha(1+ \tfrac{1}{\gamma})  \sum^m_{j=1} \tfrac{(m+1)^2(m+1-j)}{j^3(m+2-j)^2} \\
&\quad  + \tfrac{3}{2}  \sum^m_{j=1} \tfrac{(m+1)^2}{(j+1)^2(m+1-j)^2} 
+ \tfrac{5\alpha}{8\gamma} \tfrac{3^\delta}{C_t}  \sum^m_{j=1} \tfrac{(m+1)^2}{j^3(m+1-j)^2} \\
&\quad  + \tfrac{1}{3}\tfrac{(1+\alpha)3^\delta}{C_t} \sum^{m-1}_{j=0} \tfrac{m+1}{(j+1)^2(m-j)^2} 
+  (|1-\alpha| + \tfrac{5\alpha}{2\gamma}) \tfrac{3^\delta A_z \eps}{C_t} \sum_{\substack{j_1+j_2+j_3 = m \\ j_1 \geq 1}} \tfrac{(m+1)^2}{j^3_1 (j_2+1)^2 (j_3+1)^2}.
\end{align*}
Since $(1+\alpha)3^\delta \ll C_t(0)$, $A_z\eps < 2B_z \eps \ll 1$, and
\begin{equation}\label{ineq:sum:10m}
\sum^m_{j=1} \tfrac{(m+1)^2(m+1-j)}{j^3(m+2-j)^2}  
= m + \sum^m_{j=2} \tfrac{(m+1)^2(m+1-j)}{j^3(m+2-j)^2} 
\leq m + m(1+\tfrac{1}{m})\sum^m_{j=2} \tfrac{m+1}{j^3(m+2-j)}  \leq m + 10 m,
\end{equation}
it follows from our bounds in \S~\ref{sec:appendix:freqcomp} that 
\begin{equation*}
\big( \cdots \big) \leq (1+\alpha)m\OO(1) + \OO(1).
\end{equation*}
Therefore,
\begin{align}
&\dot{E}^{m,q}_k + \tfrac{1}{2} D_k \bigg[ \tfrac{(m+1)^2 \|\Sigma^{-\delta m} \p^m_t \K\|_{L^q_x}}{m! C^m_t} \bigg]^q 
\notag\\
& \leq [-mq \kappa + (4+B_z\eps)D_k + \alpha m q(B_z +\tfrac{3}{2\gamma} B_k) )\eps ] E^{m,q}_k 
\notag \\
&\quad + q m\bigg[ \tfrac{(m+1)^2 \|\Sigma^{-\delta m} \p^m_t \K\|_{L^q_x}}{m! C^m_t} \bigg]^{q-1}\OO( (1+\alpha)A_kA_z\eps^2) 
+ q \bigg[ \tfrac{(m+1)^2 \|\Sigma^{-\delta m} \p^m_t \K\|_{L^q_x}}{m! C^m_t} \bigg]^{q-1} \OO(A_k A_z\eps^2)
\notag \\
&\leq [-mq \kappa + (4+B_z\eps)D_k + \alpha m q(B_z +\tfrac{3}{2\gamma} B_k) )\eps ] E^{m,q}_k 
+ m\OO((1+\alpha)A_k A_z \eps^{3/2})^q 
\notag \\
&\quad
+ (m+1)(q-1) \eps^{\frac{q}{2(q-1)}}\bigg[ \tfrac{(m+1)^2 \|\Sigma^{-\delta m} \p^m_t \K\|_{L^q_x}}{m! C^m_t} \bigg]^q
 +\OO(A_kA_z \eps^{3/2})^q.
\label{ineq:energy:K:n}
\end{align}
Since $\eps < 1$ and $\frac{q}{q-1} > 1$, this can be rewritten as
\begin{align*}
&\dot{E}^{m,q}_k + \bigg[\tfrac{1}{2} D_k -(m+1)(q-1) \eps^{\frac{1}{2}}\bigg]  \bigg[ \tfrac{(m+1)^2 \|\Sigma^{-\delta m} \p^m_t \K\|_{L^q_x}}{m! C^m_t} \bigg]^q \\
& \leq [-mq \kappa + (4+B_z\eps)D_k + \alpha m q(B_z +\tfrac{3}{2\gamma} B_k) )\eps ] E^{m,q}_k 
+ m\OO((1+\alpha)A_k A_z \eps^{3/2})^q +\OO(A_kA_z \eps^{3/2})^q .
\end{align*}
It now follows from the lower and upper bounds in \eqref{ineq:D_k} that
\begin{equation}\label{ineq:energy:ODE:K}
\dot{E}^{m,q}_k   \leq -mq \max(1,\alpha) E^{m,q}_k  + m\OO((1+\alpha)A_k A_z \eps^{3/2})^q +\OO(A_kA_z \eps^{3/2})^q . 
\end{equation}
This implies (see Lemma \ref{lem:ODE}) that for all $t \in [0,T]$ we have
\begin{align*}
E^{m,q}_k(t)  & \leq \max\bigg(E^{m,q}_k(0),  \frac{\OO( (1+\alpha)A_kA_z \eps^{3/2})^q}{q \max(1,\alpha)} + \frac{\OO(A_kA_z \eps^{3/2})^q}{mq\max(1,\alpha)}\bigg), \\
\implies E^{m,q}_k(t)^{1/q}  &\leq \max\bigg( E^{m,q}_k(0)^{1/q},  \OO((1+\alpha)^{1-\frac{1}{q}}A_kA_z \eps^{3/2}) + \OO(A_kA_z \eps^{3/2}) \bigg).
\end{align*}
 Therefore, \eqref{ineq:bsh:t=0} and the fact that $A_k < 2B_k, A_z < 2B_z,$ now gives us
\begin{align*}
E^{m,q}_k(t)^{1/q}  & \leq \max\big(\tfrac{1}{2^m B_k} , \OO(B_z(1+\alpha)\eps^{\frac{1}{2}})\big) B_k\eps & \forall \: t \in [0,T].
\end{align*}
Sending $q \rightarrow \infty$ now yields
\begin{align*}
\frac{(m+1)^2 \|\Sigma^{-\delta m} \p^m_t \K \|_{L^\infty_x}}{m! C^m_t} 
& \leq \max\big(\tfrac{1}{2^m B_k} , \OO(B_z(1+\alpha)\eps^{\frac{1}{2}})\big) B_k\eps, & \forall \: t \in [0,T].
\end{align*}
Since $(1+\alpha)^2B^2_z \eps \ll 1$, $m \geq 1$, and $B_k > 1$, \eqref{ineq:prop:EE:K_t} now follows.

To obtain the bound \eqref{ineq:prop:EE:Z_t} for $t \in [0,T]$ and thus conclude our bootstrap argument, it is best to bound the sum $E^{m,q}_z + E^{m,q}_k$ instead of dealing with $E^{m,q}_z$ on its own. The proof for bounding $E^{m,q}_z + E^{m,q}_k$, however, is essentially the same as the computations that were just carried out to bound $E^{m,k}_q$.

Taking $\p^m_t$ of \eqref{id:Z_t} yields the identity
\begin{align}
\eta_x \p^{m+1}_t \Z 
& = 2\alpha \Sigma \p^m_t \p_x \Z 
- [ \tfrac{1-\alpha}{2} \eta_x \W + (1+\alpha) \eta_x \Z + m \eta_{xt}-\tfrac{\alpha}{4\gamma} \eta_x \Sigma \K ] \p^m_t \Z 
- \tfrac{\alpha}{4\gamma} \Sigma \p^m_t \K \eta_x \W 
\notag \\
&\quad - \sum^m_{j=2} {m\choose{j}} \p^{j-1}_t \eta_{xt} \p^{m+1-j}_t \Z + 2\alpha \sum^m_{j=1} {m\choose{j}} \p^{j-1}_t \Sigma_t \p^{m-j}_t \p_x \Z 
\notag \\
&\quad -\tfrac{1+\alpha}{2} \eta_x \sum^{m-1}_{j=1}{m\choose{j}} \p^j_t \Z \p^{m-j}_t \Z + \tfrac{\alpha}{4\gamma} \eta_x \Sigma \sum^m_{j=1} {m\choose{j}} \p^j_t \K \p^{m-j}_t \Z 
\notag \\
&\quad - \sum^m_{j=1} {m\choose{j}}\bigg[ \tfrac{1-\alpha}{2} \p^{j-1}_t (\eta_x \W)_t \p^{m-j}_t \Z + \tfrac{\alpha}{4\gamma} \Sigma \p^{j-1}_t (\eta_x \W)_t \p^{m-j}_t \K + \tfrac{\alpha}{4\gamma} \p^{j-1}_t \Sigma_t \eta_x \W \p^{m-j}_t \K\bigg] 
\notag \\
&\quad + \sum_{\substack{j_1+j_2+j_3 = m \\ j_1 \geq 1}} {m\choose{j_1 j_2 j_3}} \p^{j_1-1}_t \eta_{xt} \big( -\tfrac{1+\alpha}{2} \p^{j_2}_t \Z + \tfrac{\alpha}{4\gamma} \Sigma \p^{j_2}_t \K\big) \p^{j_3}_t \Z 
\notag \\
&\quad + \tfrac{\alpha}{4\gamma} \eta_x \sum_{\substack{j_1+j_2+j_3 = m \\ j_1 \geq 1}} {m\choose{j_1 j_2 j_3}} \p^{j_1-1}_t \Sigma_t  \p^{j_2}_t \K \p^{j_3}_t \Z 
\notag \\
&\quad - \tfrac{\alpha}{4\gamma}  \sum_{\substack{j_1+j_2+j_3 = m \\ j_1, j_3 \geq 1}} {m\choose{j_1 j_2 j_3}} \p^{j_1-1}_t \Sigma_t  \p^{j_2}_t \K \p^{j_3-1}_t (\eta_x \W)_t 
\notag \\
&\quad + \tfrac{\alpha}{4\gamma}  \sum_{\substack{j_1+j_2+j_3 +j_4= m \\ j_1, j_2 \geq 1}} {m\choose{j_1 j_2 j_3 j_4}} \p^{j_1-1}_t \eta_{xt} \p^{j_2-1}_t \Sigma_t \p^{j_3}_t \K \p^{j_4}_t \Z.
\label{id:Z_t:n+1}
\end{align}

Computations analogous to those performed above for $E^{m,q}_k$ give us the inequality 
\begin{align}
&\dot{E}^{m,q}_z + \tfrac{1}{2} D_z \bigg[ \tfrac{(m+1)^2 \|\Sigma^{-\delta m} \p^m_t \Z\|_{L^q_x}}{m! C^m_t} \bigg]^q 
\notag\\
& \leq [-mq\kappa + 4D_z + \OO( (1+\alpha)(mq+1)B_z\eps) ] E^{m,q}_z 
\notag\\
&\quad + \tfrac{\alpha}{\gamma} q \bigg[ \tfrac{(m+1)^2 \|\Sigma^{-\delta m} \p^m_t \Z\|_{L^q_x}}{m! C^m_t} \bigg]^{q-1} \bigg[ \tfrac{(m+1)^2 \|\Sigma^{-\delta m} \p^m_t \K\|_{L^q_x}}{m! C^m_t} \bigg]
\notag \\
&\quad + qm \bigg[ \tfrac{(m+1)^2 \|\Sigma^{-\delta m} \p^m_t \Z\|_{L^q_x}}{m! C^m_t} \bigg]^{q-1} \OO((1+\alpha)A^2_z\eps^2) 
+ q \bigg[ \tfrac{(m+1)^2 \|\Sigma^{-\delta m} \p^m_t \Z\|_{L^q_x}}{m! C^m_t} \bigg]^{q-1} \OO(A_z^2\eps^2),
\label{ineq:energy:Z:n}
\end{align}
where $D_z$ is defined as
\begin{equation}\label{def:D_z}
D_z : = \tfrac{1-\alpha}{2} + q[ (\alpha\delta - \tfrac{1+\alpha}{2}) m - \tfrac{1-\alpha}{2}].
\end{equation}
The constant $D_z$ satisfies the bounds
\begin{equation}\label{ineq:D_z}
-\tfrac{1}{2}(q-1) +5mq \max(1,\alpha) \leq D_z < mq \alpha \delta.
\end{equation}
Adding \eqref{ineq:energy:K:n} and \eqref{ineq:energy:Z:n}  together gives us
\begin{align*}
& \dot{E}^{m,q}_k + \dot{E}^{m,q}_z
+ \bigg[ \tfrac{1}{2}D_k - \tfrac{\alpha}{\gamma} - (m+1)(q-1) \eps^{\frac{1}{2}} \bigg] \bigg[ \tfrac{(m+1)^2 \|\Sigma^{-\delta m} \p^m_t \K\|_{L^q_x}}{m! C^n_t} \bigg]^q \\
&\quad + \bigg[ \tfrac{1}{2}D_z -(q-1)(\tfrac{\alpha}{\gamma} + (m+1) \eps^{\frac{1}{2}}) \bigg] \bigg[ \tfrac{(m+1)^2 \|\Sigma^{-\delta m} \p^m_t \Z\|_{L^q_x}}{m! C^m_t} \bigg]^q \\
& \leq [-mq\kappa + 4\max(D_k,D_z) + \OO((1+\alpha)(mq+1)B_z\eps) ] \big( E^{m,q}_k + E^{m,q}_z\big) \\
&\quad  + 2m\OO((1+\alpha)A^2_z \eps^{3/2})^q +2\OO(A^2_z \eps^{3/2})^q.
\end{align*}
Using \eqref{ineq:D_k}, \eqref{ineq:D_z} , and \eqref{def:delta} gives us
\begin{align*}
\dot{E}^{m,q}_k + \dot{E}^{m,q}_z 
& \leq -mq\max(1,\alpha) \big( E^{m,q}_k + E^{m,q}_z\big) + 2m\OO((1+\alpha)A^2_z \eps^{3/2})^q +2\OO(A^2_z \eps^{3/2})^q. 
\end{align*}
From here, carrying out the same ODE comparison argument as before and then sending $q \rightarrow \infty$ gives us
\begin{align*}
\max\bigg( \frac{(m+1)^2 \|\Sigma^{-\delta m} \p^m_t \K \|_{L^\infty_x}}{m! C^m_t}, \frac{(m+1)^2 \|\Sigma^{-\delta m} \p^m_t \Z \|_{L^\infty_x}}{m! C^m_t} \bigg) & \leq B_z \eps.
\end{align*}
\end{proof}

\subsection{Bounds on the remaining derivatives}
\label{sec:ineq:dx}
It follows from the definition \eqref{def:delta} of $\kappa$ that $C_t$ satisfies
\begin{align*}
C_t(t) & \leq C_t(0) e^{\kappa T_*} < C_t(0) e^{2(1+\eps^{\frac{1}{2}})(2+5\delta)}
\end{align*}
for all $t \in [0,T_*)$. Therefore, if we define the constant
\begin{align}\label{def:C_t:bar}
\bar{C}_t & : = C_t(0) e^{(11+\log 3) \delta + 5},
\end{align}
then
\begin{equation}\label{ineq:C_t:bar}
\bar{C}_t  > 3^\delta C_t(t) 
\end{equation}
for all $t \in [0,T_*)$. It follows from Proposition \ref{prop:EE:infty} and \eqref{ineq:0thorder} that
\begin{align}\label{ineq:KZ:dt:new}
\frac{(m+1)^2 \|\p^m_t \K \|_{L^\infty_x}}{m! \bar{C}^m_t} & \leq B_k \eps, &
\frac{(m+1)^2 \|\p^m_t \Z \|_{L^\infty_x}}{m! \bar{C}^m_t} & \leq B_z \eps,
\end{align}
for all $t \in [0, T_*), 0 \leq m \leq 2n+1$.

\begin{proposition}\label{prop:ineq:dx}
Suppose the hypotheses of Proposition \ref{prop:EE:infty} are satisfied, and additionally $\bar{C}_x$ is a constant satisfying
\begin{align*}
\bar{C}_x & \gg 1, &
\alpha\bar{C}_x & \gg \bar{C}_t.
\end{align*}
Then if $(1+\alpha)B_z \eps \ll 1$ the following bounds hold for all $t \in [0,T_*)$ and $|\beta| \leq 2n+1$:
\begin{subequations}
\begin{align}
&
\frac{(|\beta|+1)^2 \inorm{\p^\beta \K}_{L^\infty_x}}{|\beta|! \bar{C}^{\beta_x}_x \bar{C}^{\beta_t}_t}  \leq B_k \eps, 
&&
\frac{(|\beta|+1)^2 \inorm{\p^\beta \Z}_{L^\infty_x}}{|\beta|! \bar{C}^{\beta_x}_x \bar{C}^{\beta_t}_t}  \leq B_z \eps, \label{ineq:KZ:dx}
\\
&
\frac{(|\beta|+1)^2 \inorm{\p^\beta \Sigma_t}_{L^\infty_x}}{|\beta|! \bar{C}^{\beta_x}_x \bar{C}^{\beta_t}_t} \leq 4\alpha  B_z \eps,
&&
\frac{(|\beta|+1)^2 \inorm{\p^\beta (\Sigma^{-1})_t}_{L^\infty_x}}{|\beta|! \bar{C}^{\beta_x}_x \bar{C}^{\beta_t}_t} \leq 4\alpha  B_z \eps,
\label{ineq:Sigma_t:dx}
\\
&
\frac{(|\beta|+1)^2 \inorm{\p^\beta \eta_x}_{L^\infty_x}}{|\beta|! \bar{C}^{\beta_x}_x \bar{C}^{\beta_t}_t} \leq 3, 
\label{ineq:eta_x:dx}
&&
\frac{(|\beta|+1)^2 \inorm{\p^\beta (\eta_x \W)}_{L^\infty_x}}{|\beta|! \bar{C}^{\beta_x}_x \bar{C}^{\beta_t}_t}  \leq  \tfrac{4}{3}, 
\\
&
\frac{(|\beta|+1)^2 \inorm{\p^\beta \Sigma_x}_{L^\infty_x}}{|\beta|! \bar{C}^{\beta_x}_x \bar{C}^{\beta_t}_t} \leq 1, 
&&
\frac{(|\beta|+1)^2 \inorm{\p^\beta (\Sigma^{-1})_x}_{L^\infty_x}}{|\beta|! \bar{C}^{\beta_x}_x \bar{C}^{\beta_t}_t}  \leq 10.
\label{ineq:Sigma_x:dx}
\end{align}
\end{subequations}
\end{proposition}

\begin{proof}
We proceed by induction on $\beta_x$. 

{\it{Base Case:}} When $\beta_x = 0$, \eqref{ineq:KZ:dx} is implied by \eqref{ineq:KZ:dt:new}. The rest of the inequalities \eqref{ineq:Sigma_t:dx}--\eqref{ineq:Sigma_x:dx} in the case $\beta_x = 0$ now follow by applying the a priori estimates of \S~\ref{sec:apriori} with $A_k= B_k, A_z = B_z, \delta = 0$, and $C_t = \bar{C}_t$.

{\it{Inductive Step:}} Fix $m$ with $0 \leq m \leq 2n$ and suppose that the estimates \eqref{ineq:KZ:dx} - \eqref{ineq:Sigma_x:dx} are true for choices of $\beta$ with $|\beta| \leq 2n+1$ and $\beta_x \leq m$.

We first prove that \eqref{ineq:KZ:dx} is true for choices of $\beta$ with $|\beta| \leq 2n+1$ and $\beta_x = m+1$. This needs to be done first, because it will be used to prove that the rest of the bounds \eqref{ineq:Sigma_t:dx} - \eqref{ineq:Sigma_x:dx} hold for $\beta$ with $|\beta| \leq 2n+1$ and $\beta_x = m+1$. 

Fix a choice of $\beta$ with $|\beta| \leq 2n$ and $\beta_x =m$. We will prove that \eqref{ineq:KZ:dx} is true for $\beta + e_x$. To do this, take $\p^\beta$ of \eqref{id:K_x} to obtain
\begin{align*}
\alpha \p^{\beta+e_x} \K & = \Sigma^{-1} \eta_x \p^{\beta+e_t} \K + \sum_{\substack{je_x + \gamma_2 + \gamma_3 = \beta \\ j \geq 1}} { \beta \choose je_x \gamma_2 \gamma_3} \p^{j-1}_x(\Sigma^{-1})_x \p^{\gamma_2} \eta_x \p^{\gamma_3+e_t} \K \\
& \hspace{-15mm} + \sum_{\substack{\gamma_1 + \gamma_2 + \gamma_3 = \beta \\ \gamma_1 \geq e_t}} { \beta \choose \gamma_1 \gamma_2 \gamma_3} \p^{\gamma_1-e_t}(\Sigma^{-1})_t \p^{\gamma_2} \eta_x \p^{\gamma_3+e_t} \K + \tfrac{1}{2} \Sigma^{-1} \sum_{ \gamma \leq \beta} {\beta \choose \gamma} \p^\gamma \K \p^{\beta-\gamma} (\eta_x \W)  \\
& \hspace{-15mm} + \tfrac{1}{2}\Sigma^{-1} \hspace{-5mm} \sum_{\gamma_1+\gamma_2 + \gamma_3 = \beta} {\beta \choose \gamma_1 \gamma_2 \gamma_3} \p^{\gamma_1} \eta_x \p^{\gamma_2} \K \p^{\gamma_3} \Z \\
& \hspace{-15mm}+ \tfrac{1}{2}\hspace{-6mm}\sum_{\substack{je_x + \gamma_2 + \gamma_3 = \beta \\ j \geq 1}} { \beta \choose je_x \gamma_2 \gamma_3} \p^{j-1}_x(\Sigma^{-1})_x \p^{\gamma_2} \K \p^{\gamma_3}(\eta_x \W) \hspace{1mm}+ \hspace{1mm} \tfrac{1}{2}\hspace{-9mm}\sum_{\substack{je_x + \gamma_2 + \gamma_3+\gamma_4 = \beta \\ j \geq 1}} { \beta \choose je_x \gamma_2 \gamma_3 \gamma_4} \p^{j-1}_x(\Sigma^{-1})_x \p^{\gamma_2} \eta_x \p^{\gamma_3} \K \p^{\gamma_4} \Z \\
& \hspace{-15mm}+ \tfrac{1}{2}\hspace{-6mm}\sum_{\substack{\gamma_1 + \gamma_2 + \gamma_3 = \beta \\ \gamma_1 \geq e_t}} { \beta \choose \gamma_1 \gamma_2 \gamma_3} \p^{\gamma_1-e_t} (\Sigma^{-1})_t \p^{\gamma_2} \K \p^{\gamma_3}(\eta_x \W) \hspace{1mm}+ \hspace{1mm} \tfrac{1}{2}\hspace{-9mm}\sum_{\substack{je_x + \gamma_2 + \gamma_3+\gamma_4 = \beta \\ \gamma_1 \geq e_t}} { \beta \choose \gamma_1 \gamma_2 \gamma_3 \gamma_4} \p^{\gamma_1-e_t} (\Sigma^{-1})_t \p^{\gamma_2} \eta_x \p^{\gamma_3} \K \p^{\gamma_4} \Z. 
\end{align*}
Applying our inductive hypothesis and Lemma \ref{lem:id:multi} gives us 
\begin{align*}
\alpha\frac{(|\beta|+e_x|+1)^2 \|\p^\beta \K\|_{L^\infty_x}}{|\beta +e_x|! \bar{C}^{\beta_x+1}_x \bar{C}^{\beta_t}_t B_k \eps} & \leq 9\tfrac{\bar{C}_t}{\bar{C}_x} + 30 \tfrac{\bar{C}_t}{\bar{C}^2_x} \tfrac{|\beta|+2}{|\beta|+1}\hspace{-6mm}\sum_{\substack{j_1+j_2+j_3 = |\beta| \\ j_1 \geq 1}} \tfrac{|\beta|+2}{j^3_1 (j_2+1)^2(j_3+1)^2}   \\
& \hspace{-35mm}+ \tfrac{1}{\bar{C}_x} \tfrac{|\beta|+2}{|\beta|+1}\bigg[ 24\alpha B_z \eps \hspace{-6mm}\sum_{\substack{j_1+j_2+j_3 = |\beta| \\ j_1 \geq 1}} \tfrac{|\beta|+2}{j^3_1 (j_2+1)^2(j_3+2)} + 2\sum^{|\beta|}_{j=0} \tfrac{|\beta|+2}{(j+1)^2(|\beta|+1-j)^2} + \tfrac{9}{2} B_z \eps \hspace{-5mm} \sum_{j_1+j_2+j_3=|\beta|} \tfrac{|\beta|+2}{(j_1+1)^2(j_2+1)^2(j_3+1)^2} \bigg] \\
& \hspace{-35mm}+ \tfrac{1}{\bar{C}^2_x} \tfrac{|\beta|+2}{|\beta|+1}\bigg[ \hspace{2mm} \tfrac{20}{3} \hspace{-5mm} \sum_{\substack{j_1+j_2+j_3 = |\beta| \\ j_1 \geq 1}} \tfrac{|\beta|+2}{j^3_1 (j_2+1)^2(j_3+1)^2} + 15B_z \eps \hspace{-8mm} \sum_{\substack{j_1+j_2+j_3 +j_4 = |\beta| \\ j_1 \geq 1}} \tfrac{|\beta|+2}{j^3_1 (j_2+1)^2(j_3+1)^2(j_4+1)^2}  \bigg] \\
& \hspace{-35mm}+ \tfrac{4\alpha B_z \eps}{\bar{C}_x \bar{C}_t} \tfrac{|\beta|+2}{|\beta|+1}\bigg[ \hspace{2mm} \tfrac{4}{3} \hspace{-5mm} \sum_{\substack{j_1+j_2+j_3 = |\beta| \\ j_1 \geq 1}} \tfrac{|\beta|+2}{j^3_1 (j_2+1)^2(j_3+1)^2} + 3B_z \eps \hspace{-8mm} \sum_{\substack{j_1+j_2+j_3 +j_4 = |\beta| \\ j_1 \geq 1}} \tfrac{|\beta|+2}{j^3_1 (j_2+1)^2(j_3+1)^2(j_4+1)^2}  \bigg] .
\end{align*}
The computations from \S~\ref{sec:appendix:freqcomp} now imply that
\begin{equation*}
\frac{(|\beta|+e_x|+1)^2 \|\p^\beta \K\|_{L^\infty_x}}{|\beta +e_x|! \bar{C}^{\beta_x+1}_x \bar{C}^{\beta_t}_t B_k \eps} \lesssim \tfrac{1}{\alpha \bar{C}_x} [\bar{C}_t + \tfrac{\bar{C}_t}{\bar{C}_x} + 1 + (1+\alpha)B_z \eps + \tfrac{1+B_z\eps}{\bar{C}_x} + \tfrac{\alpha B_z \eps(1+B_z \eps)}{\bar{C}_t}].
\end{equation*}
If $\alpha \bar{C}_x \gg \bar{C}_t$ and $\bar{C}_x \gg 1$,  then the righthand side is less than 1, which proves \eqref{ineq:KZ:dx} for $\p^{\beta+e_x} \K$. The bound on $\p^{\beta+e_x} \Z$ can be proven in an identical manner, so the inductive step for \eqref{ineq:KZ:dx} is complete.

We next prove \eqref{ineq:Sigma_t:dx} is true for $\beta$ with $|\beta| \leq 2n+1$ and $\beta_x = m+1$. Pick  $\beta$ with $|\beta| \leq 2n+1$ and $\beta_x = m+1$, and take $\p^\beta$ of \eqref{id:Sigma_t}. Using the the inductive hypothesis and the fact that \eqref{ineq:KZ:dx} has been proven already for multi-indices with $\beta_x \leq m+1$, our hypotheses on $\bar{C}_t, \bar{C}_x$ and $\eps$ imply that 
\begin{align*}
\frac{(|\beta|+1)^2 \inorm{\p^\beta \Sigma_t}_{L^\infty_x}}{|\beta|! \bar{C}^{\beta_x}_x \bar{C}^{\beta_t}_t \alpha B_z \eps} & \leq 3 + \tfrac{9}{2\gamma} \tfrac{B_k}{B_z} + \OO(\tfrac{\alpha B_z \eps}{\bar{C}_t} + \tfrac{1}{\bar{C}_x}) 
< 4, \\
\frac{(|\beta|+1)^2 \inorm{\p^\beta (\Sigma^{-1})_t}_{L^\infty_x}}{|\beta|! \bar{C}^{\beta_x}_x \bar{C}^{\beta_t}_t \alpha B_z \eps} & \leq 3 + \tfrac{1}{2\gamma} \tfrac{B_k}{B_z} + \OO(\tfrac{\alpha B_z \eps}{\bar{C}_t} + \tfrac{1}{\bar{C}_x}) 
< 4.
\end{align*}

Now we prove the inequalities \eqref{ineq:eta_x:dx} are  true for $\beta$ with $|\beta| \leq 2n+1$ and $\beta_x = m+1$.  We will do this via a bootstrap argument. Fix  a time  $T \in [0,T_*)$, and suppose that 
\begin{equation*}
\frac{(|\beta|+1)^2 \inorm{\p^\beta \eta_x}_{L^\infty_x}}{|\beta|! \bar{C}^{\beta_x}_x \bar{C}^{\beta_t}_t} \leq 4,
\qquad \mbox{and} \qquad
\frac{(|\beta|+1)^2 \inorm{\p^\beta (\eta_x \W)}_{L^\infty_x}}{|\beta|! \bar{C}^{\beta_x}_x \bar{C}^{\beta_t}_t} \leq  2,
\end{equation*}
for all $t \in [0,T]$ and all $\beta$ with$|\beta| \leq 2n+1,\beta_x = m+1$ and. Taking $\p^\beta$ of \eqref{id:eta_xt} and \eqref{id:W_t} and applying our bootstrap hypotheses, our inductive hypotheses, and the fact that \eqref{ineq:KZ:dx} has already been proven for multi-indices with $\beta_x \leq m+1$, we arrive at
\begin{align*}
\frac{(|\beta|+1)^2 \inorm{\p^\beta \eta_{xt}}_{L^\infty_x}}{|\beta|! \bar{C}^{\beta_x}_x \bar{C}^{\beta_t}_t} 
& \leq \tfrac{1+\alpha}{2}\frac{(|\beta|+1)^2 \inorm{\p^\beta (\eta_x \W)}_{L^\infty_x}}{|\beta|! \bar{C}^{\beta_x}_x \bar{C}^{\beta_t}_t} + \OO((1+\alpha) B_z \eps)  
\leq \tfrac{1+\alpha}{2}(2 + \OO(B_z \eps) ), \\
\frac{(|\beta|+1)^2 \inorm{\p^\beta (\eta_x \W)_t}_{L^\infty_x}}{|\beta|! \bar{C}^{\beta_x}_x \bar{C}^{\beta_t}_t} & \lesssim B_z \eps,
\end{align*}
for all $t \in [0,T]$ and $\beta$ with $|\beta| \leq 2n+1, \beta_x = m+1$. At time $t=0$, Proposition \ref{prop:t=0} and  \eqref{ineq:W_x:t=0}  imply that
\begin{align*}
\frac{(|\beta|+1)^2 \inorm{\p^\beta (\eta_x \W)}_{L^\infty_x}}{|\beta|! \bar{C}^{\beta_x}_x \bar{C}^{\beta_t}_t} & \leq \tfrac{1}{2}
\end{align*}
for all $\beta$ with $1 \leq |\beta| \leq 2n+1$. Now the identities
\begin{align*}
\p^\beta \eta_x & = \int^t_0 \p^\beta \eta_{xt} \: ds, \\
\p^\beta (\eta_x \W) & = \p^\beta (\eta_x \W)(\cdot, 0) + \int^t_0 \p^\beta (\eta_x \W)_t \: ds
\end{align*}
together with \eqref{ineq:T_*}  and our hypotheses on $\eps$ imply
\begin{align*}
\frac{(|\beta|+1)^2 \inorm{\p^\beta \eta_x}_{L^\infty_x}}{|\beta|! \bar{C}^{\beta_x}_x \bar{C}^{\beta_t}_t} & \leq (1+\eps^{\frac{1}{2}})(2 + \OO(B_z \eps) ) < 3, \\
\frac{(|\beta|+1)^2 \inorm{\p^\beta (\eta_x \W)}_{L^\infty_x}}{|\beta|! \bar{C}^{\beta_x}_x \bar{C}^{\beta_t}_t} & \leq  1.
\end{align*}
This completes our bootstrap argument and proves \eqref{ineq:eta_x:dx}  for $\beta$ with $|\beta| \leq 2n+1$ and $\beta_x \leq m+1$.

From here,  one can prove that \eqref{ineq:Sigma_x:dx}  is true for $\p^\beta \Sigma_x$ with $|\beta| \leq 2n+1$ and $\beta_x = m+1$ by applying $\p^\beta$ to \eqref{id:Sigma_x} and using the inductive hypotheses along with the fact that we have already proven \eqref{ineq:KZ:dx} for $\beta$ with $|\beta| \leq 2n+1$ and $\beta_x = m+1$.

Lastly, our bounds on $\p^\beta (\Sigma^{-1})_x$ can now be proven by taking $\p^\beta$ of the identity \eqref{id:Sigma-1:dx} and using the fact that our bounds on $\p^\beta \Sigma_x$ have been proven for $\beta$ with $|\beta| \leq 2n+1$ and $\beta_x \leq m+1$.
\end{proof}

\begin{corollary}\label{cor:ineq:dx}
Under the hypotheses of the previous proposition, we also have that
\begin{subequations}
\begin{align}\label{ineq:Sigma:dx}
\frac{(|\beta|+1)^2 \|\p^\beta \Sigma\|_{L^\infty_x}}{|\beta|! \bar{C}^{\beta_x}_x \bar{C}^{\beta_t}_t} & \leq 3, \\
\label{ineq:W_t:dx}
\frac{(|\beta|+1)^2 \inorm{\p^\beta (\eta_x \W)_t}_{L^\infty_x}}{|\beta|! \bar{C}^{\beta_x}_x \bar{C}^{\beta_t}_t} & \leq  \tfrac{5}{8} B_z \eps, \\
\label{ineq:eta_xt:dx}
\frac{(|\beta|+1)^2 \inorm{\p^\beta \eta_{xt}- \tfrac{1+\alpha}{2} \p^\beta (\eta_x \W)}_{L^\infty_x}}{|\beta|! \bar{C}^{\beta_x}_x \bar{C}^{\beta_t}_t} & \leq 20 \max(1,\alpha) B_z \eps,
\end{align}
\end{subequations}
for $|\beta| \leq 2n+1$.
\end{corollary}

\begin{proof}
The inequality \eqref{ineq:Sigma:dx} follows from \eqref{ineq:0thorder}, \eqref{ineq:Sigma_t:dx}, \eqref{ineq:Sigma_x:dx}, and our hypotheses on $\bar{C}_t$ and $\bar{C}_x$. The inequality \eqref{ineq:W_t:dx} follows from taking $\p^\beta$ of \eqref{id:W_t} and applying the bounds from Proposition \ref{prop:ineq:dx}.

Taking $\p^\beta$ of \eqref{id:eta_xt} and applying \eqref{ineq:KZ:dx}, \eqref{ineq:eta_x:dx}, \eqref{ineq:Sigma:dx}, and Lemma \ref{lem:ineq:sum:nk} gives us
\begin{align*}
\frac{(|\beta|+1)^2 \inorm{\p^\beta \eta_{xt}- \tfrac{1+\alpha}{2} \p^\beta (\eta_x \W)}_{L^\infty_x}}{|\beta|! \bar{C}^{\beta_x}_x \bar{C}^{\beta_t}_t B_z \eps} & \leq \tfrac{3|1-\alpha|}{2}\sum^{|\beta|}_{j=0} \tfrac{(|\beta|+1)^2}{(j+1)^2 (|\beta|+1-j)^2} + \tfrac{9\alpha}{2\gamma} \tfrac{B_k}{B_z}\sum_{j_1 +j_2 +j_3 = |\beta|} \tfrac{(|\beta|+1)^2}{(j_1+1)^2(j_2+1)^2(j_3+1)^2} \\
& \leq \tfrac{3|1-\alpha|}{2} \cdot \tfrac{4\pi^2}{3} + \tfrac{9\alpha}{2\gamma} \cdot \tfrac{1}{3e^{21}} \cdot \tfrac{3\pi^4}{4} \\
& = 2\pi^2|1-\alpha| + \tfrac{9\pi^4 \alpha}{8(2\alpha+1) e^{21}} \\
& < 20\max(1,\alpha) .
\end{align*}

\end{proof}

\begin{corollary}\label{cor:eta_x:approx}
Under the hypotheses of the previous propositions, we have that
\begin{subequations}
\begin{align}
\label{approx:eta_xt:A}
\|\p^i_x \eta_{xt} - \tfrac{1+\alpha}{2} \p^{i+1}_x w_0\|_{L^\infty_x} & \leq 21\max(1,\alpha) \tfrac{i!}{(i+1)^2} \bar{C}^i_x B_z \eps
\end{align}
for all $t \in [0, T_*), i = 0, \hdots, 2n+1$,
\begin{align}
\label{approx:eta_x:A}
\|\p^i_x \eta_x - (\delta_{i0}+ \tfrac{1+\alpha}{2} t\p^{i+1}_x w_0)\|_{L^\infty_x} & \leq 21\max(1,\alpha) t \tfrac{i!}{(i+1)^2} \bar{C}^i_x B_z \eps, 
\end{align}
for all $t \in [0, T_*], i = 0, \hdots, 2n+1$, and
\begin{align}\label{ineq:eta_xtt:dx}
\|\p^i_x \p_t \eta_{xt}\|_{L^\infty_x} & \leq 21(1+\alpha) \tfrac{i!}{i+1} \bar{C}^i_x B_z \eps,
\end{align}
\end{subequations}
for $t \in [0, T_*), i = 0, \hdots, 2n$.
\end{corollary}

\begin{proof}
Since 
\begin{equation*}
\p^i_x \eta_{xt} - \tfrac{1+\alpha}{2}\p^{i+1}_x w_0 = \p^i_x \eta_{xt} - \tfrac{1+\alpha}{2}\p^i_x(\eta_x \W) - \tfrac{1+\alpha}{2} \tfrac{1}{2\gamma} \p^i_x( \sigma_0 k'_0) + \tfrac{1+\alpha}{2} \int^t_0 \p^i_x(\eta_x \W)_t \: ds,
\end{equation*}
it follows from  Corollary \ref{cor:ineq:dx} that for $i=0,\hdots, 2n+1$ we have
\begin{align*}
\tfrac{(i+1)^2 \|\p^i_x \eta_{xt} - \tfrac{1+\alpha}{2}\p^{i+1}_x w_0 \|_{L^\infty_x}}{i! \bar{C}^i_x} & \leq 20 \max(1,\alpha) B_z \eps + \tfrac{1+\alpha}{4}\tfrac{(i+1)^2 \|\p^i_x( \sigma_0 k'_0) \|_{L^\infty_x}}{i! \bar{C}^i_x} + \tfrac{1+\alpha}{2}T_* \tfrac{5\alpha}{4\gamma} B_z \eps.
\end{align*}
Since $\bar{C}_x > \bar{C}_t > 3^\delta C_t(0) > 6^\delta C_0$, it follows (see the computations in the proof of Proposition \ref{prop:t=0}) that 
\begin{equation*}
\tfrac{(i+1)^2 \|\p^i_x( \sigma_0 k'_0) \|_{L^\infty_x}}{i! \bar{C}^i_x} < \eps.
\end{equation*}
Therefore, we arrive at
\begin{align*}
\tfrac{(i+1)^2 \|\p^i_x \eta_{xt} - \tfrac{1+\alpha}{2}\p^{i+1}_x w_0 \|_{L^\infty_x}}{i! \bar{C}^i_x} & \leq \max(1,\alpha) B_z \eps[ 20 + \tfrac{1}{2B_z} + \tfrac{5}{8}(1+ \eps^{\frac{1}{2}})] < 21\max(1,\alpha) B_z \eps.
\end{align*}
This bound implies \eqref{approx:eta_xt:A}.

To obtain \eqref{approx:eta_x:A},  integrate \eqref{approx:eta_xt:A} in time. The last inequality \eqref{ineq:eta_xtt:dx} is a direct application of \eqref{ineq:W_t:dx} and \eqref{ineq:eta_xt:dx}. 
\end{proof}



\section{Stability estimates  with respect to perturbations of $w_0$}
\label{sec:stab}

We will now quantify the stability of solutions with initial data in $\mathcal A_n(\eps, C_0)$ with respect to perturbations in $w_0$. Fix a specific choice of data $(\bar{w}_0, z_0, k_0) \in \mathcal A_n(\eps, C_0)$. In this section, we will consider initial data $(w_0,z_0,k_0)$ of the form
\begin{equation*}
w_0 = \bar{w}_0 + \lambda \tilde{w}_0\,, \qquad z_0 = z_0\,, \qquad k_0 = k_0 \, ,
\end{equation*}
where $\tilde{w}_0 \in W^{2n+2, \infty}(\TT)$ satisfies 
\begin{align*}
\|\tilde{w}_0\|_{L^\infty_x} & \leq L, &  \\
\frac{\|\p^{i+1}_x \tilde{w}_0\|_{L^\infty}}{i!} & \leq L C_0^i & \forall \: i = 0, \hdots, 2n+1,
\end{align*}
for some positive constant $L$, and $\lambda$ is in a small enough neighborhood of 0 that $(w_0, z_0, k_0)$ is still in $\mathcal A_n(\eps, C_0)$ for all $\lambda$ considered. Such initial data now defines a family of solutions $\K = \K(x,t,\lambda), \Z = \Z(x,t,\lambda)$, etc. with initial data in $\mathcal A_n(\eps, C_0)$.  Our goal is the estimate the partial derivatives $(\eta_x \W)_\lambda, \Z_\lambda, \K_\lambda, \Sigma_\lambda,$ etc. and their derivatives in space $x$ and time $t$. In the end, we will obtain the estimates stated in \S~\ref{sec:ineq:dx:stab}, which should be viewed as extensions to the bounds obtained earlier in \S~\ref{sec:ineq:dx}. These estimates will be used in \S~ \ref{sec:FCBM}. The results in this section and their proofs will be essentially analogous to those found in \S~ \ref{sec:HOE}, with only minor differences. For this reason, less detail will be provided; however, the key differences between the proofs in this section and the corresponding proofs in the previous section will be highlighted.

Taking $\p_\lambda$ of \eqref{system:lagrange} gives us the system
\begin{subequations}\label{system:stab}
\begin{align}
\label{id:Sigma_t:stab}
\p_\lambda \Sigma_t 
& = \alpha\Sigma_\lambda (- \Z + \tfrac{1}{\gamma} \Sigma \K) - \alpha\Sigma \Z_\lambda + \tfrac{\alpha}{2\gamma} \Sigma^2 \K_\lambda, 
\\
\label{id:eta_xt:stab}
\p_\lambda \eta_{xt} 
& = \tfrac{1+\alpha}{2} (\eta_x \W)_\lambda + \p_\lambda \eta_x ( \tfrac{1-\alpha}{2} \Z + \tfrac{\alpha}{2\gamma} \Sigma \K) + \tfrac{\alpha}{2\gamma} \eta_x \Sigma_\lambda \K 
+ \tfrac{1-\alpha}{2} \eta_x \Z_\lambda + \tfrac{\alpha}{2\gamma} \eta_x \Sigma \K_\lambda, 
\\
\label{id:W_t:stab}
 \p_\lambda (\eta_x \W)_t 
& = \tfrac{\alpha}{4\gamma}(\Sigma_\lambda \K + \Sigma \K_\lambda)(\eta_x \W + \eta_x \Z) + \tfrac{\alpha}{4\gamma}\Sigma \K [ (\eta_x \W)_\lambda + \p_\lambda \eta_x \Z + \eta_x \Z_\lambda], \\
\label{id:Sigma_x:stab}
\p_\lambda \Sigma_x & = \tfrac{1}{2}(\eta_x \W)_\lambda - \tfrac{1}{2}\p_\lambda \eta_x \Z - \tfrac{1}{2}\eta_x \Z_\lambda + \tfrac{1}{2\gamma}[ \p_\lambda \eta_x \Sigma \K + \eta_x \Sigma_\lambda \K + \eta_x \Sigma \K_\lambda ],
\\
\label{id:K_t:stab}
\eta_x \p_t \K_\lambda & = \alpha \Sigma \p_x \K_\lambda - \tfrac{1}{2}\K_\lambda[ \eta_x \W + \eta_x \Z]  - \p_\lambda \eta_x \p_t \K + \alpha \Sigma_\lambda \p_x \K\\
& - \tfrac{1}{2} \eta_x \K \Z_\lambda - \tfrac{1}{2}\K [ (\eta_x \W)_\lambda + \p_\lambda \eta_x \Z], \notag \\
\label{id:Z_t:stab}
\eta_x \p_t \Z_\lambda & = 2\alpha \Sigma \p_x \Z_\lambda + \Z_\lambda [ - \tfrac{1-\alpha}{2} \eta_x \W - (1+\alpha) \eta_x \Z + \tfrac{\alpha}{4\gamma} \eta_x \Sigma \K ] \\
& - \p_\lambda \eta_x \p_t \Z + 2\alpha \Sigma_\lambda \p_x \Z + \tfrac{\alpha}{4\gamma} \K_\lambda \Sigma(\eta_x \Z- \eta_x \W) \notag \\
& + \Z[ - \tfrac{1-\alpha}{2} (\eta_x \W)_\lambda - \tfrac{1+\alpha}{2} \p_\lambda \eta_x \Z + \tfrac{\alpha}{4\gamma} (\p_\lambda \eta_x \Sigma + \eta_x \Sigma_\lambda) \K ]. \notag
\end{align}
Mulitplying \eqref{id:K_t:stab} and \eqref{id:Z_t:stab} by $\Sigma^{-1}$ and rearranging gives us
\begin{align}
\label{id:K_x:stab}
 \alpha  \p_x \K_\lambda& = \eta_x \Sigma^{-1} \p_t \K_\lambda + \tfrac{1}{2}\K_\lambda \Sigma^{-1}[ \eta_x \W + \eta_x \Z]  + \p_\lambda \eta_x \Sigma^{-1} \p_t \K - \alpha \Sigma_\lambda \Sigma^{-1} \p_x \K\\
& + \tfrac{1}{2} \eta_x \K \Sigma^{-1} \Z_\lambda + \tfrac{1}{2}\Sigma^{-1}\K [ (\eta_x \W)_\lambda + \p_\lambda \eta_x \Z], \notag \\
\label{id:Z_x:stab}
2\alpha \p_x \Z_\lambda  & = \eta_x \Sigma^{-1} \p_t \Z_\lambda - \Z_\lambda [ - \tfrac{1-\alpha}{2}\Sigma^{-1}( \eta_x \W) - (1+\alpha) \eta_x \Sigma^{-1} \Z + \tfrac{\alpha}{4\gamma} \eta_x \K ] \\
& + \p_\lambda \eta_x \Sigma^{-1} \p_t \Z - 2\alpha \Sigma_\lambda \Sigma^{-1} \p_x \Z - \tfrac{\alpha}{4\gamma} \K_\lambda (\eta_x \Z- \eta_x \W) \notag \\
& - \Z[ - \tfrac{1-\alpha}{2} \Sigma^{-1} (\eta_x \W)_\lambda - \tfrac{1+\alpha}{2} \p_\lambda \eta_x \Sigma^{-1} \Z + \tfrac{\alpha}{4\gamma} (\p_\lambda \eta_x  + \eta_x \Sigma_\lambda \Sigma^{-1}) \K ]. \notag
\end{align}
\end{subequations}
Notice that $\p_\lambda(\Sigma^{-1})$ does not appear in any of the above equations; for this reason, we will not have to estimate $\p_\lambda(\Sigma^{-1})$ or any of its derivatives.

\subsection{Estimates at time zero}

\begin{proposition}\label{prop:t=0:stab}
 Let $\C_x$ and $\C_t$ and be constants satisfying
\begin{itemize}[leftmargin=*]
\item $\C_x \geq 2e^3 C_0$,
\item $\C_x \gg 1$,
\item $\C_t \gg (1+\alpha)$,
\item $\alpha \C_x \ll \C_t$.
\end{itemize}
Then if $(1+\alpha)\eps \ll 1$ we have the following bounds at time $t=0$ for all multi indices $\beta$ with $|\beta| \leq 2n$:
\begin{align*}
\frac{(|\beta|+1)^2\| \p^\beta \K_\lambda\|_{L^\infty_x}}{|\beta|! L  \C^{\beta_x}_x \C^{\beta_t}_t} & \leq \eps,
&
\frac{(|\beta|+1)^2\| \p^\beta \Z_\lambda\|_{L^\infty_x}}{|\beta|! L \C^{\beta_x}_x \C^{\beta_t}_t} & \leq \eps, \\
\frac{(|\beta|+1)^2\| \p^\beta \p_\lambda \Sigma_t\|_{L^\infty_x}}{|\beta|! L\C^{\beta_x}_x \C^{\beta_t}_t} & \leq 7\alpha \eps, 
&
\frac{(|\beta|+1)^2\| \p^\beta \p_\lambda \eta_{xt}\|_{L^\infty_x}}{|\beta|! L \C^{\beta_x}_x \C^{\beta_t}_t} & \leq (1+\alpha), \\
\frac{(|\beta|+1)^2\| \p^\beta \p_\lambda (\eta_x \W)_t\|_{L^\infty_x}}{|\beta|! L \C^{\beta_x}_x \C^{\beta_t}_t} & \leq 8\eps. & &
\end{align*}
\end{proposition}

\begin{proof}
Our proof will be almost identical to the proof of Proposition \ref{prop:t=0}. We will do induction on $\beta_t$.

{\it{Base Case:}} Recall from the proof of Proposition \ref{prop:t=0} that at time $t=0$
\begin{align*}
\Sigma & = \sigma_0, &
\eta_x & = 1, &
\eta_x \W & = w'_0 - \tfrac{1}{2\gamma} \sigma_0 k'_0, &
\Z & = z'_0 + \tfrac{1}{2\gamma} \sigma_0 k'_0, \\
\K & = k'_0, &
\Sigma_t & = -\alpha \sigma_0 z'_0 &
\eta_{xt} & = \tfrac{1+\alpha}{2}w'_0 + \tfrac{1-\alpha}{2} z'_0, &
(\eta_x \W)_t & = \tfrac{\alpha}{4\gamma} \sigma_0 k'_0(w'_0 + z'_0). &
\end{align*}
Taking $\p_\lambda$ of these equations gives
\begin{equation}\label{id:t=0:stab}
\begin{aligned}
\Sigma_\lambda & = \tfrac{1}{2}\tilde{w}_0 &
\p_\lambda \eta_x & = 0, &
(\eta_x \W)_\lambda & = \tilde{w}'_0 - \tfrac{1}{4\gamma} \tilde{w}_0 k'_0, &
\Z_\lambda & = \tfrac{1}{4\gamma} \tilde{w}_0 k'_0, \\
\K_\lambda & = 0, &
\p_\lambda \Sigma_t & = -\tfrac{\alpha}{2} \tilde{w}_0 z'_0 &
\p_\lambda \eta_{xt} & = \tfrac{1+\alpha}{2}\tilde{w}'_0, &
\p_\lambda (\eta_x \W)_t & = \tfrac{\alpha}{4\gamma}[ \tfrac{1}{2}\tilde{w}_0 k'_0(w'_0 + z'_0) + \sigma_0 k'_0 \tilde{w}_0'], &
\end{aligned}
\end{equation}
at time $t=0$. Now take $p^m_x$ of these equations for $m =0, \hdots, 2n$ and perform computations analogous to those in the $\beta_t =0$ step of the proof of Proposition \ref{prop:t=0} to conclude our base case. One can also compute the analog of \eqref{ineq:W_x:t=0},
\begin{equation*}
\frac{(m+1)^2\|\p^m_x (\eta_x \W)_\lambda\|_{L^\infty_x}}{m! L \C^m_x} \leq 1 + \tfrac{\eps}{4\gamma},
\end{equation*}
to use in the inductive step.

{\it{Inductive Step:}} The inductive step will also be analogous to the inductive step in the proof of Proposition \ref{prop:t=0}: fix $0 \leq m \leq 2n-1$, and suppose our result is true for all multi indices $\beta$ with $|\beta| \leq 2n$ and $\beta_t \leq m$.

We first do the inductive step for $\K_\lambda$ and $\Z_\lambda$.  Let $|\beta| \leq 2n-1, \beta_t = m$. Take the identities for $\p^{\beta+e_t} \K$ and $\p^{\beta+e_t} \Z$ used in the proof of Proposition \ref{prop:t=0}, apply $\p_\lambda$ to each of these equations, and use \eqref{id:t=0:stab} to simplify. From here, use the inductive hypotheses and \eqref{ineq:prop:t=0} to obtain bounds on $\p^{\beta+e_t} \K_\lambda$ and $\p^{\beta+e_t} \Z_\lambda$ via the same types of computations as used in the proof of Proposition \ref{prop:t=0}. Because $\C_x \gg 1, \C_t \gg \alpha \C_x,$ and $\C_t \gg 1+\alpha$, the computations go through just like before.

The inductive step for the remaining three inequalities can be done using  \eqref{ineq:prop:t=0}  and our inductive hypotheses in conjunction with our bounds on $\p^\beta \K_\lambda$ and $\p^\beta \Z_\lambda$ in a manner completely analogous to the corresponding steps in the proof of Proposition \ref{prop:t=0}. 
\end{proof}


\subsection{A Priori estimates for time derivatives} \label{sec:apriori:stab}

In this section, we will fix a time $T \in (0, T_*]$, a positive function $C_t:[0,T_*) \rightarrow \R^+$, a constant $\delta \geq 0$, and constants $A,M$ satisfying
\begin{equation*}
A \geq 10B_z,
\qquad
M \geq \max(1, 5L),
\end{equation*}
and assume that 
\begin{equation}\label{ineq:apriori:hyp:stab}
\frac{(m+1)^2 \|\Sigma^{-\delta m} \p^m_t \K_\lambda\|_{L^\infty_x}}{m! M C^m_t} \leq A \eps, 
\qquad
\frac{(m+1)^2 \|\Sigma^{-\delta m} \p^m_t \Z_\lambda\|_{L^\infty_x}}{m! M C^m_t} \leq A \eps,
\end{equation}
for all  $t \in [0,T], \: 0 \leq m \leq 2n.$ We will furthermore assume that 
\begin{equation}\label{ineq:apriori:hyp:previous}
\frac{(m+1)^2 \|\Sigma^{-\delta m} \p^m_t \K\|_{L^\infty_x}}{m!C^m_t} \leq B_k \eps, 
\qquad
\frac{(m+1)^2 \|\Sigma^{-\delta m} \p^m_t \Z\|_{L^\infty_x}}{m! C^m_t} \leq B_z \eps,
\end{equation}
for all  $t \in [0,T], \: 0 \leq m \leq 2n+1.$ Note that the hypotheses \eqref{ineq:apriori:hyp:previous} imply that the a priori estimates of \S~\ref{sec:apriori} all hold with $A_k = B_k, A_z = B_z$ for these choices of $C_t$ and $\delta$. We will use this fact throughout this subsection and the next.

\begin{proposition}\label{prop:apriori:0:stab}
If $(1+\alpha) B_z \eps \ll1$ and $A \eps \ll 1$, then for all $t \in [0,T]$ we have
\begin{equation}\label{ineq:0:stab}
\frac{\|\Sigma_\lambda\|_{L^\infty_x}}{M} \leq  \tfrac{1}{9}, 
\qquad
\frac{\|\p_\lambda \eta_x\|_{L^\infty_x}}{M}  \leq 1,
\qquad
\frac{\|(\eta_x \W)_\lambda\|_{L^\infty_x}}{M} \leq \tfrac{3}{4}.
\end{equation}
\end{proposition}

\begin{proof}
\eqref{id:Sigma_t:stab} gives us a Duhamel formula for $\Sigma_\lambda$ from which it follows that 
\begin{align*}
|\Sigma_\lambda (x,t)| & \leq \tfrac{1}{2}|\tilde{w}_0(x)| e^{\OO(\alpha t B_z \eps)} + \alpha t MA\eps [3 + \tfrac{9}{2\gamma}]
\end{align*}
for all $(x,t) \in \TT \times [0,T]$. Since $M \geq \max(1, 5L)$, if $ B_z \eps$ and $A\eps$ are small enough we have $\|\Sigma\|_{L^\infty_x} \leq \frac{M}{9}$.

The remaining bounds follow from a bootstrap argument. Let $T' \in [0,T]$  and assume that
\begin{equation*}
\frac{\|\p_\lambda \eta_x\|_{L^\infty_x}}{M} \leq 2,
\qquad
\frac{\|(\eta_x \W)_\lambda\|_{L^\infty_x}}{M} \leq \tfrac{3}{4}+ \eps,
\end{equation*}
for all $t \in [0,T']$. Then using \eqref{id:eta_xt:stab}, \eqref{id:W_t:stab}, and the equations
\begin{align*}
\p_\lambda \eta_x & = \int^t_0 \p_\lambda \eta_{xt} \: ds, &
(\eta_x \W)_\lambda & = \tilde{w}_0' - \tfrac{1}{4\gamma} \tilde{w}_0 k'_0 + \int^t_0 \p_\lambda (\eta_x \W)_t \: ds
\end{align*}
allows us to prove \eqref{ineq:0:stab} for $t \in [0,T']$, which concludes our bootstrap argument.
\end{proof}

\begin{proposition}\label{prop:apriori:est:stab}
If the hypotheses of the previous proposition are satisfied, and $C_t(0), \eps$ also satisfy
\begin{equation*}
(1+\alpha)3^\delta \ll C_t(0),
\qquad
(1+\alpha)B_z \eps \ll 1,
\end{equation*}
then
\begin{equation}\label{apriori:est:stab}
\begin{aligned}
\frac{(m+1)^2\|\Sigma^{-\delta m} \p^{m}_t \p_\lambda \Sigma_t\|_{L^\infty_x}}{m! M C^m_t} & \leq 8\alpha A\eps, & & \\
\frac{(m+1)^2\|\Sigma^{-\delta m} \p^{m}_t \p_\lambda \eta_{xt} \|_{L^\infty_x}}{m! M C^m_t} & \leq  \tfrac{1+\alpha}{2}(\delta_{0m} + 9A\eps), &
\frac{(m+1)^2\|\Sigma^{-\delta m} \p^{m}_t \p_\lambda (\eta_x\W)_t\|_{L^\infty_x}}{m! M C^m_t} & \leq A\eps,
\end{aligned}
\end{equation}
for all $t \in [0,T], 0 \leq m \leq 2n$.
\end{proposition}

\begin{proof}
The proof of these estimates is analogous to the proof of Proposition \ref{prop:apriori:est}: do induction on $m$; at each step, take $\p^m_t$ of the identities \eqref{system:stab}, apply \eqref{ineq:apriori:hyp:stab}, the inductive hypotheses, and the estimates from \S~\ref{sec:apriori}; then use the bounds from \S~\ref{sec:appendix:freqcomp} and the fact that $\frac{(1+\alpha)3^\delta}{C_t}$ is sufficiently small to conclude.
\end{proof}

\begin{corollary}\label{cor:apriori:est:stab}
Under the hypotheses of the previous proposition, we have
\begin{align*}
\frac{(m+1)^2 \|\Sigma^{-\delta m} \p^m_t \p_\lambda \eta_x\|_{L^\infty_x}}{m! M C^m_t} & \leq 1,  & 
\frac{(m+1)^2 \|\Sigma^{-\delta m} \p^m_t( \eta_x \W)_\lambda \|_{L^\infty_x}}{m! C^m_t} & \leq \tfrac{3}{4}, \\
\frac{(m+1)^2 \|\Sigma^{-\delta m} \p^m_t \p_\lambda \Sigma_x\|_{L^\infty_x}}{m! M C^m_t} & \leq 1,  &  &  
\end{align*}
for all $t \in [0, T], 0 \leq m \leq 2n$ and
\begin{equation*}
\frac{(m+1)^2 \|\Sigma^{-\delta m} \p^{m-1}_t \p_x\K_\lambda\|_{L^\infty_x}}{m! M C^m_t} \leq \tfrac{10}{\alpha} A\eps,  \qquad
\frac{(m+1)^2 \|\Sigma^{-\delta m} \p^{m-1}_t \p_x \Z_\lambda \|_{L^\infty_x}}{m! M C^m_t} \leq \tfrac{5}{\alpha}A\eps , 
\end{equation*}
for all  $t \in [0,T], \: 1 \leq m \leq 2n$.
\end{corollary}

\begin{proof}
This follows from Proposition \ref{prop:apriori:est:stab} the same way that Corollary \ref{cor:apriori:est} follows from Proposition \ref{prop:apriori:est}.
\end{proof}

\subsection{Energy estimates for time derivatives}

\begin{proposition}\label{prop:EE:stab}
Define the constant
\begin{equation}\label{def:B:lambda}
B_\lambda : = \tfrac{5}{2} \cdot 9^{\frac{3}{\min(1,\alpha)}} \cdot e^{31}.
\end{equation}
Let $\delta, \kappa$ be as in \eqref{def:delta}, and let $C_t$ and $\bar{C}_t$ be defined as in \eqref{def:C_t} and \eqref{def:C_t:bar} with these values of $\delta, \kappa$. If  $C_t(0)$ and $M$ satisfy
\begin{equation*}
M \geq 5 L,
\qquad
C_t(0) \gg (1+\alpha) 3^\delta, 
\qquad
C_t(0) \gg \alpha C_0
\end{equation*}
and $\eps$ satisfies
\begin{align*}
(1+\alpha)^2 B_z^2 \eps & \ll 1, & B_\lambda \bar{C}_tB_z\eps & \ll 1, 
\end{align*}
then 
\begin{equation}\label{ineq:EE:dt:stab}
\frac{(m+1)^2 \|\Sigma^{-\delta m} \p^{m}_t \K_\lambda\|_{L^\infty_x}}{m! M C^m_t} 
\leq B_\lambda\bar{C}_tB_z\eps, 
\qquad
\frac{(m+1)^2 \|\Sigma^{-\delta m} \p^{m}_t \Z_\lambda\|_{L^\infty_x}}{m! M C^m_t} 
\leq B_\lambda \bar{C}_tB_z\eps,
\end{equation}
for all  $t \in [0,T_*), \: 0 \leq m \leq 2n$.
\end{proposition}

\begin{proof}
The proof of these estimates is analogous to the energy estimates in the proofs of Proposition \ref{prop:0thorder} and Proposition \ref{prop:EE:infty}. We will detail the keys steps of the proof and highlight the slight differences with the previous energy estimates. We proceed as before with a bootstrap argument. Pick a time $T \in [0, T_*)$ and constant $A$ satisfying
\begin{equation*}
B_\lambda \bar{C}_tB_z < A < \tfrac{4}{3}B_\lambda \bar{C}_tB_z,
\end{equation*}
and assume that
\begin{equation*}
\frac{(m+1)^2 \|\Sigma^{-\delta m} \p^{m}_t \K_\lambda\|_{L^\infty_x}}{m! M C^m_t} 
\leq A\eps, 
\qquad
\frac{(m+1)^2 \|\Sigma^{-\delta m} \p^{m}_t \Z_\lambda\|_{L^\infty_x}}{m! M C^m_t} 
\leq A\eps,
\end{equation*}
for all  $t \in [0,T], \: 0 \leq m \leq 2n$. Since $\delta$ and $\kappa$ satisfy \eqref{def:delta}, if $C_t(0)$ is chosen large enough and $\eps$ is chosen small enough such that they satisfy the hypotheses of Proposition \ref{prop:EE:infty}, we can conclude that \eqref{ineq:apriori:hyp:previous} holds.  Suppose that $B_\lambda \bar{C}_t B_z \eps$ is small enough that the hypotheses of Proposition \ref{prop:apriori:0:stab} are satisfied with our choice of $A$. If $(1+\alpha)B_z \eps$ is small enough, these bootstrap hypotheses now imply that we can apply the a priori estimates of \S~\ref{sec:apriori:stab} to our solution with our chosen values of $A$ and $M$. Just like in the proof of Proposition \ref{prop:EE:infty}, our hypotheses on $C_t(0), M$, and $\eps$ allow us to apply Proposition \ref{prop:t=0:stab} with $\C_t : = \frac{1}{2}2^{-\delta} C_t(0)$ and get
\begin{equation}\label{ineq:bsh:t=0:stab}
\begin{aligned}
\frac{(m+1)^2 \|\Sigma^{-\delta m} \p^m_t \K_\lambda\|_{L^\infty_x}}{m! M C_t(0)^m} & \leq \tfrac{1}{5}\tfrac{\eps}{2^m} , &
\frac{(m+1)^2 \|\Sigma^{-\delta m} \p^m_t \Z_\lambda\|_{L^\infty_x}}{m! MC_t(0)^m} & \leq \tfrac{1}{5}\tfrac{\eps}{2^m} ,
\end{aligned}
\end{equation}
at time $t=0$ for all $0 \leq m \leq 2n$.

For $m=0, \hdots, 2n$ define the energies
\begin{align*}
\tilde{\E}^{m,q}_k & : =\begin{cases} \int_\TT \Sigma^{-q\frac{3}{\min(1,\alpha)}} \eta_x | \K_\lambda|^q \: dx & m=0 \\ \int_\TT \Sigma^{-\delta q m} \eta_x |\p^m_t \K_\lambda|^q \: dx & 1 \leq m \leq 2n \end{cases} ,&
\tilde{\E}^{m,q}_z & : =\begin{cases} \int_\TT \Sigma^{-q\frac{3}{\min(1,\alpha)}} \eta_x | \Z_\lambda|^q \: dx & m=0 \\ \int_\TT \Sigma^{-\delta q m} \eta_x |\p^m_t \Z_\lambda|^q \: dx & 1 \leq m \leq 2n \end{cases} ,
\end{align*}
and for $1 \leq m \leq 2n$ define the energies
\begin{align*}
\tilde{E}^{m,q}_k & : = \frac{(m+1)^{2q} \tilde{\E}^{m,q}_k}{(m!)^q M^q C^{mq}_t},  &
\tilde{E}^{m,q}_z & : = \frac{(m+1)^{2q} \tilde{\E}^{m,q}_z}{(m!)^q M^q C^{mq}_t}  .
\end{align*}
We will prove \eqref{ineq:EE:dt:stab} for $t \in [0,T]$ separately in the case $m=0$ and the case $1 \leq m \leq 2n$, and since this strictly improves upon our bootstrap assumptions our bootstrap argument will be complete.

{\bf{Case 1 ($m=0$):}} We will use our bootstrap hypotheses and the estimates from \S~\ref{sec:HOE} to perform energy estimates on $\tilde{\E}^{0,q}_k, \tilde{\E}^{0,q}_z$ similar to those in the proof of Proposition \ref{prop:0thorder}. Since the parameter $\delta$ does not appear in $\tilde{\E}^{0,q}_k, \tilde{\E}^{0,q}_z$, let us abuse notation and define $\delta : = \frac{3}{\min(1,\alpha)}$ just for the $m=0$ case.

Just like in the proof of Proposition \ref{prop:0thorder}, taking the time derivative of $\tilde{\E}^{0,q}_k$ and using our identities \eqref{system:stab} gives us
\begin{align*}
\dot{\tilde{\E}}^{0,q}_k & = [ \tfrac{\alpha}{2}\delta q - \tfrac{q-1}{2} ]  \int \Sigma^{-\delta q}  |\K_\lambda|^q \eta_x \W +  [ \tfrac{\alpha}{2}\delta q - \tfrac{q-1}{2} ] \int \Sigma^{-\delta q} \eta_x |\K_\lambda|^q \Z
\\
& + q\int \Sigma^{-\delta q}  |\K_\lambda|^{q-1} \sgn(\K_\lambda)[ - \p_\lambda \eta_x \p_t \K + \alpha \p_\lambda \Sigma \p_x \K]
\\
& - q\int \Sigma^{-\delta q}  |\K_\lambda|^{q-1} \sgn(\K_\lambda) \tfrac{1}{2} \eta_x \K \Z_\lambda
\\
& - q\int \Sigma^{-\delta q}  |\K_\lambda|^{q-1} \sgn(\K_\lambda) \tfrac{1}{2}\K[(\eta_x \W)_\lambda + \p_\lambda \eta_x \Z],
\\
\dot{\tilde{\E}}^{0,q}_z & = [\alpha \delta q - \tfrac{1-\alpha}{2}(q-1)]\int \Sigma^{-\delta q}  |\Z_\lambda|^q \eta_x \W 
\\
& +  \int \Sigma^{-\delta q} \eta_x |\Z_\lambda|^q\big[ - \tfrac{1+\alpha}{2}(q-1) \Z + [\alpha \delta q + \alpha (\tfrac{q}{2}-1)] \tfrac{1}{2\gamma} \Sigma \K ]
\\
& + q\int \Sigma^{-\delta q} \eta_x |\Z_\lambda|^{q-1} \sgn(\Z_\lambda)[ -\p_\lambda \eta_x \p_t \Z + 2\alpha \p_\lambda \Sigma \p_x \Z]
\\
& + q\int \Sigma^{-\delta q} \eta_x |\Z_\lambda|^{q-1} \sgn(\Z_\lambda) \tfrac{\alpha}{4\gamma} \K_\lambda \Sigma( \eta_x \Z- \eta_x \W)
\\ 
& + q\int \Sigma^{-\delta q} \eta_x |\Z_\lambda|^{q-1} \sgn(\Z_\lambda) \Z [-\tfrac{1-\alpha}{2}(\eta_x \W)_\lambda - \tfrac{1+\alpha}{2} \p_\lambda \eta_x \Z] 
\\
& + q\int \Sigma^{-\delta q} \eta_x |\Z_\lambda|^{q-1} \sgn(\Z_\lambda) \tfrac{\alpha}{4\gamma} \K \Z (\p_\lambda \eta_x \Sigma + \eta_x \p_\lambda \Sigma).
\end{align*}
 The bounds \eqref{cor:apriori:KZ:dx} and \eqref{ineq:KZ:dt:new} tell us that
 \begin{equation}\label{ineq:KZ:C_t:bar}
 \begin{aligned}
 \| \p_x \K \|_{L^\infty_x} & \leq \tfrac{5}{2\alpha} \bar{C}_t B_k \eps, &  \| \p_x \Z \|_{L^\infty_x} & \leq \tfrac{5}{4\alpha} \bar{C}_t B_z \eps,
 \\
 \| \p_t \K \|_{L^\infty_x} & \leq \tfrac{1}{4}\bar{C}_tB_k \eps, &  \| \p_t \Z \|_{L^\infty_x} & \leq \tfrac{1}{4}\bar{C}_tB_z \eps.
\end{aligned}
\end{equation}
Therefore, applying \eqref{ineq:W} as in the proof of Proposition \ref{prop:0thorder} gives us
\begin{align*}
& \dot{\tilde{\E}}^{0,q}_k 
+ \bigg[ \tfrac{1}{2}[ \tfrac{\alpha}{2}\delta q - \tfrac{q-1}{2} ]  - (q-1)( \tfrac{B_k}{B_z}+ \tfrac{3}{2}B_k \eps) \bigg] \|\Sigma^{-\delta} \K_\lambda \|^q_{L^q_x} 
- \tfrac{3}{2}B_k \eps \|\Sigma^{-\delta} \Z_\lambda \|^q_{L^q_x}
\\
& \leq (4+B_z \eps) [ \tfrac{\alpha}{2}\delta q - \tfrac{q-1}{2} ] \tilde{\E}^{0,q}_k + \big( 3^\delta MB_z ( \tfrac{19}{36} \bar{C}_t + \tfrac{3}{8} + \tfrac{1}{2}B_z \eps) \eps \big)^q,
\\
 & \dot{\tilde{\E}}^{0,q}_z - \tfrac{\alpha}{\gamma}(1+ \tfrac{9}{4}B_z \eps) \|\Sigma^{-\delta} \K_\lambda\|^q_{L^q_x} + \bigg[ \tfrac{1}{2}[\alpha \delta q- \tfrac{q-1}{2}] - (q-1)[ \tfrac{\alpha}{\gamma}(1+ \tfrac{9}{4}B_z \eps) + \tfrac{1}{2}] \bigg] \|\Sigma^{-\delta} \Z_\lambda\|^q_{L^q_x}
 \\
 & \leq \big( 4 [\alpha \delta q - \tfrac{1-\alpha}{2}(q-1)] + \tfrac{1+\alpha}{2}(q-1)B_z \eps + \tfrac{3}{2\gamma} [\alpha \delta q + \alpha (\tfrac{q}{2}-1)] B_k \eps  \big) \tilde{\E}^{0,q}_z 
 \\
 & + \bigg(2 \cdot 3^\delta M B_z ( \tfrac{19}{36} \bar{C}_t + \tfrac{3|1-\alpha|}{8} + \tfrac{1+\alpha}{2} B_z \eps + \tfrac{5\alpha}{6 \gamma} B_k \eps) \eps \bigg)^q.
\end{align*}
Since $\bar{C}_t > e^5C_t(0) \gg 1+\alpha$, it follows that if $(1+\alpha) B_z \eps$ is small enough then
\begin{align*}
\tfrac{19}{36} \bar{C}_t + \tfrac{3}{8} + \tfrac{1}{2}B_z \eps & < \tfrac{2}{3}\bar{C}_t,
\\
2(\tfrac{19}{36} \bar{C}_t + \tfrac{3|1-\alpha|}{8} + \tfrac{1+\alpha}{2} B_z \eps + \tfrac{5\alpha}{6 \gamma} B_k \eps) & < \tfrac{4}{3}\bar{C}_t.
\end{align*}
Adding our inequalities together and using our choice $\delta = \frac{3}{\min(1,\alpha)}$ implies that 
\begin{align*}
\dot{\tilde{\E}}^{0,q}_k + \dot{\tilde{\E}}^{0,q}_z & \leq 15\max(1,\alpha) q(\tilde{\E}^{0,q}_k + \tilde{\E}^{0,q}_z) + (\tfrac{2}{3} 3^{\frac{3}{\min(1,\alpha)}} M \bar{C}_t B_z \eps)^q + (\tfrac{4}{3} 3^{\frac{3}{\min(1,\alpha)}} M \bar{C}_t B_z \eps)^q.
\end{align*}
ODE comparison now gives us
\begin{align*}
(\tilde{\E}^{0,q}_k(t) + \tilde{\E}^{0,q}_z(t)) & \leq \bigg[ (\tilde{\E}^{0,q}_k(0) + \tilde{\E}^{0,q}_z(0)) + \frac{(\tfrac{2}{3} 3^{\frac{3}{\min(1,\alpha)}} M \bar{C}_t B_z \eps)^q + (\tfrac{4}{3} 3^{\frac{3}{\min(1,\alpha)}} M \bar{C}_t B_z \eps)^q}{15 \max(1,\alpha) q} \bigg] e^{15 \max(1,\alpha) q t}
\end{align*}
for all $t \in [0, T]$. Taking $q$th roots of both sides, sending $q \rightarrow \infty$, and using \eqref{ineq:bsh:t=0:stab} and \eqref{ineq:T_*} now produces
\begin{align*}
\max( \|\Sigma^{-\frac{3}{\min(1,\alpha)}} \K_\lambda\|_{L^\infty_x}, \|\Sigma^{-\frac{3}{\min(1,\alpha)}} \Z_\lambda\|_{L^\infty_x}) \leq M \eps \bigg[\tfrac{2^{\frac{3}{\min(1,\alpha)}}}{5} + 2 \cdot  3^{\frac{3}{\min(1,\alpha)}} \bar{C}_t B_z \bigg] e^{31}
\end{align*}
for all $t \in [0,T]$. This gives us \eqref{ineq:EE:dt:stab} for $m=0$, $t\in [0,T]$.


{\bf{Case 2 ($1 \leq m \leq 2n$):}} Let $m \geq 1$. If $D_k = D_k(m,q,\alpha)$ is the same constant defined by \eqref{def:D_k}, then the same computations as in the proof of Proposition \ref{prop:EE:infty} give us 
\begin{align*}
\dot{\tilde{\E}}^{m,q}_k 
& = D_k \int \Sigma^{-\delta m q} |\p^m_t \K_\lambda|^q \eta_x \W 
+ (D_k + \alpha m q)\int \Sigma^{-\delta m q} \eta_x |\p^m_t \K_\lambda|^q \Z 
\\
&  - \tfrac{\alpha}{2\gamma} mq \int \Sigma^{-\delta m q} \eta_x |\p^m_t \K_\lambda|^q \Sigma \K 
- q \int \Sigma^{-\delta m q} |\p^m_t \K_\lambda|^{q-1} \sgn(\p^m_t \K_\lambda)  \p_\lambda \eta_x \p^{m+1}_t \K \\
& + q\alpha \int \Sigma^{-\delta m q} |\p^m_t \K_\lambda|^{q-1} \sgn(\p^m_t \K_\lambda) \Sigma_\lambda \p^m_t \p_x \K \\
& -q \int \Sigma^{-\delta m q} |\p^m_t \K_\lambda|^{q-1} \sgn(\p^m_t \K_\lambda) \p^m_t \K[\tfrac{1}{2} (\eta_x \W)_\lambda + m \p_\lambda \eta_{xt}] \\
& - q\int \Sigma^{-\delta m q} |\p^m_t \K_\lambda|^{q-1} \sgn(\p^m_t \K_\lambda) \p^m_t \K [ \tfrac{1}{2} \p_\lambda \eta_x \Z + \tfrac{1}{2} \eta_x \Z_\lambda] \\
&  - q\sum^m_{j=2} {m\choose{j}}\int \Sigma^{-\delta m q} |\p^m_t \K_\lambda|^{q-1} \sgn(\p^m_t \K_\lambda)[ \p^{j-1}_t \eta_{xt} \p^{m+1-j}_t \K_\lambda +  \p^{j-1}_t \p_\lambda \eta_{xt} \p^{m+1-j}_t \K] \\
& + \alpha q \sum^m_{j=1} {m\choose{j}}\int \Sigma^{-\delta m q} |\p^m_t \K_\lambda|^{q-1} \sgn(\p^m_t \K_\lambda) [\p^{j-1}_t \Sigma_t \p^{m-j}_t \p_x \K_\lambda + \p^{j-1}_t \p_\lambda \Sigma_t \p^{m-j}_t \p_x \K] \\
&   - \tfrac{1}{2}q  \sum^m_{j=1} {m\choose{j}} \int \Sigma^{-\delta m q} |\p^m_t \K_\lambda|^{q-1} \sgn(\p^m_t \K_\lambda) [\p_\lambda \eta_x \p^j_t \Z \p^{m-j}_t \K + \eta_x \p^j_t \Z_\lambda \p^{m-j}_t \K + \eta_x \p^j_t \Z \p^{m-j}_t \K_\lambda]\\
&  - \tfrac{1}{2}q  \sum^m_{j=1} {m\choose{j}} \int \Sigma^{-\delta m q} |\p^m_t \K_\lambda|^{q-1} \sgn(\p^m_t \K_\lambda )[\p^{j-1}_t \p_\lambda (\eta_x \W)_t \p^{m-j}_t \K + \p^{j-1}_t (\eta_x \W)_t \p^{m-j}_t \K_\lambda]  \\
&  - \tfrac{1}{2}q \sum_{\substack{j_1+j_2+j_3 = m \\ j_1 \geq 1}} {m\choose{j_1 j_2 j_3}} \int \Sigma^{-\delta m q} |\p^m_t \K_\lambda |^{q-1} \sgn(\p^m_t \K_\lambda ) \p^{j_1-1}_t \p_\lambda \eta_{xt} \p^{j_2}_t \K \p^{j_3}_t \Z  \\
&  - \tfrac{1}{2}q \sum_{\substack{j_1+j_2+j_3 = m \\ j_1 \geq 1}} {m\choose{j_1 j_2 j_3}} \int \Sigma^{-\delta m q} |\p^m_t \K_\lambda |^{q-1} \sgn(\p^m_t \K_\lambda )\p^{j_1-1}_t \eta_{xt} \p^{j_2}_t \K_\lambda \p^{j_3}_t \Z \\
& - \tfrac{1}{2}q \sum_{\substack{j_1+j_2+j_3 = m \\ j_1 \geq 1}} {m\choose{j_1 j_2 j_3}} \int \Sigma^{-\delta m q} |\p^m_t \K_\lambda |^{q-1} \sgn(\p^m_t \K_\lambda )\p^{j_1-1}_t \eta_{xt} \p^{j_2}_t \K \p^{j_3}_t \Z_\lambda.
\end{align*}
Just like in the proof of Proposition \ref{prop:EE:infty}, applying \eqref{ineq:W} to the first term of the righthand side, multiplying through by $(\frac{(m+1)^2}{m! M C^m_t})^q$, adding $-\kappa m q \tilde{E}^{m,q}_k$ to both sides, and using the formulas from \S~\ref{sec:appendix:freqcomp} gives us
\begin{align*}
&\dot{\tilde{E}}^{m,q}_k  + \tfrac{1}{2} D_k \bigg[ \tfrac{(m+1)^2 \|\Sigma^{-\delta m} \p^m_t \K_\lambda\|_{L^q_x}}{m! M C^m_t} \bigg]^q  
\notag\\
& \leq [-mq \kappa + (4+B_z\eps)D_k + \alpha m q(B_z +\tfrac{3}{2\gamma} B_k) )\eps ] \tilde{E}^{m,q}_k 
\notag \\
&\quad  + q \bigg[ \tfrac{(m+1)^2 \|\Sigma^{-\delta m} \p^m_t \K_\lambda\|_{L^q_x}}{m! MC^m_t} \bigg]^{q-1}  \bigg[ \tfrac{(m+1)^2 \|\Sigma^{-\delta m} \p^{m+1}_t \K\|_{L^q_x}}{m! C^m_t} \bigg] 
\notag \\
&\quad + \tfrac{\alpha}{9} q \bigg[ \tfrac{(m+1)^2 \|\Sigma^{-\delta m} \p^m_t \K_\lambda\|_{L^q_x}}{m!M C^m_t} \bigg]^{q-1}  \bigg[ \tfrac{(m+1)^2 \|\Sigma^{-\delta m} \p^m_t \p_x \K\|_{L^q_x}}{m! C^m_t} \bigg] 
\notag \\
&\quad + q \bigg[ \tfrac{(m+1)^2 \|\Sigma^{-\delta m} \p^m_t \K_\lambda\|_{L^q_x}}{m! MC^m_t} \bigg]^{q-1} (\tfrac{3}{8}+ m\tfrac{1+\alpha}{2})B_k \eps 
\notag \\
&\quad + q \bigg[ \tfrac{(m+1)^2 \|\Sigma^{-\delta m} \p^m_t \K_\lambda\|_{L^q_x}}{m! MC^m_t} \bigg]^{q-1} \tfrac{15}{2}(1+\alpha)AB_z\eps^2 \sum^m_{j=2} \tfrac{(m+1)^2(m+1-j)}{j^3(m+2-j)^2} \notag \\
&\quad + q \bigg[ \tfrac{(m+1)^2 \|\Sigma^{-\delta m} \p^m_t \K_\lambda\|_{L^q_x}}{m! MC^m_t} \bigg]^{q-1} 40\alpha(1+ 2\tfrac{B_k}{B_z})AB_z\eps^2 \sum^m_{j=1} \tfrac{(m+1)^2(m+1-j)}{j^3(m+2-j)^2} \notag \\
&\quad + q \bigg[ \tfrac{(m+1)^2 \|\Sigma^{-\delta m} \p^m_t \K_\lambda\|_{L^q_x}}{m! MC^m_t} \bigg]^{q-1} \OO((1+\alpha)A B_z\eps^2) \\
&\quad + q \bigg[ \tfrac{(m+1)^2 \|\Sigma^{-\delta m} \p^m_t \K_\lambda\|_{L^q_x}}{m! MC^m_t} \bigg]^{q-1} \OO((1+\alpha) B^2_z\eps^2). 
\end{align*}
Since $m \leq 2n$, Corollary \ref{cor:apriori:est} and \eqref{ineq:C_t:bar} imply
\begin{equation} \label{ineq:K_t:m+1:EE}
\begin{aligned}
\tfrac{(m+1)^2 \|\Sigma^{-\delta m} \p^m_t \p_x \K\|_{L^\infty_x}}{m! C^m_t} 
& \leq \tfrac{10}{\alpha}\tfrac{(m+1)^3}{(m+2)^2} 3^\delta C_tB_k \eps 
\leq \tfrac{10}{\alpha}(m+1) \bar{C}_tB_k \eps, \\
\tfrac{(m+1)^2 \|\Sigma^{-\delta m} \p^{m+1}_t \K\|_{L^\infty_x}}{m! C^m_t} & \leq \tfrac{(m+1)^3}{(m+2)^2} 3^\delta C_tB_k \eps 
\leq (m+1) \bar{C}_tB_k \eps.
\end{aligned}
\end{equation}
Therefore, using \eqref{ineq:sum:10m} and the bounds from \S~\ref{sec:appendix:freqcomp} yields 
\begin{align*}
&\dot{\tilde{E}}^{m,q}_k  + \tfrac{1}{2} D_k \bigg[ \tfrac{(m+1)^2 \|\Sigma^{-\delta m} \p^m_t \K_\lambda\|_{L^q_x}}{m! M C^m_t} \bigg]^q  
\notag\\
& \leq [-mq \kappa + (4+B_z\eps)D_k + \alpha m q(B_z +\tfrac{3}{2\gamma} B_k) )\eps ] \tilde{E}^{m,q}_k 
\notag \\
&\quad  + (m+1)q \bigg[ \tfrac{(m+1)^2 \|\Sigma^{-\delta m} \p^m_t \K_\lambda\|_{L^q_x}}{m! MC^m_t} \bigg]^{q-1} [\tfrac{19}{9}\bar{C}_t + \tfrac{3}{8(m+1)}+ \tfrac{m}{m+1}\tfrac{1+\alpha}{2}(1+9A\eps) ] B_k \eps
\notag \\
&\quad  + mq \bigg[ \tfrac{(m+1)^2 \|\Sigma^{-\delta m} \p^m_t \K_\lambda\|_{L^q_x}}{m! MC^m_t} \bigg]^{q-1} (1+\alpha) AB_z\eps^2(75 + 440(1+2\tfrac{B_k}{B_z}) \tfrac{\alpha}{1+\alpha})
\notag \\
&\quad + q \bigg[ \tfrac{(m+1)^2 \|\Sigma^{-\delta m} \p^m_t \K_\lambda\|_{L^q_x}}{m! MC^m_t} \bigg]^{q-1} \OO((1+\alpha)A B_z\eps^2) 
\notag \\
&\quad + q \bigg[ \tfrac{(m+1)^2 \|\Sigma^{-\delta m} \p^m_t \K_\lambda\|_{L^q_x}}{m! MC^m_t} \bigg]^{q-1} \OO((1+\alpha) B^2_z\eps^2). 
\end{align*}
Since $A  > B_\lambda B_z \bar{C}_t$  and $B_\lambda > e^{31}$, it follows that if $(1+\alpha) B_z \eps$ is small enough we have
 \begin{equation*}
 [\tfrac{19}{9}\bar{C}_t + \frac{3}{8(m+1)}+ \tfrac{m}{m+1}\tfrac{1+\alpha}{2}(1+9A\eps) ]B_z < [\tfrac{19}{9}\bar{C}_t + \frac{3}{16}+ \tfrac{1+\alpha}{2} ]B_z+\tfrac{1}{200}A < \frac{1}{100} A.
 \end{equation*}
Therefore if $(1+\alpha)B_z \eps$ is small enough we have
\begin{align*}
&\dot{\tilde{E}}^{m,q}_k  + \tfrac{1}{2} D_k \bigg[ \tfrac{(m+1)^2 \|\Sigma^{-\delta m} \p^m_t \K_\lambda\|_{L^q_x}}{m!  C^m_t} \bigg]^q  
\notag\\
& \leq [-mq \kappa + (4+B_z\eps)D_k + \alpha m q(B_z +\tfrac{3}{2\gamma} B_k) )\eps ] \tilde{E}^{m,q}_k 
\notag \\
&\quad  + (m+1)q \bigg[ \tfrac{(m+1)^2 \|\Sigma^{-\delta m} \p^m_t \K_\lambda\|_{L^q_x}}{m! C^m_t} \bigg]^{q-1} \tfrac{1}{100}A.
\end{align*}
It now follows from the upper bound \eqref{ineq:D_k} on $D_k$ and the definition of $\kappa$ that if $B_z \eps$ is sufficiently small we get
\begin{align}
& \dot{\tilde{E}}^{m,q}_k 
+ \bigg[ \tfrac{1}{2} D_k - \tfrac{1}{10}(q-1)(m+1)]\bigg] \bigg[ \tfrac{(m+1)^2 \|\Sigma^{-\delta m} \p^m_t \K_\lambda\|_{L^q_x}}{m! MC^m_t} \bigg]^{q}  
\notag \\
& \leq -mq\max(1,\alpha) \tilde{E}^{m,q}_k + (m+1) \big(\tfrac{A}{10}\big)^q.
\label{ineq:energy:K:n:stab}
\end{align}

Analogous computations for $\tilde{\E}^{m,q}_z$ give us
\begin{align*}
\dot{\tilde{\E}}^{m,q}_z
& = D_z \int \Sigma^{-\delta m q} |\p^m_t \Z_\lambda|^q \eta_x \W \\
& + [\tfrac{1+\alpha}{2} -q(\tfrac{1-\alpha}{2} m + 1 + \alpha)] \int \Sigma^{-\delta m q} \eta_x |\p^m_t \Z_\lambda |^q \Z \\
& + \tfrac{\alpha}{2\gamma}[-1 + q( \tfrac{1}{2} + m(\delta-1)] \int \Sigma^{-\delta m q} \eta_x |\p^m_t \Z_\lambda|^q \Sigma \K \\
& - \tfrac{\alpha}{4\gamma} q \int \Sigma^{-\delta m q} |\p^m_t \Z_\lambda|^{q-1} \sgn(\p^m_t \Z_\lambda) \Sigma (\eta_x \W) \p^m_t \K_\lambda \\
& - q \int \Sigma^{-\delta m q} |\p^m_t \Z_\lambda|^{q-1} \sgn(\p^m_t \Z_\lambda) \p_\lambda \eta_x \p^{m+1}_t \Z \\
& + 2\alpha q \int \Sigma^{-\delta m q} |\p^m_t \Z_\lambda|^{q-1} \sgn(\p^m_t \Z_\lambda) \Sigma_\lambda \p^m_t \p_x \Z \\
& -m q \int \Sigma^{-\delta m q} |\p^m_t \Z_\lambda|^{q-1} \sgn(\p^m_t \Z_\lambda) \p_\lambda \eta_{xt} \p^m_t \Z \\
& - \tfrac{\alpha}{4\gamma} q \int \Sigma^{-\delta m q} |\p^m_t \Z_\lambda|^{q-1} \sgn(\p^m_t \Z_\lambda)  \p^m_t \K [ \Sigma_\lambda \eta_x \W + \Sigma (\eta_x \W)_\lambda ] \\
& -\tfrac{1-\alpha}{2} q \int \Sigma^{-\delta m q} |\p^m_t \Z_\lambda|^{q-1} \sgn(\p^m_t \Z_\lambda) (\eta_x \W)_\lambda \p^m_t \Z \\
& - q \sum^m_{j=2} {m\choose{j}}\int \Sigma^{-\delta m q} |\p^m_t \Z_\lambda|^{q-1} \sgn(\p^m_t \Z_\lambda) [ \p^{j-1}_t \p_\lambda \eta_{xt} \p^{m+1-j}_t \Z + \p^{j-1}_t \eta_{xt} \p^{m+1-j}_t \Z_\lambda] \\
& +2\alpha q \sum^m_{j=2} {m\choose{j}}\int \Sigma^{-\delta m q} |\p^m_t \Z_\lambda|^{q-1} \sgn(\p^m_t \Z_\lambda) [\p^{j-1}_t \Sigma_t \p^{m-j}_t \p_x \Z + \p^{j-1}_t \p_\lambda \Sigma_t \p^{m-j}_t \p_x \Z_\lambda] \\
& + (\hdots) 
\end{align*}
where $D_z$ is the same constant defined in \eqref{def:D_z} and the remaining terms referred to as ``$\hdots$'' are bounded  via
\begin{equation*}
\bigg[ \frac{(m+1)^2}{m! M C^m_t}\bigg]^q |(\hdots)| \leq q \bigg[ \frac{(m+1)^2 \|\Sigma^{-\delta m} \p^m_t \Z_\lambda\|_{L^q_x}}{m! MC^m_t} \bigg]^{q-1} \OO( (1+ \alpha)(A+B_z)B_z \eps^2).
\end{equation*}
Therefore, performing the same types of computations done for $\tilde{E}^{m,q}_k$ gives us
\begin{align}
& \dot{\tilde{E}}^{m,q}_z  - \tfrac{\alpha}{\gamma}\bigg[ \tfrac{(m+1)^2 \|\Sigma^{-\delta m} \p^m_t \K_\lambda\|_{L^q_x}}{m! M C^m_t} \bigg]^{q} 
\notag \\
&\quad + \bigg[\tfrac{1}{2} D_z - (q-1)[(m+1)(\tfrac{1}{10}+\tfrac{\alpha}{\gamma}) \bigg] \bigg[ \tfrac{(m+1)^2 \|\Sigma^{-\delta m} \p^m_t \Z_\lambda\|_{L^q_x}}{m! MC^m_t} \bigg]^{q} \notag \\
& \leq -mq\max(1,\alpha) \tilde{E}^{m,q}_z + (m+1)(\tfrac{1}{10}A \eps)^q.
\label{ineq:energy:Z:n:stab}
\end{align}
Using our bounds \eqref{ineq:D_k} and \eqref{ineq:D_z}, we get 
\begin{align*}
\tfrac{1}{2}D_k & > \tfrac{\alpha}{\gamma} + \tfrac{1}{10}(q-1)(m+1), \\
\tfrac{1}{2}D_z & > (q-1)(m+1)(\tfrac{1}{10}+\tfrac{\alpha}{\gamma}),
\end{align*}
so we can add \eqref{ineq:energy:K:n:stab} and \eqref{ineq:energy:Z:n:stab} together to arrive at
\begin{align*}
\dot{\tilde{E}}^{m,q}_k + \dot{\tilde{E}}^{m,q}_z & \leq -mq \max(1,\alpha) ( \tilde{E}^{m,q}_k + \tilde{E}^{m,q}_z) + 2(m+1)(\tfrac{1}{10}A\eps)^q.
\end{align*}
Using ODE comparison (see  Lemma \ref{lem:ODE}), raising both sides to the $\frac{1}{q}$, applying \eqref{ineq:bsh:t=0:stab}, and sending $ q\rightarrow \infty$ gives us
\begin{align*}
\max\bigg( \tfrac{(m+1)^2 \|\Sigma^{-\delta m} \p^m_t \K_\lambda\|_{L^\infty_x}}{m! M C^m_t}, \tfrac{(m+1)^2 \|\Sigma^{-\delta m} \p^m_t \Z_\lambda\|_{L^\infty_x}}{m!  MC^m_t} \bigg) 
& \leq \max\bigg( \tfrac{1}{5}\tfrac{\eps}{2^m}, \tfrac{1}{10}A \eps \bigg) = \tfrac{1}{10}A\eps.
\end{align*}
Since $A  < \frac{4}{3}\bar{C}_t B_\lambda B_z$, we conclude that \eqref{ineq:EE:dt:stab} is true for $t \in [0,T]$. This closes our bootstrap argument.\end{proof}

\subsection{Bounds on the remaining derivatives}
\label{sec:ineq:dx:stab}
The results of this section are all analogs of the results of \S~\ref{sec:ineq:dx}. Their proofs are also completely analogous. For this reason, they will be omitted.

\begin{proposition}\label{prop:ineq:dx:stab}
Suppose the hypotheses of Proposition \ref{prop:EE:stab} are satisfied, and additionally $\bar{C}_x$ is a constant satisfying
\begin{equation*}
\bar{C}_x \gg 1,
\qquad
\alpha\bar{C}_x  \gg \bar{C}_t.
\end{equation*}
Under these conditions, if $\eps$ is small enough such that
\begin{equation*}
(1+\alpha) B_\lambda \bar{C}_t B_z \eps \ll 1,
\end{equation*}
then the following bounds hold for each $t \in [0,T_*)$ and each $\beta$ with $|\beta| \leq 2n$: 
\begin{align*}
&\frac{(|\beta|+1)^2 \|\p^\beta \K_\lambda\|_{L^\infty_x}}{|\beta|! M \bar{C}^{\beta_x}_x \bar{C}^{\beta_t}_t} \leq B_\lambda \bar{C}_tB_z \eps, 
&&\frac{(|\beta|+1)^2 \|\p^\beta \Z_\lambda\|_{L^\infty_x}}{|\beta|! M \bar{C}^{\beta_x}_x \bar{C}^{\beta_t}_t} \leq B_\lambda \bar{C}_tB_z \eps, \\
&\frac{(|\beta|+1)^2 \|\p^\beta \p_\lambda \Sigma_t\|_{L^\infty_x}}{|\beta|! M \bar{C}^{\beta_x}_x \bar{C}^{\beta_t}_t} \leq 8 \alpha B_\lambda \bar{C}_tB_z \eps,
&&  \\
&\frac{(|\beta|+1)^2 \|\p^\beta \p_\lambda \eta_x\|_{L^\infty_x}}{|\beta|! M \bar{C}^{\beta_x}_x \bar{C}^{\beta_t}_t}  \leq 1, 
&&\frac{(|\beta|+1)^2 \|\p^\beta (\eta_x \W)_\lambda \|_{L^\infty_x}}{|\beta|! M \bar{C}^{\beta_x}_x \bar{C}^{\beta_t}_t} \leq  \tfrac{3}{4}, \\
&\frac{(|\beta|+1)^2 \|\p^\beta \p_\lambda \Sigma_x\|_{L^\infty_x}}{|\beta|! M \bar{C}^{\beta_x}_x \bar{C}^{\beta_t}_t} \leq 1.
&&
\end{align*}
\end{proposition}

\begin{corollary}\label{cor:ineq:dx:stab}
Under the hypotheses of the previous proposition, we also have that
\begin{align*}
\frac{(|\beta|+1)^2 \|\p^\beta \Sigma_\lambda\|_{L^\infty_x}}{|\beta|! M\bar{C}^{\beta_x}_x \bar{C}^{\beta_t}_t} & \leq \tfrac{1}{9},\\
\frac{(|\beta|+1)^2 \|\p^\beta \p_\lambda (\eta_x \W)_t\|_{L^\infty_x}}{|\beta|! M\bar{C}^{\beta_x}_x \bar{C}^{\beta_t}_t} & \leq  B_\lambda\bar{C}_tB_z \eps, \\
\frac{(|\beta|+1)^2 \|\p^\beta \p_\lambda \eta_{xt}- \tfrac{1+\alpha}{2} \p^\beta (\eta_x \W)_\lambda\|_{L^\infty_x}}{|\beta|! M\bar{C}^{\beta_x}_x \bar{C}^{\beta_t}_t} & \leq 185\max(1,\alpha) B_\lambda \bar{C}_tB_z \eps.
\end{align*}
for all $t \in [0,T_*)$ and $\beta$ with $|\beta| \leq 2n$.
\end{corollary}

\begin{corollary}\label{cor:eta_x:approx:stab}
Under the hypotheses of the previous propositions, we have that
\begin{subequations}\label{approx:eta_x:stab}
\begin{align}
\|\p^i_x \p_\lambda \eta_{xt} - \tfrac{1+\alpha}{2} \p^{i+1}_x \tilde{w}_0\|_{L^\infty_x} & \leq 190\max(1,\alpha)  \tfrac{i!}{(i+1)^2} M \bar{C}^i_x \bar{C}_t B_\lambda B_z \eps,
\\
\|\p^i_x \p_\lambda \eta_x - \tfrac{1+\alpha}{2} t\p^{i+1}_x \tilde{w}_0\|_{L^\infty_x} & \leq 190\max(1,\alpha) t \tfrac{i!}{(i+1)^2} M \bar{C}^i_x \bar{C}_t B_\lambda B_z \eps
\end{align}
\end{subequations}
 for all $t \in [0, T_*), i= 0, \hdots, 2n$. 
\end{corollary}


\section{Solution continuity with respect to initial data}
\label{sec:cont}

In this section, we will show that  solutions with initial data in $\mathcal A_n(\eps, C_0)$ depend continuously on their initial data with respect to the $(W^{2n+2,\infty}(\TT))^3$ topology. For $m = 0, \hdots, 2n+1$ define the seminorm\footnote{ It is a seminorm because it doesn't detect constant changes in $k_0$; however our system \eqref{eq:euler:RV} does not detect constant changes in $k_0$ either.} on $(W^{2n+2, \infty}(\TT))^3$
\begin{equation}\label{def:seminorm}
\|(w_0,z_0,k_0)\|_{m, C_0} : = \max_{0 \leq i \leq m}\bigg(\|w_0\|_{L^\infty_x}, \|z_0\|_{L^\infty_x}, \frac{\|\p^{i+1}_x w_0\|_{L^\infty_x}}{i! C^i_0}, \frac{\|\p^{i+1}_x z_0\|_{L^\infty_x}}{i! C^i_0}, \frac{\|\p^{i+1}_x k_0\|_{L^\infty_x}}{i! C^i_0}\bigg).
\end{equation}
Note that $\|(w_0,z_0,k_0)\|_{m, C_0}  \leq \|(w_0,z_0,k_0)\|_{(W^{m+1, \infty}(\TT))^3}$. We will estimate the difference of two solutions in terms of the distance between their initial data with respect to these seminorms.

The methods and results of this section will be completely analogous to those in the previous two sections. For this reason, we will only provide a sketch of the proofs.

For the rest of this section, fix $(w_0,z_0, k_0), (\tilde{w}_0, \tilde{z}_0, \tilde{k}_0) \in \mathcal A_n(\eps, C_0)$ and let the corresponding solutions of \eqref{system:lagrange} with initial data $(w_0,z_0, k_0)$ and $(\tilde{w}_0, \tilde{z}_0, \tilde{k}_0)$ be $\K, \Z, \eta_x, \Sigma, \eta_x \W$ and $\tilde{\K}, \tilde{\Z}, \tilde{\eta}_x, \tilde{\Sigma}, \tilde{\eta}_x \tilde{\W}$ respectively. Let their respective blowup times be $T_*$ and $\tilde{T}_*$. Adopt the notation
\begin{align*}
\mu_m & : = \|(w_0-\tilde{w}_0, z_0-\tilde{z}_0, k_0-\tilde{k}_0)\|_{m, C_0} & m = 0, \hdots, 2n.
\end{align*}

It follows from \eqref{system:lagrange} that the difference of the two solutions satisfies
\begin{subequations}\label{system:diff}
\begin{align}\label{id:Sigma_t:diff}
(\Sigma-\tilde{\Sigma})_t & = \alpha (\Sigma-\tilde{\Sigma})[ -\Z + \tfrac{1}{2\gamma}(\Sigma+\tilde{\Sigma}) \K] -\alpha(\Z-\tilde{\Z}) \tilde{\Sigma} + \tfrac{\alpha}{2\gamma} \tilde{\Sigma}^2 (\K-\tilde{\K}), \\
\label{id:eta_xt:diff}
\eta_{xt} -\tilde{\eta}_{xt} & = \tfrac{1+\alpha}{2}[ (\eta_x\W) - (\tilde{\eta}_x\tilde{\W})] + \tfrac{1-\alpha}{2} (\eta_x -\tilde{\eta}_x)\Z + \tfrac{1-\alpha}{2} \tilde{\eta}_x (\Z-\tilde{\Z}) \\
& + \tfrac{\alpha}{2\gamma} (\eta_x -\tilde{\eta}_x) \Sigma \K + \tfrac{\alpha}{2\gamma} \tilde{\eta}_x (\Sigma-\tilde{\Sigma}) \K + \tfrac{\alpha}{2\gamma} \tilde{\eta}_x \tilde{\Sigma} (\K -\tilde{\K}), \notag \\
\label{id:W_t:diff}
(\eta_x \W - \tilde{\eta}_x \tilde{\W})_t & = \tfrac{\alpha}{4\gamma} (\Sigma - \tilde{\Sigma}) \K (\eta_x \W + \eta_x \Z) +\tfrac{\alpha}{4\gamma}\tilde{\Sigma} (\K - \tilde{\K}) ( \eta_x \W + \eta_x \Z) \\
&  + \tfrac{\alpha}{4\gamma} \tilde{\Sigma} \tilde{\K} (\eta_x \W - \tilde{\eta}_x \tilde{\W}) + \tfrac{\alpha}{4\gamma} \tilde{\Sigma} \tilde{\K} (\eta_x-\tilde{\eta}_x) \Z + \tfrac{\alpha}{4\gamma} \tilde{\Sigma} \tilde{\K} \tilde{\eta}_x (\Z - \tilde{\Z}), \notag \\
\label{id:Sigma_x:diff}
(\Sigma- \tilde{\Sigma})_x & = \tfrac{1}{2}(\eta_x \W - \tilde{\eta}_x \tilde{\W}) - \tfrac{1}{2}(\eta_x - \tilde{\eta}_x) \Z  - \tfrac{1}{2} \tilde{\eta}_x (\Z- \tilde{\Z}) 
\\
& + \tfrac{1}{2\gamma} (\eta_x -\tilde{\eta}_x) \Sigma \K + \tfrac{1}{2\gamma} \tilde{\eta}_x (\Sigma-\tilde{\Sigma}) \K + \tfrac{1}{2\gamma}\tilde{\eta}_x \tilde{\Sigma} (\K - \tilde{\K}). \notag
\end{align}
\end{subequations}

We will also replace \eqref{id:K_t} and \eqref{id:Z_t} with the identities
\begin{subequations}\label{id:KZ:dt:diff}
\begin{align}
\label{id:K_t:diff}
\eta_x \p_t (\K - \tilde{\K}) & = \alpha \Sigma \p_x( \K - \tilde{\K}) - \tfrac{1}{2}(\K - \tilde{\K}) (\eta_x \W) - \tfrac{1}{2} \eta_x (\K - \tilde{\K}) \Z \\
& - (\eta_x-\tilde{\eta}_x) \p_t \tilde{\K} + \alpha (\Sigma - \tilde{\Sigma}) \p_x \tilde{\K} \notag \\
& -\tfrac{1}{2} \tilde{\K}( \eta_x \W - \tilde{\eta}_x \tilde{\W}) - \tfrac{1}{2} (\eta_x - \tilde{\eta}_x) \tilde{\K} \Z - \tfrac{1}{2} \tilde{\eta}_x \tilde{\K} (\Z- \tilde{\Z}), \notag \\
\label{id:Z_t:diff}
\eta_x\p_t(\Z-\tilde{\Z}) & = 2\alpha \Sigma \p_x(\Z-\tilde{\Z}) - \tfrac{1-\alpha}{2} \eta_x \W (\Z-\tilde{\Z}) - \tfrac{1+\alpha}{2} \eta_x \Z (\Z-\tilde{\Z}) + \tfrac{\alpha}{4\gamma} \eta_x \Sigma \K (\Z - \tilde{\Z}) \\
& - (\eta_x-\tilde{\eta}_x) \p_t \tilde{\Z} + 2\alpha (\Sigma - \tilde{\Sigma}) \p_x \tilde{\Z} \notag \\
& -\tfrac{\alpha}{4\gamma} (\Sigma-\tilde{\Sigma}) \K \eta_x \W -\tfrac{\alpha}{4\gamma}\tilde{\Sigma}(\K-\tilde{\K}) \eta_x \W -\tfrac{\alpha}{4\gamma}\tilde{\Sigma}\tilde{\K} (\eta_x \W - \tilde{\eta}_x \tilde{\W}) \notag \\
& -\tfrac{1-\alpha}{2} (\eta_x \W - \tilde{\eta}_x \tilde{\W}) \tilde{\Z} \notag \\
& - \tfrac{1+\alpha}{2} \eta_x (\Z-\tilde{\Z}) \tilde{\Z} - \tfrac{1+\alpha}{2}(\eta_x -\tilde{\eta}_x) \tilde{\Z}^2 \notag \\
& + \tfrac{\alpha}{4\gamma} \eta_x \Sigma (\K - \tilde{\K}) \tilde{\Z} + \tfrac{\alpha}{4\gamma} \eta_x (\Sigma-\tilde{\Sigma}) \tilde{\K} \tilde{\Z} + \tfrac{\alpha}{4\gamma} (\eta_x -\tilde{\eta}_x) \tilde{\Sigma} \tilde{\K} \tilde{\Z}. \notag 
\end{align}
\end{subequations}
Multiplying these identities by $\Sigma^{-1}$ and rearranging gives us
\begin{subequations}
\begin{align}
\label{id:K_x:diff}
\alpha \p_x ( \K - \tilde{\K}) & = \eta_x \Sigma^{-1}\p_t(\K - \tilde{\K}) + (\eta_x-\tilde{\eta}_x) \Sigma^{-1} \p_t \tilde{\K} - \alpha \Sigma^{-1}(\Sigma - \tilde{\Sigma}) \p_x \tilde{\K} \\
& + \tfrac{1}{2}\Sigma^{-1}(\K - \tilde{\K}) (\eta_x \W) + \tfrac{1}{2} \eta_x \Sigma^{-1}(\K - \tilde{\K}) \Z \notag \\
& +\tfrac{1}{2} \Sigma^{-1}\tilde{\K}( \eta_x \W - \tilde{\eta}_x \tilde{\W}) + \tfrac{1}{2} \Sigma^{-1}(\eta_x - \tilde{\eta}_x) \tilde{\K} \Z + \tfrac{1}{2}\Sigma^{-1} \tilde{\eta}_x \tilde{\K} (\Z- \tilde{\Z}), \notag \\
\label{id:Z_x:diff}
2\alpha \p_x (\Z-\tilde{\Z}) & = \eta_x\Sigma^{-1}\p_t (\Z-\tilde{\Z})  +(\eta_x-\tilde{\eta}_x) \Sigma^{-1}\p_t \tilde{\Z} - 2\alpha \Sigma^{-1}(\Sigma - \tilde{\Sigma}) \p_x \tilde{\Z} \\
& + \tfrac{1-\alpha}{2} \Sigma^{-1}\eta_x \W (\Z-\tilde{\Z}) + \tfrac{1+\alpha}{2}\Sigma^{-1} \eta_x \Z (\Z-\tilde{\Z}) - \tfrac{\alpha}{4\gamma} \Sigma^{-1}\eta_x \Sigma \K (\Z - \tilde{\Z}) \notag \\
& +\tfrac{1-\alpha}{2} \Sigma^{-1}(\eta_x \W - \tilde{\eta}_x \Sigma^{-1} \tilde{\W}) \tilde{\Z} \notag \\
& + \tfrac{1+\alpha}{2} \eta_x \Sigma^{-1}(\Z-\tilde{\Z}) \tilde{\Z} + \tfrac{1+\alpha}{2}(\eta_x -\tilde{\eta}_x) \Sigma^{-1}\tilde{\Z}^2 \notag \\
& - \tfrac{\alpha}{4\gamma} \eta_x  (\K - \tilde{\K}) \tilde{\Z} - \tfrac{\alpha}{4\gamma} \eta_x \Sigma^{-1} (\Sigma-\tilde{\Sigma}) \tilde{\K} \tilde{\Z} - \tfrac{\alpha}{4\gamma} (\eta_x -\tilde{\eta}_x) \Sigma^{-1} \tilde{\Sigma} \tilde{\K} \tilde{\Z}. \notag 
\end{align}
\end{subequations}
Notice that the system above doesn't feature any terms with the difference $\Sigma^{-1} - \tilde{\Sigma}^{-1}$, so we do not need equations for $(\Sigma^{-1} - \tilde{\Sigma}^{-1})_t$ or $(\Sigma^{-1} - \tilde{\Sigma}^{-1})_x$.

If the constants $\C_x, \C_t$, and $\eps$ satisfy the same hypotheses as in Propositions \ref{prop:t=0}, \ref{prop:t=0:stab}, one can prove an analog of these propositions for the difference of our two solutions. In particular, at time $t=0$ one obtains
\begin{align*}
\frac{(|\beta|+1)^2\|\p^\beta (\K-\tilde{\K})\|_{L^\infty_x}}{|\beta|! \C^{\beta_x}_x \C^{\beta_t}_t} & \leq 3\mu_{|\beta|}, &
\frac{(|\beta|+1)^2\|\p^\beta (\Z-\tilde{\Z})\|_{L^\infty_x}}{|\beta|! \C^{\beta_x}_x \C^{\beta_t}_t} & \leq 3\mu_{|\beta|}  
\end{align*}
for all $|\beta| \leq 2n$.

Next, one performs a priori estimates similar to those in \S~\ref{sec:apriori}, \ref{sec:apriori:stab}: fix $T \in [0, T_*\wedge \tilde{T}_*]$, a function $C_t : [0,T_* \wedge \tilde{T}_*) \rightarrow \R^+$, and a constant $\delta \geq 0$. Suppose that $A \geq 1$ is a constant such that
\begin{subequations}
\begin{align}
\frac{(m+1)^2 \|\Sigma^{-\delta m} \p^m_t (\K-\tilde{\K})\|_{L^\infty_x}}{m! C^m_t} & \leq A\mu_m, &
\frac{(m+1)^2 \|\Sigma^{-\delta m} \p^m_t (\Z-\tilde{\Z})\|_{L^\infty_x}}{m! C^m_t} & \leq A \mu_m
\end{align}
\end{subequations}
for all  $t \in [0,T], \: 0 \leq m \leq 2n.$ Suppose also that
\begin{subequations} \label{hypotheses:apriori:cont}
\begin{align}
\frac{(m+1)^2 \|\Sigma^{-\delta m} \p^m_t \K\|_{L^\infty_x}}{m! C^m_t} & \leq B_k \eps, &
\frac{(m+1)^2 \|\Sigma^{-\delta m} \p^m_t \Z\|_{L^\infty_x}}{m! C^m_t} & \leq B_z \eps, \\
\frac{(m+1)^2 \|\Sigma^{-\delta m} \p^m_t \tilde{\K}\|_{L^\infty_x}}{m! C^m_t} & \leq B_k \eps, &
\frac{(m+1)^2 \|\Sigma^{-\delta m} \p^m_t \tilde{\Z}\|_{L^\infty_x}}{m! C^m_t} & \leq B_z \eps, 
\end{align}
\end{subequations}
for all  $t \in [0,T], \: 0 \leq m \leq 2n+1.$ These hypotheses \eqref{hypotheses:apriori:cont} allow us to apply the estimates of \S~\ref{sec:apriori} to our two solutions, the same way that \eqref{ineq:apriori:hyp:previous} is used in \S~\ref{sec:apriori:stab}. If $B_z \eps \ll 1$ and $C_t \gg (1+\alpha)3^\delta$ for all $t \in [0,T]$, one obtains
\begin{align}
\|\Sigma - \tilde{\Sigma}\|_{L^\infty_x} & \leq (1+ 16A) \mu_0, & & \\
\|\eta_x \W - \tilde{\eta}_x \tilde{\W}\|_{L^\infty_x} & \leq (3+2A) \mu_0, & \|\eta_x-\tilde{\eta}_x\|_{L^\infty_x} & \leq (5+10A) \mu_0.
\end{align}
and 
\begin{align*}
\frac{(m+1)^2 \|\Sigma^{-\delta m} \p^m_t (\Sigma-\tilde{\Sigma})_t\|_{L^\infty_x}}{m! C^m_t} & \leq 8\alpha A \mu_m, &
& \frac{(m+1)^2 \|\Sigma^{-\delta m} \p^m_t (\eta_{xt}-\tilde{\eta}_{xt}) \|_{L^\infty_x}}{m! C^m_t}  \leq 6(1+\alpha)A \mu_m, \\
\frac{(m+1)^2 \|\Sigma^{-\delta m} \p^m_t (\eta_x\W-\tilde{\eta}_x \tilde{\W})_t\|_{L^\infty_x}}{m! C^m_t}  &\leq \tfrac{5}{8} A \mu_m, &
& \\
\frac{(m+1)^2 \|\Sigma^{-\delta m} \p^m_t (\eta_x-\tilde{\eta}_x)\|_{L^\infty_x}}{m! C^m_t} & \leq 15 A\mu_m,  &
& \frac{(m+1)^2 \|\Sigma^{-\delta m} \p^m_t( \eta_x \W-\tilde{\eta}_x\tilde{\W})\|_{L^\infty_x}}{m! C^m_t}  \leq 5A\mu_m ,\\
\frac{(m+1)^2 \|\Sigma^{-\delta m} \p^m_t (\Sigma-\tilde{\Sigma})\|_{L^\infty_x}}{m! C^m_t} & \leq 17 A \mu_m, &
& \frac{(m+1)^2 \|\Sigma^{-\delta m} \p^m_t (\Sigma-\tilde{\Sigma})_x\|_{L^\infty_x}}{m! C^m_t}  \leq 7A\mu_m,   
\end{align*}
for all  $t \in [0,T], \: 0 \leq m \leq 2n$, as well as
\begin{align}
\frac{(m+1)^2 \|\Sigma^{-\delta m} \p^{m-1}_t (\K-\tilde{\K})_x\|_{L^\infty_x}}{m! C^m_t} & \leq \tfrac{10}{\alpha} A \mu_m,  &
\frac{(m+1)^2 \|\Sigma^{-\delta m} \p^{m-1}_t (\Z-\tilde{\Z})_x \|_{L^\infty_x}}{m! C^m_t} & \leq \tfrac{5}{\alpha}A \mu_m , 
\end{align}
for all  $t \in [0,T], \: 1 \leq m \leq 2n$.

Next, we use these a priori estimates to perform $L^q$ energy estimates on the time derivatives of $(\K-\tilde{\K})$ and $(\Z-\tilde{\Z})$, similar to those in Propositions \ref{prop:EE:infty} and \ref{prop:EE:stab}.

\begin{proposition}\label{prop:EE:cont}
Define the constant
 \begin{equation}\label{def:B:Delta}
 B_d : = 2\cdot 6^{\frac{3}{\min(1,\alpha)}} e^{31}.
 \end{equation}
  Let $\delta$ and $\kappa$ satisfy \eqref{def:delta} and let $C_t$ be defined by \eqref{def:C_t}. Let $\bar{C}_t$ be the constant defined by \eqref{def:C_t:bar}. If $C_t(0)$ satisfies
\begin{align*}
C_t(0) & \gg (1+\alpha) 3^\delta, & 
C_t(0) & \gg \max(1,\alpha) 2^\delta C_0,
\end{align*}
and $\eps$ satisfies
\begin{align*}
 (1+\alpha)(1+B_z) \eps^{\frac{1}{2}} & \ll 1, & 3^{\frac{3}{\min(1,\alpha)}} B_d (\bar{C}_t + 1+\alpha)B_z \eps^{\frac{1}{2}} & \ll 1, 
\end{align*}
then
\begin{subequations}\label{ineq:prop:EE:cont}
\begin{align}
\frac{(m+1)^2 \|\Sigma^{-\delta m} \p^m_t (\K-\tilde{\K}) \|_{L^\infty_x}}{m! C^m_t} & \leq B_d \mu_m, &
\frac{(m+1)^2 \|\Sigma^{-\delta m} \p^m_t (\Z-\tilde{\Z}) \|_{L^\infty_x}}{m! C^m_t} & \leq  B_d \mu_m, 
\end{align}
\end{subequations}
for all $t \in [0,T_*\wedge \tilde{T}_*)$ and $0 \leq m \leq 2n$.
\end{proposition}

The proof of this proposition is analogous to the proof of Proposition \ref{prop:EE:infty} and Proposition \ref{prop:EE:stab}. Fix $T \in [0, T_*\wedge \tilde{T}_*)$ and pick a constant $A$ satisfying $B_d < A < 2B_d$. Our bootstrap hypothesis will be that 
\begin{align*}
\frac{(m+1)^2 \|\Sigma^{-\delta m} \p^m_t (\K-\tilde{\K})\|_{L^\infty_x}}{m! C^m_t} & \leq A \mu_m, &
\frac{(m+1)^2 \|\Sigma^{-\delta m} \p^m_t (\Z-\tilde{\Z})\|_{L^\infty_x}}{m! C^m_t} & \leq A \mu_m,
\end{align*}
for all  $t \in [0,T], \: 0 \leq m \leq 2n.$ Proposition \ref{prop:EE:infty} implies that \eqref{hypotheses:apriori:cont} holds with the $\delta, \kappa$ and $C_t$ we have chosen. Therefore, all of the a priori estimates we have derived apply on $[0,T]$ for this choice of $A$.

Just like in the proof of Proposition \ref{prop:EE:infty},  if we define $\C_t : = \frac{1}{2}2^{-\delta} C_t(0)$, then our hypotheses on $C_t(0)$ allow us to apply our estimates at time $t=0$ and get
\begin{equation}\label{ineq:bsh:t=0:cont}
\begin{aligned}
\frac{(m+1)^2 \|\Sigma^{-\delta m} \p^m_t (\K-\tilde{\K})\|_{L^\infty_x}}{m! C^m_t} & \leq \tfrac{3\mu_m}{2^m} , &
\frac{(m+1)^2 \|\Sigma^{-\delta m} \p^m_t (\Z-\tilde{\Z})\|_{L^\infty_x}}{m! C^m_t} & \leq \tfrac{3\mu_m}{2^{m-1}} ,
\end{aligned}
\end{equation}
for $m=0, \hdots, 2n, t = 0$. Therefore, our bootstrap hypothesis is true for $T = 0$.

For $0 \leq m \leq 2n$ and $1< q < \infty$ define the energies
\begin{align*}
\hat{\E}^{m,q}_k & : =\begin{cases} \int_\TT \Sigma^{-\frac{3}{\min(1,\alpha)}} \eta_x | (\K-\tilde{\K})|^q \: dx & m=0 \\ \int_\TT \Sigma^{-\delta qm} \eta_x |\p^m_t (\K-\tilde{\K})|^q \: dx & 1 \leq m \leq 2n \end{cases}, \\
\hat{\E}^{m,q}_z & : =\begin{cases} \int_\TT \Sigma^{-\frac{3}{\min(1,\alpha)}} \eta_x | (\Z-\tilde{\Z})|^q \: dx & m=0 \\ \int_\TT \Sigma^{-\delta qm} \eta_x |\p^m_t (\Z-\tilde{\Z})|^q \: dx & 1 \leq m \leq 2n \end{cases} ,
\end{align*}
and for $1 \leq m \leq 2n$, $1 < q < \infty$, define the energies
\begin{align*}
\hat{E}^{m,q}_k & : =\frac{(m+1)^{2q} \hat{\E}^{m,q}_k}{(m!)^qC^{mq}_t} , &
\hat{E}^{m,q}_z & : = \frac{(m+1)^{2q} \hat{\E}^{m,q}_z}{(m!)^q C^{mq}_t}  .
\end{align*}
We prove \eqref{ineq:prop:EE:cont} in the case $m=0$ and $1 \leq m \leq 2n$ separately. In both cases, just like in the proofs of Propositions \ref{prop:EE:infty}, \ref{prop:EE:stab}, we take the time derivative of $\hat{\E}^{m,q}_k$ and $\hat{\E}^{m,q}_z$, simplify and rearrange using \eqref{id:K_t:diff} and \eqref{id:Z_t:diff}, and use \eqref{ineq:W} to produce damping terms which make  $\hat{\E}^{m,q}_k+\hat{\E}^{m,q}_z$ satisfy 
\begin{align*}
&\dot{\hat{\E}}^{0,q}_k + \dot{\hat{\E}}^{0,q}_z \leq 15\max(1,\alpha) q (\hat{\E}^{0,q}_k + \hat{\E}^{0,q}_z) + 2(A\mu_0)^q\OO( 3^\delta(\bar{C}_t + 1+\alpha)B_z \eps^{\frac{1}{2}})^q, \\
&\dot{\hat{E}}^{m,q}_k+ \dot{\hat{E}}^{m,q}_z \leq -mq\max(1,\alpha) (\hat{E}^{m,q}_k + \hat{E}^{m,q}_z) + 2(A\mu_m)^q[ m\OO( [\bar{C}_t B_z  +(1+\alpha) B_z] \eps^{\frac{1}{2}})^q + \OO((1+\alpha) B_z \eps^{\frac{1}{2}})^q].
\end{align*}
From here the bootstrap argument closes by using ODE comparison and \eqref{ineq:bsh:t=0:cont}, provided that $\eps$ is small enough.

After proving Proposition \ref{prop:EE:cont}, one can deduce bounds on the rest of the derivatives of the difference via an induction argument, just like in Proposition \ref{prop:ineq:dx}: if $\bar{C}_x$ is a constant satisfying
\begin{align*}
\bar{C}_x & \gg 1, &
\alpha\bar{C}_x & \gg \bar{C}_t,
\end{align*}
$\eps$ satisfies $(1+\alpha)B_z \eps \ll 1$,  and $C_t(0)$ satisfies $C_t(0) \gg 1$, then the following bounds hold for all $t \in [0,T_*\wedge \tilde{T}_*)$ and $|\beta| \leq 2n$:
\begin{align*}
&
\frac{(|\beta|+1)^2 \inorm{\p^\beta (\K-\tilde{\K})}_{L^\infty_x}}{|\beta|! \bar{C}^{\beta_x}_x \bar{C}^{\beta_t}_t}  \leq B_d \mu_{|\beta|}, 
&&
\frac{(|\beta|+1)^2 \inorm{\p^\beta (\Z-\tilde{\Z})}_{L^\infty_x}}{|\beta|! \bar{C}^{\beta_x}_x \bar{C}^{\beta_t}_t}  \leq B_d \mu_{|\beta|}, 
\\
&
\frac{(|\beta|+1)^2 \inorm{\p^\beta (\Sigma-\tilde{\Sigma})_t}_{L^\infty_x}}{|\beta|! \bar{C}^{\beta_x}_x \bar{C}^{\beta_t}_t} \leq 8\alpha  B_d \mu_{|\beta|},
&&
\\
&
\frac{(|\beta|+1)^2 \inorm{\p^\beta (\eta_x-\tilde{\eta}_x)}_{L^\infty_x}}{|\beta|! \bar{C}^{\beta_x}_x \bar{C}^{\beta_t}_t} \leq 70 B_d \mu_{|\beta|}, 
&&
\frac{(|\beta|+1)^2 \inorm{\p^\beta (\eta_x \W-\tilde{\eta}_x \tilde{\W})}_{L^\infty_x}}{|\beta|! \bar{C}^{\beta_x}_x \bar{C}^{\beta_t}_t}  \leq  7 B_d \mu_{|\beta|}, 
\\
&
\frac{(|\beta|+1)^2 \inorm{\p^\beta (\Sigma-\tilde{\Sigma})_x}_{L^\infty_x}}{|\beta|! \bar{C}^{\beta_x}_x \bar{C}^{\beta_t}_t} \leq 120 B_d \mu_{|\beta|}, 
&&
\end{align*}

Our efforts culminate in the following bounds:
\begin{subequations} \label{cor:ineq:dx:cont}
\begin{align}
\frac{(|\beta|+1)^2 \inorm{\p^\beta (\eta_{xt}-\tilde{\eta}_{xt})}_{L^\infty_x}}{|\beta|! \bar{C}^{\beta_x}_x \bar{C}^{\beta_t}_t} & \leq (1+\alpha) 38 B_d \mu_{|\beta|}
\end{align}
for all $t \in [0, T_* \wedge \tilde{T}_*), |\beta| \leq 2n$, and, as an immediate corollary,
\begin{align}
\| \p^i_x(\eta_x-\tilde{\eta}_x) \|_{L^\infty_x} & \leq \tfrac{i!}{(i+1)^2} 38 (1+\alpha) t  \bar{C}^i_x B_d \mu_i
\end{align}
\end{subequations}
for all $t \in [0, T_* \wedge \tilde{T}_*), i = 0, \hdots, 2n$. We will use these bounds \eqref{cor:ineq:dx:cont}  in \S~\ref{sec:Lip}.


\section{Finite-codimension Banach manifolds of initial data}
\label{FCBM}

Recall that our aim (see Theorem \ref{thm:HR} below) is to construct a codimension-$(2n-2)$ Banach manifold of initial data $(w_0,z_0, k_0) \in (W^{2n+2,\infty}(\TT))^3$ for which the unique classical solutions $(w,z,k)$ to the corresponding Cauchy problem \eqref{eq:euler:RV}  form $C^{0, \frac{1}{2n+1}}$ pre-shocks. In this section, we will show that the initial data $(w_0, z_0, k_0)$ in a particular neighborhood of $(W^{2n+2,\infty}(\TT))^3$ for which the system
\begin{equation*}
\eta_x(x_*, T_*) = 0, \qquad \p_x \eta_x(x_*, T_*) = 0, \qquad \hdots \qquad \p^{2n-1}_x \eta_x(x_*, T_*) = 0,
\end{equation*}
has a solution $(x_*,T_*)$ is a codimension-$(2n-2)$ Banach manifold. In \S~\ref{sec:thm} we will show that  solutions $(w,z,k)$ with initial data in this manifold form cusps resembling $-y^{\frac{1}{2n+1}}$ at their first singularities.

\subsection{Assumptions on the initial data}
Fix a positive integer $n$ and a constant $C_0 \geq 3$. Now fix a function $\bar{w}_0 \in W^{2n+2, \infty}(\TT)$ satisfying
\begin{subequations}\label{def:w_0:bar}
\begin{align}
\bar{w}_0(x)  &= \tfrac{5}{2} -x + \tfrac{x^{2n+1}}{2n+1}, & |x-0| \leq \tfrac{1}{C_0}, \\
\bar{w}'_0(x) & \geq -1 + C^{-2n}_0,   & |x-0| \geq \tfrac{1}{C_0}, \\
\frac{\|\p^{i+1}_x \bar{w}_0\|_{L^\infty_x}}{i!} & \leq C^i_0, & \forall \: i = 0, \hdots, 2n+1.
\end{align}
\end{subequations}
Note that
\begin{align*}
\p^{i+1}_x \bar{w}_0(x) = -\delta_{i0} + \mathbbm{1}_{\{i \leq 2n\}}i!{2n\choose{i}}x^{2n-i}, \qquad |x-0| \leq \tfrac{1}{C_0},
\end{align*}
so that for $i = 1, \hdots, 2n+1$ we have
\begin{align*}
\frac{\|\p^{i+1}_x \bar{w}_0\|_{L^\infty_x(-\frac{1}{C_0}, \frac{1}{C_0})}}{i! C^i_0} \leq  \mathbbm{1}_{\{i \leq 2n\}} {2n\choose{i}} C^{-2n}_0 < 1.
\end{align*}
Therefore, our assumptions on $\bar{w}_0$ are consistent.

With $\bar{w}_0$ chosen as above, define 
\begin{equation}
\label{eq:cal:B:n:def}
\mathcal B_n(\eps, C_0) \subset (W^{2n+2,\infty}(\TT))^3
\end{equation}
to be the set of all $(w_0,z_0, k_0) \in (W^{2n+2, \infty}(\TT))^3$ satisfying
\begin{align*}
\|w_0 - \bar{w}_0\|_{L^\infty_x} & < \eps, & \\
\|z_0\|_{L^\infty_x} & < \eps, & \\
\frac{\|\p^{i+1}_x(w_0- \bar{w}_0)\|_{L^\infty_x}}{i!} & < C^i_0 \eps, & \forall \: i = 0, \hdots, 2n+1, \\
\frac{\|\p^{i+1}_x z_0\|_{L^\infty_x}}{i!} & < C^i_0\eps, & \forall \: i = 0, \hdots, 2n+1, \\
\frac{\|\p^{i+1}_x k_0\|_{L^\infty_x}}{i!} & < C^i_0 \eps, & \forall \: i = 0, \hdots, 2n+1.
\end{align*}
Since $\|\bar{w}_0'\|_{L^\infty_x} = 1$, we know that $\frac{3}{2} \leq \bar{w}_0 \leq \frac{7}{2}$ everywhere in $\TT$. Therefore, if $\eps < \frac{1}{4}$ we have that $\frac{1}{2} < \sigma_0 < 2$ so that $\mathcal B_n(\eps, C_0) \subset \mathcal A_n(\eps, C_0)$. For the remainder of this paper, assume that $\eps < \frac{1}{4}$ so that this inclusion always holds.

\subsection{Preliminary Estimates along the Fast Acoustic Characteristics}
\label{sec:FCBM:prelim}

For the remainder of this section, suppose that $\eps, \bar{C}_t,$ and $\bar{C}_x$ are chosen such that the conclusions of \S~\ref{sec:HOE} hold for all initial data $(w_0,z_0,k_0) \in \mathcal A_n(\eps, C_0)$. Furthermore, we will restrict our attention to only solutions with initial data $(w_0, z_0, k_0) \in \mathcal B_n(\eps, C_0)$.

\begin{lemma}\label{lem:eta_x}
For all $i=0, \hdots, 2n+1, t \in [0, T_*)$, the following approximate identities are true:
\begin{subequations}\label{approx:eta:B}
\begin{align}
\label{approx:eta_xt:B}
\|\p^i_x \eta_{xt}-\tfrac{1+\alpha}{2} \p^{i+1}_x \bar{w}_0\|_{L^\infty_x} & \leq 22 \max(1,\alpha) \tfrac{i!}{(i+1)^2} \bar{C}^i_x B_z \eps, \\
\label{approx:eta_x:B}
\|\p^i_x \eta_x-(\delta_{i0} +\tfrac{1+\alpha}{2} t\p^{i+1}_x \bar{w}_0)\|_{L^\infty_x} & \leq 22 \max(1,\alpha) t \tfrac{i!}{(i+1)^2} \bar{C}^i_x B_z \eps.
\end{align}
\end{subequations}
As a result, for $|x-0| \leq \tfrac{1}{C_0}, t \in [0, T_*), i=0, \hdots, 2n+1$ we have
\begin{subequations}
\begin{align}
\label{approx:eta_xt:C}
\bigg\vert \p^i_x \eta_{xt}(x,t) - \tfrac{1+\alpha}{2} \bigg[-\delta_{i0} + \mathbbm{1}_{\{i \leq 2n\}}i!{2n\choose{i}}x^{2n-i}\bigg] \bigg\vert & \leq  22 \max(1,\alpha) \tfrac{i!}{(i+1)^2} \bar{C}^i_x B_z \eps, \\
\label{approx:eta_x:C}
\bigg\vert\p^i_x \eta_{x}(x,t) -\bigg[ \delta_{i0} + \tfrac{1+\alpha}{2}t\bigg[-\delta_{i0} + \mathbbm{1}_{\{i \leq 2n\}}i!{2n\choose{i}}x^{2n-i}\bigg] \bigg\vert & \leq 22 \max(1,\alpha)t \tfrac{i!}{(i+1)^2} \bar{C}^i_x B_z \eps. 
\end{align}
\end{subequations}
\end{lemma}

\begin{proof}
The bounds \eqref{approx:eta:B} follow immediately from Corollary \ref{cor:eta_x:approx}, the definition of $\mathcal B_n(\eps, C_0)$, and the fact that $\bar{C}_x \geq e^2 C_0$ and $B_z > 1$.
\end{proof}


\begin{lemma}\label{lem:LB:0}
If
\begin{equation*}
C_0^{2n} B_z \eps \ll 1,
\end{equation*}
then
\begin{equation}\label{LB:eta_x:0}
\eta_x(x,t) \geq \tfrac{1}{2}C^{-2n}_0 \qquad \forall \: |x-0| \geq \tfrac{1}{C_0}, t \in [0,T_*].
\end{equation}
It follows from this that $\eta_x(\cdot, T_*)$ must have at least one zero in $|x-0| < \frac{1}{C_0}$.
\end{lemma}

\begin{proof}
For $|x-0| \geq \frac{1}{C_0}$, using \eqref{approx:eta_x:B} with $i=0$,  \eqref{def:w_0:bar} , and \eqref{eq:T_*} gives us
\begin{align*}
\eta_x & \geq 1 + \tfrac{1+\alpha}{2} t[ -1 + C^{-2n}_0 + \OO( B_z \eps) ] \\
& \geq 1 + \tfrac{1+\alpha}{2}T_*[ -1 + C^{-2n}_0 + \OO( B_z \eps) ]  \\
&=1 + (1+ \OO(B_z \eps))[ -1 + C^{-2n}_0 + \OO( B_z \eps) ] 
= C^{-2n}_0 + \OO(B_z \eps).
\end{align*}
Therefore, if $C^{2n}_0 B_z \eps \ll 1$ then we obtain \eqref{LB:eta_x:0}. 

Since $\min_\TT \eta_x(\cdot, T_*) = 0$, $\eta_x(\cdot, T_*)$ must have a zero in $|x-0| <\frac{1}{C_0}$.
\end{proof}

Now, define the function $\tilde{\eta}_x : \TT \times [0, \infty) \rightarrow \R$,
\begin{equation}\label{def:extension}
\tilde{\eta}_x(x,t) : = \begin{cases} \eta_x(x,t) & t\leq T_* \\ \eta_x(x,T_*) + (t-T_*)\eta_{xt}(x,T_*) & t \geq T_* \end{cases} .
\end{equation}
$\tilde{\eta}_x$ extends $\eta_x$ to a $C^{2n,1}_x C^1_t$ function on all of $\TT \times [0, \infty)$ in a manner analogous to the way that for Burgers equation the family of lines $1+ tw'_0(x)$ extends $\eta_x$ to all of $\TT \times [0, \infty)$ (see \S~\ref{sec:burgers}).

\begin{lemma}\label{lem:extension:approx}
For all $i=0, \hdots, 2n, t  \in [0, \infty)$, the following approximate identities are true:
\begin{subequations}\label{approx:eta:B'}
\begin{align}
\label{approx:eta_xt:B'}
\|\p^i_x \tilde{\eta}_{xt}-\tfrac{1+\alpha}{2} \p^{i+1}_x \bar{w}_0\|_{L^\infty_x} & \leq 22 \max(1,\alpha) \tfrac{i!}{(i+1)^2} \bar{C}^i_x B_z \eps, \\
\label{approx:eta_x:B'}
\|\p^i_x \tilde{\eta}_x-(\delta_{i0} +\tfrac{1+\alpha}{2} \p^{i+1}_x \bar{w}_0)\|_{L^\infty_x} & \leq 22 \max(1,\alpha) t \tfrac{i!}{(i+1)^2} \bar{C}^i_x B_z \eps.
\end{align}
\end{subequations}
As a result, for $|x-0| \leq \tfrac{1}{C_0}, t \geq 0, i=0, \hdots, 2n$, we have that
\begin{subequations}\label{approx:eta:C'}
\begin{align}
\label{approx:eta_xt:C'}
\bigg\vert \p^i_x \tilde{\eta}_{xt}(x,t) - \tfrac{1+\alpha}{2} \bigg[-\delta_{i0} + \mathbbm{1}_{\{i \leq 2n\}}i!{2n\choose{i}}x^{2n-i}\bigg] \bigg\vert & \leq  22 \max(1,\alpha) \tfrac{i!}{(i+1)^2} \bar{C}^i_x B_z \eps, \\
\label{approx:eta_x:C'}
\bigg\vert\p^i_x \tilde{\eta}_{x}(x,t) -\bigg[ \delta_{i0} + \tfrac{1+\alpha}{2}t\bigg[-\delta_{i0} + \mathbbm{1}_{\{i \leq 2n\}}i!{2n\choose{i}}x^{2n-i}\bigg] \bigg\vert & \leq 22 \max(1,\alpha)t \tfrac{i!}{(i+1)^2} \bar{C}^i_x B_z \eps. 
\end{align}
\end{subequations}
In particular, if $B_z \eps \ll 1$ we obtain the bounds
\begin{align}\label{ineq:eta_xt:UL}
-\tfrac{2}{3}(1+\alpha) & \leq \tilde{\eta}_{xt}(x,t) \leq - \tfrac{1+\alpha}{2} && \forall \: |x-0| \leq \frac{1}{C_0}.
\end{align}
\end{lemma}

\begin{proof}
It is immediate from the definition of $\tilde{\eta}_x$ that
\begin{equation*}
\tilde{\eta}_{xt}(x,t) = \begin{cases} \eta_{xt}(x,t) & t \leq T_* \\ \eta_{xt}(x,T_*) & t \geq T_* \end{cases} .
\end{equation*}
Taking $\p^i_x$ of $\tilde{\eta}_{xt}$ and applying \eqref{approx:eta_xt:B} and \eqref{approx:eta_xt:C} gives us \eqref{approx:eta_xt:B'} and \eqref{approx:eta_xt:C'} respectively.

Integrating \eqref{approx:eta_xt:B'} and \eqref{approx:eta_xt:C'} in time gives us \eqref{approx:eta_x:B'} and \eqref{approx:eta_x:C'} respectively. \eqref{ineq:eta_xt:UL} is an immediate consequence of \eqref{approx:eta_xt:C'} if $44B_z \eps <\frac{1}{9}$.
\end{proof}

\begin{lemma}\label{lem:LB:2n}
If $\bar{C}_x^{2n} B_z \eps \ll 1$, then for all $|x-0| \leq \frac{1}{C_0}$ and $t \geq \frac{2}{3} \frac{2}{1+\alpha}$, we have the lower bound
\begin{equation}\label{ineq:LB:2n}
\p^{2n}_x\tilde{\eta}_x(x,t) \geq \tfrac{(2n)!}{2}.
\end{equation}
It follows that $\eta_x(\cdot, T_*)$ can have at most $n$ zeros in $\TT$. In particular, when $n=1$, $\eta_x$ always has a unique zero in $\TT \times [0, T_*]$.
\end{lemma}

\begin{proof}
Apply \eqref{approx:eta_x:C} with $i=2n$ to obtain
\begin{equation*}
\p^{2n}_x \tilde{\eta}_x 
= \tfrac{1+\alpha}{2} t(2n)![1+ \OO(\bar{C}_x^{2n} B_z \eps)] 
\geq \tfrac{2}{3}(2n)![1+ \OO(\bar{C}_x^{2n} B_z \eps)].
\end{equation*}
If $\bar{C}_x^{2n} B_z \eps$ is small enough, the lower bound \eqref{ineq:LB:2n} follows.

It now follows from \eqref{ineq:T_*} that $\p^{2n}_x \eta_x(x, T_*) \geq \frac{(2n)!}{2}$ for all $|x-0| \leq \frac{1}{C_0}$. Therefore $\p^{2n-1}_x \eta_x( \cdot, T_*)$ is a strictly increasing function on the interval $[-\frac{1}{C_0}, \frac{1}{C_0}]$, so it has no critical points and at most one zero in $[-\frac{1}{C_0}, \frac{1}{C_0}]$.
This implies that $\p^{2n-2}_x \eta_x( \cdot, T_*)$ has at most one critical point and at most two zeros in $[-\frac{1}{C_0}, \frac{1}{C_0}]$. Continuing inductively in this manner, $\eta_x(\cdot, T_*)$ has at most $2n-1$ critical points and at most $2n$ zeros in $[-\frac{1}{C_0}, \frac{1}{C_0}]$. However, we know that $\eta_x(\cdot, T_*) \geq 0$ everywhere, so it follows from \eqref{LB:eta_x:0} that all zeros of $\eta_x(\cdot, T_*)$ in $\TT$ must themselves be critical points in the interval $|x-0| \leq \frac{1}{C_0}$. In between any two zeros of $\eta_x(\cdot, T_*)$ must be at least one other critical point, so $\eta_x(\cdot, T_*)$ can have at most $n$ zeros.
\end{proof}

\begin{proposition}\label{prop:vanish:apriori}
Suppose $(x_*,t_*) \in \TT \times [0, \infty)$ is a point such that $|x_* - 0| \leq \frac{1}{C_0}$ and
\begin{equation}\label{system:eta_x:vanish:t_*}
\tilde{\eta}_x(x_*, t_*) = \p_x \tilde{\eta}_x(x_*, t_*) = \hdots = \p^{2n-1}_x \tilde{\eta}_x(x_*,t_*) = 0.
\end{equation}
Then $t_* = T_*$, and $x_*$ is the unique zero of $\eta_x$ in $\TT \times [0,T_*]$. Furthermore,  $x_*$  satisfies the bound
\begin{subequations} \label{ineq:x_*:apriori}
\begin{align}
\label{ineq:x_*:1}
|x_*|^{2n} & \lesssim B_z \eps.
\end{align}
\end{subequations}
\end{proposition}

\begin{proof}
Since $\tilde{\eta}_x(x_*, t_*) =0$, we know $t_* \geq T_*$. Taylor expanding $\tilde{\eta}_x(\cdot, t_*)$ about $\x$  and using \eqref{ineq:LB:2n} gives us
\begin{equation} \label{ineq:Taylor:apriori}
\tilde{\eta}_x(x,t_*) 
\geq \tfrac{1}{2} (x-x_*)^{2n},
\qquad  \forall |x-0| \leq \tfrac{1}{C_0}.
\end{equation}
In particular, $\tilde{\eta}_x(\cdot, t_*) \geq 0$. It now follows from \eqref{approx:eta_xt:C} that
\begin{align*}
\eta_x(x,T_*) & = \tilde{\eta}_x(x,t_*) -(t_*-T_*)\eta_{xt}(x,T_*) \\
& \geq \tfrac{1}{2} (x-x_*)^{2n} -(t_*-T_*) \tfrac{1+\alpha}{2}\bigg[ -1 + x^{2n} + \OO(B_z \eps) \bigg] \\
& = \tfrac{1}{2} (x-x_*)^{2n} +(t_*-T_*) \tfrac{1+\alpha}{2}\bigg[ 1 - x^{2n} + \OO(B_z \eps) \bigg] \\
& \geq \tfrac{1}{2} (x-x_*)^{2n} +(t_*-T_*) \tfrac{1+\alpha}{2}\bigg[ 1 - C_0^{-2n} + \OO(B_z \eps) \bigg] \\
& \geq \tfrac{1}{2} (x-x_*)^{2n} +(t_*-T_*) \tfrac{1+\alpha}{4}
\end{align*}
for all $|x-0| \leq \frac{1}{C_0}$. We know from Lemma \ref{lem:LB:0} that $\eta_x(\cdot, T_*)$ must have a zero and all zeros of $\eta_x(\cdot, T_*)$ must be in the interval  $\{|x-0| \leq \frac{1}{C_0}\}$. Since $t_* \geq T_*$, this is only possible if $t_* = T_*$ and the zero must be $x= x_*$.

Since $t_* = T_*$, plugging $x=0$ into \eqref{ineq:Taylor:apriori} and using \eqref{approx:eta_x:C} and \eqref{approx:T_*} gives us
\begin{align*}
|x_*|^{2n} & \leq 2|\eta_x(0, T_*)| \\
& \leq 2|\tfrac{1+\alpha}{2}T_*-1| + 44\max(1,\alpha)T_*B_z \eps \\
& \lesssim B_z \eps.
\end{align*}
\end{proof}

Now define the function $G: \TT \times [0, \infty) \rightarrow \R^2$,
\begin{equation}\label{def:G}
G(x,t) : = \begin{bmatrix} \tfrac{1}{(2n)!}\p^{2n-1}_x \tilde{\eta}_x(x,t) \\  -\tfrac{2}{1+\alpha}\tilde{\eta}_x(x,t)  \end{bmatrix} .
\end{equation}

\begin{proposition}\label{prop:G}
If  
\begin{align*}
(1+\alpha) C_0  < \bar{C}_x,  \qquad
 \bar{C}^{2n}_x B_z \eps  \ll 1, 
\end{align*}
then there exists a unique point $(\x, \T)$ in a ball of radius $\frac{1}{3(1+\alpha) C_0}$ around the point $(0, \frac{2}{1+\alpha})$ such that
\begin{equation*}
G(\x,\T) = (0,0).
\end{equation*}
\end{proposition}

\begin{proof}
For $|x-0| \leq \frac{1}{C_0}$, the formulas \eqref{approx:eta:C'} give us
\begin{align}
DG(x,t) & = \begin{pmatrix} \tfrac{1}{(2n)!}\p^{2n}_x \tilde{\eta}_x(x,t) & \tfrac{1}{(2n)!}\p^{2n-1}_x \tilde{\eta}_{xt}(x,t) \\
-\tfrac{2}{1+\alpha}\p_x \tilde{\eta}_x(x,t) & -\tfrac{2}{1+\alpha}\tilde{\eta}_{xt}(x,t) \end{pmatrix} \notag \\ \label{eq:DG}
& = \begin{pmatrix} \tfrac{1+\alpha}{2}t[1+ \tfrac{\OO(\bar{C}^{2n}_x B_z \eps)}{(2n+1)^2}] & \tfrac{1+\alpha}{2}[x+ \frac{\OO(\bar{C}^{2n}_x B_z \eps)}{(2n)^3\bar{C}_x}] \\
t[-2nx^{2n-1} + \frac{\OO(\bar{C}^{2n}_x B_z \eps)}{\bar{C}^{2n-1}_x}] & 1 -x^{2n} + \tfrac{\OO(\bar{C}^{2n}_x B_z \eps)}{\bar{C}^{2n}_x} \end{pmatrix}.
\end{align}

For $(x,t)$ in a ball of radius $R : = \frac{1}{3(1+\alpha) C_0}$ around $(0, \frac{2}{1+\alpha})$, the fact that $C_0 \geq 3$ implies that
\begin{align*}
|\tfrac{1+\alpha}{2}t-1| & \leq \tfrac{1+\alpha}{2}R \leq \tfrac{1}{18},    & 2n |x|^{2n-1} t & \leq 2nR^{2n-1}(\tfrac{2}{1+\alpha} + R) \\
\tfrac{1+\alpha}{2}|x| & \leq \tfrac{1+\alpha}{2}R \leq \tfrac{1}{18},   & & = 12 C_0 n(R^2)^n + 2n R^{2n} \\
|x|^{2n} & \leq \tfrac{1}{81(1+\alpha)^2}, & & \leq 12C_0 R^2 + 2R^2 \\
& & & \leq \tfrac{38}{81(1+\alpha)^2}.
\end{align*}
Therefore,
\begin{align*}
\bigg\vert \Id - DG(x,t) \bigg\vert & < \tfrac{1}{2} + \OO\big( \bar{C}^{2n}_x B_z \eps \big) && \forall \: (x,t) \in B_R((0, \tfrac{2}{1+\alpha})).
\end{align*}
It follows that if $\bar{C}^{2n}_x B_z \eps$ is small enough we have $\|\Id - DG \|_{L^\infty_{x,t}(B_R(0, \frac{2}{1+\alpha}))} \leq \frac{2}{3}$.

Plugging $(x,t) = (0, \frac{2}{1+\alpha})$ into the formulas \eqref{approx:eta:C'} yields
\begin{align*}
\big\vert \tfrac{1}{(2n)!} \p^{2n-1}_x \tilde{\eta}_x(0, \tfrac{2}{1+\alpha}) \big\vert & \leq \frac{44 \bar{C}^{2n}_x B_z \eps}{(2n)^3 \bar{C}_x},
\\
\big\vert \tfrac{2}{1+\alpha} \tilde{\eta}_x(0, \tfrac{2}{1+\alpha}) \big\vert & \leq 88 B_z \eps.
\end{align*}
Since $\bar{C}_x > (1+\alpha) C_0$, it follows that if $\bar{C}^{2n}_x B_z \eps$ is small enough we have $|G(0, \frac{2}{1+\alpha})| < \frac{1}{3} R$. Our result now follows from Lemma \ref{lem:IFT:1}.
\end{proof}

Notice that one always has the inequality $\T \geq T_*$: indeed, if  $ \: \T \leq T_*$ then
\begin{equation*}
\eta_x(\x,\T) = \tilde{\eta}_x(\x,\T) = 0,
\end{equation*}
which means that $\T = T_*$.

\begin{corollary}\label{cor:x:ring}
If $(\x,\T)$ is the solution of $G = (0,0)$ given by Proposition \ref{prop:G}, then $\x, \T$ satisfy the bounds
\begin{align}\label{ineq:T:ring}
|\tfrac{1+\alpha}{2}\T-1| & < \tfrac{1}{18},
\\
\label{ineq:x:ring}
|\x| & \leq \tfrac{93}{(2n)^3} \bar{C}^{2n-1}_x B_z \eps.
\end{align}
It follows that if $n \geq 2$ and $\bar{C}^{2n}_x B_z \eps \ll 1$  then 
\begin{subequations}
\begin{align}
\label{approx:eta_xt:D}
|\p^i_x \tilde{\eta}_{xt}(\x,\T)| & \leq \tfrac{1+\alpha}{2} \tfrac{50}{(2n)} (2n)! \bar{C}^{2n-2}_x B_z \eps, \\
\label{approx:eta_x:D}
|\p^i_x \tilde{\eta}_{x}(\x,\T)| & \leq  \tfrac{53}{(2n)} (2n)! \bar{C}^{2n-2}_x B_z \eps ,
\end{align}
for $1 \leq i \leq 2n-2$.
\end{subequations}
\end{corollary}

\begin{proof}
The inequality \eqref{ineq:T:ring} is immediate from Proposition \ref{prop:G} and the fact that $C_0 \geq 3$. Since $\p^{2n-1}_x \tilde{\eta}_x(\x,\T) = 0$, \eqref{ineq:LB:2n}  gives us
\begin{align*}
 \big\vert \p^{2n-1}_x \tilde{\eta}_x (0, \T) \big\vert 
= \bigg\vert \int^0_{\x} \p^{2n}_x\tilde{\eta}_x(x,\T) \: dx \bigg\vert 
\geq \tfrac{(2n)!}{2}|\x|.
\end{align*}
Plugging $i=2n-1, x = 0, t= \T$, into \eqref{approx:eta_x:C'} now yields
\begin{align*}
|\x| & \leq \tfrac{44}{(2n)^3}\max(1,\alpha) \T \bar{C}^{2n-1}_x B_z \eps.
\end{align*}
Our bound \eqref{ineq:x:ring} now follows from \eqref{ineq:T:ring}.

Now suppose $n \geq 2$. Since $|\x| < 1$, we have
\begin{align*}
\max_{1 \leq i \leq 2n-2} \tfrac{(2n)!}{(2n-i)!} |\x|^{2n-i}   = \tfrac{(2n)!}{2} |\x|^2 \leq \tfrac{(2n)!}{2}\big( \tfrac{93}{(2n)^3}\bar{C}^{2n-1}_x B_z \eps)^2.
\end{align*}
For $1 \leq i \leq 2n-2$ the bound \eqref{approx:eta_xt:C'} now implies
\begin{align*}
|\p^i_x \tilde{\eta}_{xt}(\x, \T)| & \leq \tfrac{1+\alpha}{2} (2n)! \bar{C}^{2n-1}_x B_z \eps \bigg[ \tfrac{(93)^2}{2(2n)^3} \bar{C}^{2n-1}_x B_z \eps + 44\tfrac{(2n-4)!}{(2n)! \bar{C}_x} \bigg].
\end{align*}
Therefore, if $\bar{C}^{2n}_x B_z \eps$ is small enough we get \eqref{approx:eta_xt:D}.  The bound \eqref{approx:eta_x:D} follows from analogous computations using \eqref{approx:eta_x:C'}.
\end{proof}

Now for $n \geq 2$ define the function $f_n:\mathcal B_n(\eps, C_0) \rightarrow \R^{2n-2},$
\begin{equation}\label{def:f}
f_n(w_0,z_0,k_0) : = \begin{bmatrix} \p_x \tilde{\eta}_x(\x, \T) \\ \p^2_x \tilde{\eta}_x(\x, \T) \\ \vdots  \\ \p^{2n-2}_x \tilde{\eta}_x(\x,\T) \end{bmatrix} .
\end{equation}
The zero set of $f_n$ is precisely the initial data in $\mathcal B_n(\eps, C_0)$ for which the flow $\eta$ forms an initial singularity of the form 
\begin{equation}\label{system:eta_x:vanish:T_*}
\eta_x(x_*, T_*) = \p_x \eta_x(x_*, T_*) = \hdots = \p^{2n-1}_x \eta_x(x_*,T_*) = 0,
\end{equation}
at a point $x_* \in \TT$, and whenever such a root $x_*$ exists it must be the only root of $\eta_x(\cdot, T_*)$. To see this, if $f_n(w_0,z_0,k_0) = 0$ then it follows from the definitions of $G$ and $(\x,\T)$ that $(x_*, t_*) = (\x,\T)$ solves \eqref{system:eta_x:vanish:t_*}, and Proposition \ref{prop:vanish:apriori} implies that $(x_*, t_*) = (\x,\T)$ is the unique root of $\eta_x$ in $\TT \times[0,T_*]$ . Conversely, if $x_* \in \TT$ is a point solving \eqref{system:eta_x:vanish:T_*} then Lemma \ref{lem:LB:0} implies that $|x_*-0| < \frac{1}{C_0}$ and Proposition \ref{prop:vanish:apriori} therefore implies that $x_*$ is the only root of $\eta_x(\cdot, T_*)$ and also satisfies the bound \eqref{ineq:x_*:1}. Since we are assuming that $\bar{C}^{2n}_x B_z \eps$ is small and $(1+\alpha) C_0 < \bar{C}_x$, as in the premise of Proposition \ref{prop:G}, it follows that $|x_*| < \frac{1}{3(1+\alpha)C_0}$ and therefore $(x_*,T_*) = (\x,\T)$.

We will now quantify the regularity of $f_n$ so that we can apply the implicit function theorem and conclude that the zero set of $f_n$ is the graph of a Lipschitz function from a codimension-$(2n-2)$ subspace of $(W^{2n+2, \infty}(\TT))^3$ into $\R^{2n-2}$.


\subsection{Stability of solutions}\label{sec:Lip}

Now suppose further that $\bar{C}_x, \bar{C}_t$ are chosen large enough and $\eps$ is chosen small enough such that the hypotheses of \S~\ref{sec:cont} are satisfied, and therefore the bounds \eqref{cor:ineq:dx:cont} apply for all initial data in $\mathcal B_n(\eps, C_0)$.

\begin{proposition}\label{prop:Lip}
The blowup time, $T_*$, and the extensions $\tilde{\eta}_x$ depend on the initial data in way that is Lipschitz with respect to the $W^{2n+2, \infty}$ norm on $\mathcal B_n(\eps, C_0)$: i.e. if $\|\cdot\|_{m, C_0}$ is the seminorm on $(W^{2n+2, \infty}(\TT))^3$ defined by \eqref{def:seminorm} and $(w^1_0,z^1_0, k^1_0), (w^2_0,z^2_0, k^2_0) \in \mathcal B_n(\eps, C_0)$ we have
\begin{align}
\label{ineq:T_*:Lip}
|T_*(w^1_0,z^1_0, k^1_0) - T_*(w^2_0,z^2_0, k^2_0)| & \leq \tfrac{140}{1+\alpha} B_d \| (w^1_0 - w^2_0, z^1_0-z^2_0, k^1_0-k^2_0)\|_{0, C_0}, 
\end{align}
and if $\bar{C}_t B_z \eps \ll 1$ then for all $(x,t) \in \TT \times [0, \infty)$, $i=0, \hdots, 2n-1$ we have
\begin{subequations}
\begin{align}
\label{ineq:eta_xt:Lip}
\hspace{-5mm} |\p^i_x\tilde{\eta}_{xt}(x,t, w^1_0, z^1_0, k^1_0)-\p^i_x\tilde{\eta}_{xt}(x,t, w^2_0, z^2_0, k^2_0)| & \leq 45\frac{(1+\alpha) i! \bar{C}^i_x B_d}{i+1}\| (w^1_0 - w^2_0, z^1_0-z^2_0, k^1_0-k^2_0)\|_{i, C_0}, 
\\
\label{ineq:eta_x:Lip}
|\p^i_x\tilde{\eta}_{x}(x,t, w^1_0, z^1_0, k^1_0)-\p^i_x\tilde{\eta}_{x}(x,t, w^2_0, z^2_0, k^2_0)| & \leq 45\frac{(1+\alpha) t i! \bar{C}^i_x B_d}{i+1}\| (w^1_0 - w^2_0, z^1_0-z^2_0, k^1_0-k^2_0)\|_{i, C_0} .
\end{align}
\end{subequations}
Here $\tilde{\eta}_x(\cdot, \cdot, w^j_0, z^j_0, k^j_0)$ is the function $\tilde{\eta}_x$ corresponding to the solution with initial data $(w^j_0, z^j_0, k^j_0)$.
\end{proposition}

\begin{proof}
Let us adopt the notation $T^j_* : = T_*(w^j_0,z^j_0, k^j_0)$ for $j=1,2$  and $\mu_i : =\| (w^1_0 - w^2_0, z^1_0-z^2_0, k^1_0-k^2_0)\|_{i, C_0}$. Suppose without loss of generality that $T^1_* \leq T^2_*$. There exists a point $x^1_* \in \TT$ with $|x^1_* - 0| < \frac{1}{C_0}$ such that
\begin{equation*}
\eta_x(x^1_*, T^1_*, w^1_0, z^1_0, k^1_0) = 0.
\end{equation*}
Since $T^2_* \geq T^1_*$, it now follows from \eqref{cor:ineq:dx:cont} and \eqref{ineq:eta_xt:UL} that
\begin{equation*}
0 \leq \eta_x(x^1_*, T^2_*, w^2_0, z^2_0, k^2_0) \leq 70B_d \mu_0 - \tfrac{1+\alpha}{2}(T^2_*-T^1_*).
\end{equation*}
The inequality \eqref{ineq:T_*:Lip} now follows immediately.

For $t \leq T^1_*$, Corollary \ref{cor:ineq:dx:cont} implies
\begin{align*}
|\p^i_x\tilde{\eta}_{xt}(x,t, w^1_0, z^1_0, k^1_0)-\p^i_x\tilde{\eta}_{xt}(x,t, w^2_0, z^2_0, k^2_0)| & \leq 38\frac{(1+\alpha) i! \bar{C}^i_x B_d}{(i+1)^2} \mu_i
\end{align*}
for all $x \in \TT$. For $T^1_* \leq t$, \eqref{ineq:eta_xtt:dx} gives us
\begin{align*}
 |\p^i_x\tilde{\eta}_{xt}(x,t, w^1_0, z^1_0, k^1_0)-\p^i_x\tilde{\eta}_{xt}(x,t, w^2_0, z^2_0, k^2_0)| & \leq  \|\p^i_x \p_t \eta_{xt}(w^2_0, z^2_0, k^2_0) \|_{L^\infty_{x,t}} (t\wedge T^2_* - T^1_*) + 38\tfrac{(1+\alpha) i! \bar{C}^i_x B_d}{(i+1)^2} \mu_i
 \\
 & \leq \frac{(1+\alpha) \bar{C}^i_x i!}{i+1} \big[ 21 \bar{C}_t B_z \eps (T^2_*-T^1_*) + 38 \tfrac{B_d \mu_i}{i+1} \big] \\
 & \leq \frac{(1+\alpha) \bar{C}^i_x i! B_d}{i+1} \mu_i \big[ 21 \bar{C}_t B_z \eps \cdot \tfrac{140}{1+\alpha} \tfrac{\mu_0}{\mu_i} + 38  \big] .
\end{align*}
Therefore, if $\bar{C}_t B_z \eps$ is small enough we obtain \eqref{ineq:eta_xt:Lip}. The inequality \eqref{ineq:eta_x:Lip} now follows from using the fact that $\p^i_x\eta_x$ is always constant at time $t=0$ with $\p^i_x \eta_x(x,0)= \delta_{i0}$ and integrating \eqref{ineq:eta_xt:Lip} in time.
\end{proof}

One immediate consequence of Proposition \ref{prop:Lip} is  that the function $G$ defined in \eqref{def:G} is Lipschitz in $(w_0,z_0,k_0)$ with
\begin{align*}
|G(x,t, w^1_0, z^1_0, k^1_0) - G(x,t, w^2_0, z^2_0, k^2_0) | & \leq  135 B_d \big[ \tfrac{\bar{C}^{2n-1}_x}{(2n)^2} + \tfrac{2}{1+\alpha}\big] \|(w^1_0-w^2_0,z^1_0-z^2_0, k^1_0-k^2_0) \|_{2n-1,C_0}
\end{align*}
for all $(x,t) \in \TT \times [0, \frac{3}{1+\alpha}]$, $(w^1_0, z^1_0, k^1_0), (w^2_0, z^2_0, k^2_0) \in \mathcal B_n(\eps, C_0).$ As a result (see \S~\ref{sec:appendix:IFT}), $(\x,\T)$ is a Lipschitz function of $(w_0,z_0,k_0)$ with
\begin{align*}
|(\x,\T)(w^1_0, z^1_0, k^1_0) - (\x,\T)(w^2_0, z^2_0, k^2_0) | & \leq  405 B_d \big[ \tfrac{\bar{C}^{2n-1}_x}{(2n)^2} + \tfrac{2}{1+\alpha}\big] \|(w^1_0-w^2_0,z^1_0-z^2_0, k^1_0-k^2_0) \|_{2n-1,C_0} .
\end{align*}
It now also follows from Proposition \ref{prop:Lip} that $f_n$ is a Lipschitz function of $(w_0,z_0,k_0)$ as well.


\subsection{An implicit function theorem argument}
\label{sec:FCBM}

In this subsection, we will further restrict our attention to the case $n \geq 2$. Define the subspace
\begin{equation}\label{def:X_n}
X_n : = \{ \tilde{w}_0 \in W^{2n+2, \infty}(\TT) : \p^i_x \tilde{w}_0(0) = 0 \:  \forall \: i = 2, \hdots, 2n-1 \}. 
\end{equation}
Note that $\bar{w}_0 \in X_n$.

Pick functions $\tilde{w}^n_1, \hdots, \tilde{w}^n_{2n-2} \in C^\infty_x(\TT)$ such that
\begin{subequations}\label{def:w_0:tilde}
\begin{align}
\|\tilde{w}^n_j\|_{L^\infty_x} & \leq L_n, & \\
\p^{i+1}_x\tilde{w}^n_j(0) & = \delta_{ij}, & i = 0, \hdots, 2n-2, \\
\frac{\|\p^{i+1}_x \tilde{w}^n_j\|_{L^\infty_x}}{i!} & \leq L_n C^i_0, & i = 0, \hdots, 2n+1, \\
|\p^{2n}_x \tilde{w}^n_j(x)| & \leq (2n-1)! L_n (\tfrac{C_0}{2})^{2n-1} \eps, & |x-0| \leq \tfrac{1}{C_0},
\end{align}
\end{subequations}
for some constant $L_n$. For such functions, Taylor expanding about $x =0$  for $x$ with $|x-0| \leq \frac{1}{C_0}$ gives us
\begin{align*}
\bigg\vert \p^{i+1}_x \tilde{w}^n_j(x) - \mathbbm{1}_{\{i \leq j\}} \tfrac{x^{j-i}}{(j-i)!} \bigg\vert & \leq  \tfrac{(2n-1)! L_n (\tfrac{C_0}{2})^{2n-1} \eps}{(2n-1-i)!} |x|^{2n-1-i}\\
& \leq i! L_n {2n-1\choose{i}}\tfrac{1}{2^{2n-1}}C^i_0 \eps  \\
& \leq i! L_n C^i_0 \eps
\end{align*}
for $i =0, \hdots, 2n$.

For a concrete example of such a collection of functions, pick a smooth bump function $\chi \in C^\infty_c(\R)$ satisfying
\begin{align*}
\mathbbm{1}_{[-1,1]} \leq \chi \leq \mathbbm{1}_{[-\frac{C_0}{2}, \frac{C_0}{2}]}
\end{align*}
and then define the $\ZZ$-periodic functions $\tilde{w}_j$ via
\begin{align*}
\tilde{w}_j (x) & : = \tfrac{x^{j+1}}{(j+1)!} \chi(C_0x)   & \text{ for } |x| \leq \tfrac{1}{2}.
\end{align*}
One can check that for any $n \geq 2$ if we define the functions $\tilde{w}^n_1, \hdots \tilde{w}^n_{2n-2}$ to be $\tilde{w}^n_j : = \tilde{w}_j$, then these functions satisfy \eqref{def:w_0:tilde} with
\begin{equation*}
L_n = C_0 (2n+2) \max_{0 \leq k \leq 2n+2} \tfrac{\|\p^k\chi\|_{L^\infty_x}}{k!}.
\end{equation*} 
Note, however, that in this example $L_n$ must diverge to $\infty$ as $n$ grows.

Since $W^{2n+2,\infty}(\TT) =  \RR \tilde{w}^n_1 \oplus \cdots \oplus \RR\tilde{w}^n_{2n-2} \oplus X_n$ and $\bar{w}_0 \in X_n$, we have the affine change of coordinates $ X_n \times \R^{2n-2} \longleftrightarrow W^{2n+2, \infty}(\TT)$
\begin{align*}
 ( \tilde{w}_0, \lambda) & \longleftrightarrow w_0, \\
 w_0 &  = \bar{w}_0 + \tilde{w}_0 + \lambda_1 \tilde{w}^n_1 + \hdots + \lambda_{2n-2} \tilde{w}^n_{2n-2}, \\
 \lambda_j & = \p^{j+1}_x w_0(0).
 \end{align*}
This extends to an affine isomorphism  $X_n \times  (W^{2n+2, \infty}(\TT))^2 \times \R^{2n-2} \longleftrightarrow (W^{2n+2, \infty}(\TT))^3$
\begin{align*}
 (\tilde{w}_0, z_0, k_0, \lambda) & \longleftrightarrow (w_0, z_0, k_0), \\
 w_0 &  = \bar{w}_0 + \tilde{w}_0 + \lambda_1 \tilde{w}^n_1 + \hdots + \lambda_{2n-2} \tilde{w}^n_{2n-2}, \\
 \lambda_j & = \p^{j+1}_x w_0(0).
 \end{align*}
This change of coordinates for $(W^{2n+2, \infty}(\TT))^3$ will allow us to characterize the zero set of $f_n$ as a codimension-$(2n-2)$ Banach manifold in a neighborhood of the point $(\bar{w}_0, 0,0) \in (W^{2n+2, \infty}(\TT))^3$.

With this choice of coordinates in mind, define the open set
\begin{align}
\Lambda_n(\eps, L_n) & : = \big\{ \lambda \in \R^{2n-2} :  L_n \sum^{2n-2}_{j=1} |\lambda_j| < \tfrac{\eps}{2} \big\}, 
\end{align}
and define the open set
\begin{equation*}
U_n(\eps, C_0) \subset X_n \times (W^{2n+2, \infty}(\TT))^2
\end{equation*}
to be $(\tilde{w}_0, z_0, k_0) \in X_n \times (W^{2n+2, \infty}(\TT))^2$ satisfying
\begin{subequations}\label{def:U_n}
\begin{align}
\|\tilde{w}_0\|_{L^\infty_x} & < \tfrac{\eps}{2}, &  \\
\|z_0\|_{L^\infty_x} & < \eps, & \\
\frac{\|\p^{i+1}_x \tilde{w}_0\|_{L^\infty_x}}{i!} & < C^i_0 \tfrac{\eps}{2}, & i=0, \hdots, 2n+1, \\
\frac{\|\p^{i+1}_x z_0\|_{L^\infty_x}}{i!} & < C^i_0 \eps, & i=0, \hdots, 2n+1, \\
\frac{\|\p^{i+1}_x k_0\|_{L^\infty_x}}{i!} & < C^i_0 \eps, & i=0, \hdots, 2n+1.
\end{align}
\end{subequations}
Note that if $(\tilde{w}_0, z_0, k_0, \lambda) \in  U_n(\eps, C_0) \times \Lambda_n(\eps, L_n)$ and $w_0 = \bar{w}_0 + \tilde{w}_0 + \lambda_1 \tilde{w}^n_1 + \hdots + \lambda_{2n-2} \tilde{w}^n_{2n-2}$, then $(w_0,z_0, k_0) \in \mathcal B_n(\eps, C_0)$. For this reason, we can treat the functions $\x, \T,$ and $f_n$ defined in \S~\ref{sec:FCBM:prelim} as functions of  $(w_0,z_0,k_0) \in\mathcal B_n(\eps, C_0)$ as functions of $( \tilde{w}_0, z_0, k_0, \lambda) \in U_n(\eps, C_0) \times \Lambda_n(\eps, L_n)$ by identifying $(\tilde{w}_0, z_0, k_0, \lambda)$ with the corresponding $(w_0,z_0,k_0) \in \mathcal B_n(\eps, C_0)$ via the affine change of coordinates. We will make this tacit identification for the rest of this subsection. 

We will further assume for the rest of this section that $\bar{C}_t, \bar{C}_x$ and $\eps$ are chosen to satisfy the hypotheses of the propositions in \S~\ref{sec:stab}
 with $L = L_n, M= 5L_n$. Therefore, for all solutions with initial data $( \tilde{w}_0, z_0, k_0, \lambda) \in U_n(\eps, C_0) \times \Lambda_n(\eps, L_n)$, the partial derivatives $\Z_{\lambda_j}, \p_{\lambda_j} \eta_x,$ etc. exist for $j=1, \hdots, 2n-2$ and the solution satisfies the estimates from \S~\ref{sec:ineq:dx:stab}.

It follows from \eqref{ineq:T_*:Lip} that if $j \in \{1, \hdots, 2n-2\}$, $(w_0, z_0, k_0) \in \mathcal B_n(\eps, C_0)$, and $(w_0 + \Delta \lambda_j \tilde{w}^n_j, z_0, k_0) \in \mathcal B_n(\eps, C_0)$ for some scalar $\Delta \lambda_j$ we have
\begin{align*}
|T_*(w_0 + \Delta \lambda_j \tilde{w}^n_j, z_0, k_0) - T_*(w_0, z_0, k_0) | & \leq \tfrac{140}{1+\alpha} L_n B_d \Delta \lambda_j.
\end{align*}
Therefore, for each $(\tilde{w}_0, z_0, k_0) \in U_n(\eps, C_0)$ the map $\lambda \in \Lambda_n(\eps, C_0) \rightarrow T_*( \tilde{w}_0, z_0, k_0, \lambda)$ is $\frac{140}{1+\alpha} L_n B_d$-Lipschitz in each $\lambda_j$. It follows that the map $ \lambda  \rightarrow T_*( \tilde{w}_0, z_0, k_0, \lambda)$ is a $(2n-2)^{\frac{1}{2}} \frac{140}{1+\alpha} L_n B_d$-Lipschitz map of $\lambda \in \Lambda_n(\eps, C_0)$ for each $(\tilde{w}_0, z_0, k_0) \in U_n(\eps, C_0)$ and therefore is differentiable at almost every $\lambda \in \Lambda_n(\eps, C_0)$ with
\begin{subequations}
\begin{align}\label{ineq:T_*:dlambda:1}
\|\p_{\lambda_j} T_*\|_{L^\infty_\lambda (\Lambda_n(\eps, L_n))} & \leq \tfrac{140}{1+\alpha} L_n B_d && \forall \:  j=1, \hdots, 2n-2.
\end{align} 
Since $B_\lambda > 4B_d$, \eqref{ineq:T_*:dlambda:1} implies
\begin{align}\label{ineq:T_*:dlambda:2}
\|\p_{\lambda_j} T_*\|_{L^\infty_\lambda (\Lambda_n(\eps, L_n))} & \leq \tfrac{35}{1+\alpha} L_n B_\lambda && \forall \: j=1, \hdots, 2n-2.
\end{align}
\end{subequations}

\begin{lemma}\label{lem:eta_x:lambda}
For all $(\tilde{w}_0, z_0, k_0) \in U_n(\eps, C_0)$, $i=0, \hdots, 2n-1$, $j=1, \hdots, 2n-2$, and $(x,t) \in \TT\times [0, \infty)$ we have that
\begin{subequations}
\begin{align}\label{id:para:eta_x:global}
\|\p^i_x \p_{\lambda_j} \tilde{\eta}_x(x,t) - \tfrac{1+\alpha}{2} t \p^{i+1}_x \tilde{w}^n_j(x) \|_{L^\infty_\lambda (\Lambda_n(\eps, C_0))} 
&  \lesssim (1+\alpha) t i! L_n \bar{C}^i_x \bar{C}_t B_\lambda B_z \eps, 
\end{align}
and for values of $x \in \TT$ with $|x-0| \leq \frac{1}{C_0}$ we have that
\begin{align}
\label{id:para:eta_x:nbhd}
\|\p^i_x \p_{\lambda_j} \tilde{\eta}_x(x,t) - \tfrac{1+\alpha}{2} t \mathbbm{1}_{\{i \leq j \}} \tfrac{x^{j-i}}{(j-i)!}\|_{L^\infty_\lambda (\Lambda_n(\eps, C_0))} 
&  \lesssim (1+\alpha) t i! L_n \bar{C}^i_x \bar{C}_t B_\lambda B_z \eps.
\end{align}
It follows that
\begin{align}
\label{id:df^id_j}
\|\p^i_x \p_{\lambda_j} \tilde{\eta}_x(\x,\T) - \tfrac{1+\alpha}{2} \T \delta_{ij} \|_{L^\infty_\lambda (\Lambda_n(\eps, C_0))} 
&  \lesssim \tfrac{1+\alpha}{2} \T \big(\tfrac{(2n)!\bar{C}^{2n-1}_x}{(2n)^3} + i! L_n \bar{C}^i_x \bar{C}_t B_\lambda\big) B_z \eps.
\end{align}
\end{subequations}
\end{lemma}

\begin{proof}
The inequality \eqref{id:para:eta_x:global} is an application of Corollary \ref{cor:eta_x:approx:stab} and \eqref{ineq:T_*:dlambda:2}. The bound \eqref{id:para:eta_x:nbhd} follows from \eqref{id:para:eta_x:global} and the identity
\begin{align*}
\p^{i+1}_x \tilde{w}^n_j(x) & = \mathbbm{1}_{\{i \leq j\}} \tfrac{x^{j-i}}{(j-i)!} + \OO\big(i! L_n C^i_0 \eps \big) & |x-0| \leq \tfrac{1}{C_0}, i = 0, \hdots, 2n-1
\end{align*}
derived earlier in this section. The last inequality \eqref{id:df^id_j} follows from plugging $(x,t) = (\x,\T)$ into \eqref{id:para:eta_x:nbhd} and applying \eqref{ineq:x:ring}.
\end{proof}

\begin{proposition}\label{prop:Df}
Let $n \geq 2$ and let $f_n$ be the function defined in \eqref{def:f}. If 
\begin{equation*}
(2n)! L_n \bar{C}^{2n-1}_x \bar{C}_t B_\lambda B_z \eps \ll 1
\end{equation*}
then there exists a function $\lambda_* : U_n(\eps^2, C_0) \rightarrow \Lambda_n(\eps, C_0)$, Lipschitz with respect to the $(W^{2n+2, \infty}(\TT))^3$ norm, such that
\begin{equation*}
\bigg\{ (\tilde{w}_0, z_0, k_0, \lambda) \in U_n(\eps^2, C_0) \times \Lambda_n(\eps, C_0) : f_n  = 0 \bigg\} = \bigg\{ (\tilde{w}_0, z_0, k_0, \lambda_*( \tilde{w}_0, z_0, k_0)) : (\tilde{w}_0, z_0, k_0) \in U_n(\eps, C_0) \bigg\}.
\end{equation*}
\end{proposition}

\begin{proof}
First, we apply $\p_{\lambda_j}$ to the equation $G(\x,\T) = (0,0)$ and get 
\begin{equation*}
\begin{bmatrix} \p_{\lambda_j} \x \\ \p_{\lambda_j} \T \end{bmatrix} = - DG(\x,\T) ^{-1}\begin{bmatrix} \tfrac{1}{(2n)!}\p_{\lambda_j} \p^{2n-1}_x \tilde{\eta}_x(\x,\T) \\ -\tfrac{2}{1+\alpha}\p_{\lambda_j} \eta_x(\x,\T) \end{bmatrix} .
\end{equation*} 
We know from the proof of Proposition \ref{prop:G} that $\| \Id- DG(\x,\T)\|_{\text{operator}} \leq | \Id - DG(\x,\T)| \leq \frac{2}{3}$, so it follows (use Neumann series) that $\|DG(\x,\T)^{-1}\|_{\text{operator}} \leq 3$. It therefore follows from \eqref{id:df^id_j} that
\begin{align*}
\| \p_{\lambda_j} \x \|_{L^\infty_\lambda (\Lambda_n(\eps, C_0))}  & \lesssim \tfrac{1+\alpha}{2} \T \tfrac{1}{2n}\big(1+L_n\bar{C}_t B_\lambda\big) \bar{C}^{2n-1}_xB_z \eps,
\\
\| \p_{\lambda_j} \T \|_{L^\infty_\lambda (\Lambda_n(\eps, C_0))}  &  \lesssim \tfrac{1+\alpha}{2} \T \tfrac{1}{2n}\big(1+L_n\bar{C}_t B_\lambda\big) \bar{C}^{2n-1}_xB_z \eps.
\end{align*}

Using these bounds for $\p_{\lambda_j} \x$ and $\p_{\lambda_j} \T$ in conjunction with the bounds \eqref{approx:eta_xt:D},  \eqref{approx:eta_x:D}, and \eqref{id:df^id_j}, we compute that for $i, j = 1, \hdots, 2n-2$ we have that
\begin{align*}
\frac{\p f^i_n}{\p \lambda_j} & = \frac{\p}{\p \lambda_j}\big[ \p^i_x \tilde{\eta}_x(\x, \T) \big] \\
& = \p^i_x \p_{\lambda_j} \tilde{\eta}_x(\x,\T) + \p^{i+1}_x \tilde{\eta}_x(\x,\T) \p_{\lambda_j} \x + \p^i_x \tilde{\eta}_{xt}(\x,\T) \p_{\lambda_j} \T 
\\
& = \tfrac{1+\alpha}{2} \T [\delta_{ij} + \tfrac{1}{(2n)^3}\OO((2n)!\bar{C}^{2n-1}_x B_z \eps) +  \OO(i! L_n \bar{C}^i_x \bar{C}_t B_\lambda B_z \eps)] 
\\
&\quad  + \tfrac{1+\alpha}{2}\T \tfrac{1}{2n}\OO((2n)!\bar{C}^{2n-1}_x B_z \eps)\OO \big(\tfrac{1}{2n}\big(1+L_n\bar{C}_t B_\lambda\big) \bar{C}^{2n-1}_xB_z \eps \big).
\end{align*}
It follows that if $(2n)! L_n \bar{C}^{2n-1}_x \bar{C}_t B_\lambda B_z \eps$ is small enough we have that
\begin{align*}
\| \p_{\lambda_j} f^i_n - \tfrac{1+\alpha}{2}\T \delta_{ij} \|_{L^\infty_\lambda (\Lambda_n(\eps, C_0))} \leq \tfrac{1}{19}\tfrac{1+\alpha}{2}\T \tfrac{1}{(2n)^2}
\end{align*}
for all $i,j = 1, \hdots, 2n-2$. Therefore,
\begin{align*}
\big\| D_\lambda f_n - \Id \|_{L^\infty_\lambda (\Lambda_n(\eps, C_0))} \leq \tfrac{1}{9}.
\end{align*}

If $(\tilde{w}_0, z_0, k_0) \in U_n(\eps^2, C_0)$ and $\lambda = 0$, then $(w_0,z_0,k_0) = (\tilde{w}_0, z_0, k_0) \in U_n(\eps^2, C_0) \subset \mathcal B_n(\eps^2, C_0)$. It follows from \eqref{approx:eta_x:D} that
\begin{align*}
|f^i_n(\tilde{w}_0, z_0, k_0, 0)| & \leq 50 \frac{(2n)!L_n \bar{C}^{2n-1}_x B_z \eps}{(2n) } \frac{\eps}{L_n}.
\end{align*}
for $i=1, \hdots, 2n-2$, $(\tilde{w}_0, z_0, k_0) \in U_n(\eps^2, C_0)$. Therefore, if $(2n)!L_n \bar{C}^{2n-1}_x B_z \eps$ is small enough we have
\begin{align*}
|f_n(\tilde{w}_0, z_0, k_0, 0)| & \leq \tfrac{\eps}{10L_n} && \forall \: (\tilde{w}_0, z_0, k_0) \in U_n(\eps^2, C_0).
\end{align*}
It now follows (see \S~\ref{sec:appendix:IFT}) that we can apply the implicit function theorem to $f_n$ and conclude our result.
\end{proof}

Another way of phrasing Proposition \ref{prop:Df} is as follows:  if we define the open neighborhood
\begin{equation}\label{def:B:tilde}
\tilde{\mathcal B}_n(\eps, C_0, L_n) : = \{ (\bar{w}_0 + \sum^{2n-2}_{j=1} \lambda_j \tilde{w}^n_j + \tilde{w}_0, z_0, k_0) : (\tilde{w}_0, z_0, k_0, \lambda ) \in U_n(\eps^2, C_0) \times \Lambda_n(\eps, L_n)  \big\}
\end{equation}
 of $(\bar{w}_0, 0,0)$ in $(W^{2n+2, \infty}(\TT))^3$, then the zero set of $f_n$ in $\tilde{\mathcal B}_n(\eps, C_0)$ is a codimension-$(2n-2)$, $C^{0,1}$ Banach manifold.


\section{A detailed description of the cusp structure at the first singularity}
\label{sec:thm}

We will now prove our main theorem, which will imply Theorem \ref{thm:HR:a}. 

\begin{theorem}[\bf Finite-codimension stable shock formation]\label{thm:HR}
Fix an integer $n \geq 1$,  a value of $\alpha = \frac{\gamma-1}{2}> 0$, and a constant $C_0 \geq 3$. Define $\bar{w}_0$ as in \eqref{def:w_0:bar}. Then there exists a codimension-$(2n-2)$ Banach manifold\footnote{In the case $n=1$, $\mathcal M_1$ is simply an open subset of $(W^{4, \infty}(\TT))^3$. When $n\geq2$, it is the graph of a Lipschitz function.} $\mathcal M_n \subset (W^{2n+2, \infty}(\TT))^3$ containing $(\bar{w}_0, 0,0)$ and a constant 
$$ \eps_0 = \eps_0(n, \alpha, C_0) $$
such that if $\eps < \eps_0$ then for all initial data $(w_0,z_0, k_0) \in \mathcal M_n$ satisfying
\begin{equation} \label{eq:thm:hyp}
\begin{aligned}
\|w_0 - \bar{w}_0\|_{L^\infty_x} & < \eps, & \\
\|z_0\|_{L^\infty_x} & < \eps, & \\
\frac{\|\p^{i+1}_x(w_0- \bar{w}_0)\|_{L^\infty_x}}{i!} & < C^i_0 \eps, & \forall \: i = 0, \hdots, 2n+1, \\
\frac{\|\p^{i+1}_x z_0\|_{L^\infty_x}}{i!} & < C^i_0\eps, & \forall \: i = 0, \hdots, 2n+1, \\
\frac{\|\p^{i+1}_x k_0\|_{L^\infty_x}}{i!} & < C^i_0 \eps & \forall \: i = 0, \hdots, 2n+1,
\end{aligned}
\end{equation}
 the unique locally-well posed solution $(w,z,k)$ to \eqref{eq:euler:RV} with initial data $(w_0,z_0,k_0)$ satisfies the following:
\begin{enumerate}[leftmargin=*]
\item $(w,z,k)$ exists as a classical solution up until a finite time $T_*$ with 
\begin{equation}\label{thm:eq:T_*}
 T_* = \tfrac{2}{1+\alpha}[1 + \OO_\alpha(\eps)];
 \end{equation}
 \item the functions $z$ and $k$ remain uniformly $C^{1,\frac{1}{2n+1}}$ on $\TT \times [0,T_*]$  with 
 \begin{align}
 [\p_y z]_{C^{0,\frac{1}{2n+1}}_y}  \lesssim_{n, \alpha, C_0} \eps, \qquad [\p_y k]_{C^{0,\frac{1}{2n+1}}_y}  \lesssim_{n, \alpha, C_0} \eps,
 \end{align}
 but $w$ has a gradient blowup at a unique point $(y_*, T_*) \in \TT \times [0,T_*]$ with $y_*$ satisfying 
 \begin{equation}\label{thm:approx:y_*}
  |y_*-\tfrac{1}{2}| \lesssim_{n, \alpha, C_0} \eps;
 \end{equation}
 \item away from $(y_*,T_*)$, $w, z,$ and $k$ remain locally $C^{2n+1,1}$ in $\TT \times[0,T_*]$ ;
 \item for all $y$ in the interval $|y-y_*| < \tfrac{1}{(2n+2)(2n+1)2^{2n+2}C^{2n+1}_0}$ we have
 \begin{subequations}
 \begin{align}\label{thm:expansion:0}
 w(y,T_*) & = b_0 + b_1 (y-y_*)^{\frac{1}{2n+1}} + \OO_{C_0} (|y-y_*|^{\frac{2}{2n+1}}), \\
 \label{thm:expansion:1}
 \p_y w(y,T_*) & = \tfrac{1}{2n+1} b_1 (y-y_*)^{-\frac{2n}{2n+1}} + \OO_{n, \alpha, C_0} ( |y-y_*|^{-\frac{2n-1}{2n+1}}),
 \end{align}
 \end{subequations}
 where $b_0$ and $b_1$ are {\it{constants}} satisfying
 \begin{align*}
 b_0 & = \tfrac{5}{2} + \OO_{n, \alpha, C_0} (\eps), \\
 b_1 & = -(2n+1)^{\frac{1}{2n+1}}[1+ \OO_{n, \alpha, C_0}(\eps)].
\end{align*}
\end{enumerate}
\end{theorem}

\begin{proof}
Recall that the constants $B_k, B_z, B_d, B_\lambda$ are determined by $\alpha$ via the definitions
\begin{align*}
B_k  & := 6^{\frac{1}{\alpha}}, &
B_z & := 6^{\frac{2}{\min(1,\alpha)}}(2+\tfrac{1}{\gamma})e^{21}, 
\\
B_d & : = 2\cdot 6^{\frac{3}{\min(1,\alpha)}} e^{31}, &
B_\lambda & : = \tfrac{5}{2} \cdot 9^{\frac{3}{\min(1,\alpha)}} \cdot e^{31}.
\end{align*}
Choose
\begin{equation*}
\delta := \tfrac{6}{\min(1,\alpha)} 
\qquad 
\mbox{and}
\qquad
\kappa  := \max(1,\alpha)(2 + 5\delta)
\end{equation*}
so that $\delta$ and $\kappa$ satisfy \eqref{def:delta}. With this choice of $\delta$ and $\kappa$, choose the positive constant $C_t(0)$ to satisfy
\begin{align*}
C_t(0)  \gg (1+\alpha) 3^\delta, \qquad
C_t(0)  \gg \max(1,\alpha)2^\delta C_0,
\end{align*}
where the implicit constants here are those given by the hypotheses of Propositions \ref{prop:EE:infty},  \ref{prop:EE:stab}, and \ref{prop:EE:cont}. Now define
\begin{align*}
\bar{C}_t & : = C_t(0) e^{(11+\log 3) \delta + 5}
\end{align*}
in accordance with definition \eqref{def:C_t:bar}, and pick a constant $\bar{C}_x$ large enough that 
\begin{align*}
\bar{C}_x  \gg 1, \qquad
\alpha \bar{C}_x  \gg \bar{C}_t + 1, \qquad
\bar{C}_x  \gg (1+\alpha) C_0,
\end{align*}
where the implicit constants are sufficient to satisfy the hypotheses on $\bar{C}_x$ in \S~\ref{sec:HOE}--\ref{sec:cont} and Proposition \ref{prop:G}.

Now choose $\eps_0> 0$ small enough that
\begin{align*}
(1+\alpha)^2B^2_z \eps_0 & \ll 1, &
3^{\frac{3}{\min(1,\alpha)}} B_d (\bar{C}_t + 1+\alpha)B_z \eps^{\frac{1}{2}} & \ll 1, \\
\bar{C}_t B_\lambda B_z \eps & \ll 1, &
\bar{C}^{2n}_x B_z \eps_0 & \ll 1,
\end{align*}
so that $\mathcal B_n(\eps, C_0) \subset \mathcal A_n(\eps, C_0)$ and the hypotheses of \S \ref{chp-estimates}--\ref{sec:FCBM:prelim} are satisfied for any choice of $\eps \leq \eps_0$. If $n \geq 2$, pick a choice of functions $\tilde{w}^n_1, \hdots, \tilde{w}^n_{2n-2}$ with a corresponding constant $L_n$ as described in \S \ref{sec:FCBM} and add the constraint 
\begin{align*}
(2n)! L_n \bar{C}^{2n-1}_x \bar{C}_tB_\lambda B_z \eps_0 & \ll 1,
\end{align*}
so that the hypotheses of Proposition \ref{prop:Df} are satisfied for any $\eps \leq \eps_0$. Lastly, if $R_n$ and $M_n$ are the constants defined in \S~\ref{sec:appendix:poly}, we will add the constraints that $\eps_0$ is small enough  that
\begin{equation}\label{thm:constraints:eps_0}
\begin{aligned}
93\bar{C}^{2n}_x B_z \eps_0 & < 1, \\
 \tfrac{22(1+\eps_0^{\frac{1}{2}})}{(n+1)^2}(\tfrac{1}{C_0} + \tfrac{1}{6n^3 \bar{C}_x}) \bar{C}^{2n+1}_x B_z \eps_0 & < 1, \\
  \tfrac{88(1+\eps_0^{\frac{1}{2}})}{(n+1)^2 R_n}(\tfrac{1}{C_0} + \tfrac{1}{6n^3 \bar{C}_x})\bar{C}^{2n+1}_x B_z \eps_0 & < 1,\\
  \tfrac{88(1+\eps_0^{\frac{1}{2}}) M_n}{(n+1)^2 R_n}\bar{C}^{2n+1}_x B_z \eps_0 & < 1, \\
  \tfrac{44(1+\eps_0^{\frac{1}{2}})}{(n+1)^2}\bar{C}^{2n+1}_x B_z \eps_0 & < 1.
\end{aligned}
\end{equation}
These last constraints \eqref{thm:constraints:eps_0} will be used in some of our computations below.

With the constants $\bar{C}_x, \bar{C}_t, L_n$, and $\eps_0$ now chosen, define $\mathcal M_n \subset (W^{2n+2, \infty}(\TT))^3$ to be
\begin{equation*}
\mathcal M_n : = \begin{cases} \mathcal B_1(\eps_0, C_0) & n = 1  \\\{ (w_0,z_0, k_0) \in \tilde{\mathcal B}_n(\eps_0, C_0, L_n) : f_n(w_0,z_0,k_0) =0\} & n \geq 2 \end{cases} ,
\end{equation*}
where $\mathcal B_1(\eps_0, C_0) \subset \mathcal A_1(\eps_0, C_0)$ is the open set defined by \eqref{eq:cal:B:n:def} when $n=1$, $f_n$ is the function defined by \eqref{def:f}, and $\tilde{\mathcal B}_n(\eps, C_0, L_n)$ is the open set defined by \eqref{def:B:tilde}. We know from Proposition \ref{prop:Df} that when $n \geq 2$, $\mathcal M_n$ is a codimension-$(2n-2)$ Banach submanifold of $(W^{2n+2, \infty}(\TT))^3$, i.e. it is the graph of a Lipschitz function from a codimension-$(2n-2)$ linear subspace of $(W^{2n+2, \infty}(\TT))^3$ into $\RR^{2n-2}$.

Fix $\eps \leq \eps_0$ and a choice of initial data $(w_0, z_0, k_0) \in \mathcal M_n$ satisfying \eqref{eq:thm:hyp}. According to Proposition \ref{prop:T_*}, the blowup time $T_*$ is finite and satisfies \eqref{ineq:T_*}--\eqref{eq:T_*}. Since
\begin{equation*}
\p_t(w\cir\eta) = \tfrac{\alpha}{2\gamma} \Sigma^2 \K, 
\qquad
\p_t(z\cir\eta) = 2\alpha \Sigma \Z - \tfrac{\alpha}{2\gamma} \Sigma^2 \K, 
\qquad
\p_t(k\cir\eta) = \alpha \Sigma \K,
\end{equation*}
it follows from \eqref{ineq:0thorder} and our assumptions on $w_0,z_0, k_0$ that
\begin{equation}\label{id:Eulerian}
w\cir \eta = \bar{w}_0 + \OO(B_k \eps), 
\qquad
z\circ \eta = z_0 + \OO(\alpha B_z \eps), 
\qquad
k\circ \eta  = k_0 + \OO(\alpha B_k \eps).
\end{equation}

We know (see the discussion at the end of \S~\ref{sec:FCBM:prelim}) that because  $(w_0, z_0, k_0) \in \mathcal M_n$, $\eta_x$ has a unique zero in $\TT \times [0,T_*]$ at a point $(x_*,T_*) $, that
\begin{equation}\label{eq:vanish}
\p_x\eta_x(x_*,T_*) = \hdots = \p^{2n-1}_x \eta_x(x_*, T_*) = 0,
\end{equation} 
 and that $x_*$ satisfies the bounds $|x_*-0| < \frac{1}{3(1+\alpha)C_0}$ and \eqref{ineq:x_*:apriori}. \eqref{ineq:x:ring} and \eqref{thm:constraints:eps_0} also imply that
 \begin{equation}\label{ineq:x_*}
|x_*|  \leq \tfrac{93}{(2n)^3}\bar{C}^{2n-1}_x B_z \eps < \tfrac{1}{(2n)^3 \bar{C}_x}.
\end{equation}

The estimates in \S~\ref{sec:HOE} prove that the map $\eta(\cdot, T_*): \TT \rightarrow \TT$ is a $C^{2n+1,1}$ homeomorphism. It follows from \eqref{eq:vanish} that  $\eta(\cdot,T_*)$ lifts to a map $x \rightarrow y$ from $\R$ to $\R$ with a Taylor expansion
\begin{align*}
y & = y_* + a_{2n+1} (x-x_*)^{2n+1} + a_{2n+2}(x) (x-x_*)^{2n+2}.
\end{align*}
In what follows, we will continue to abuse notation and use $x$ to refer both to the Lagrangian variable $x \in \TT$ and to the lifted variable $x \in \R$. Additionally, we will abuse notation and use $y$ to refer both to the Eulerian variable $y = \eta(x, T_*) \in \TT$ and the lifted variable $y \in \R$. So, for example, when we show below that the zeroth order term $y_* \in \R$ of the above Taylor expansion satisfies $y_* = \frac{5}{2} + \OO_{n,\alpha, C_0} (\eps)$ we can conclude that the corresponding Eulerian variable $y_* = \eta(x_*, T_*) \in \TT = \R/\ZZ$ satisfies $|y_* - \frac{1}{2}| \lesssim_{n,\alpha,C_0} \eps$.

Because
\begin{equation*}
a_{2n+1} = \tfrac{\p^{2n}_x \eta_x(x_*, T_*)}{(2n+1)!},
\end{equation*}
it follows from \eqref{approx:eta_x:C}, \eqref{eq:T_*}, and \eqref{ineq:LB:2n} that
\begin{subequations}
\begin{align}
\label{a:2n+1:1}
a_{2n+1} 
 &  
= \tfrac{1}{2n+1}[1+ \OO( \bar{C}^{2n}_x B_z \eps)], \\
\label{a:2n+1:2}
a_{2n+1} & > \tfrac{1}{2(2n+1)}.
\end{align}
\end{subequations}
Using \eqref{approx:eta_x:C} and \eqref{ineq:T_*}, we obtain
\begin{equation}\label{ineq:a_2n+2:local}
|a_{2n+2}(x)| 
\leq 44(1+\eps^{\frac{1}{2}}) \tfrac{\bar{C}^{2n+1}_x B_z \eps}{(2n+2)^3},
\qquad \forall \: |x| \leq \tfrac{1}{C_0}.
\end{equation}
The inequalities \eqref{ineq:x_*},\eqref{a:2n+1:2}, \eqref{ineq:a_2n+2:local} together with the hypotheses \eqref{thm:constraints:eps_0} allow us to apply Lemma \ref{lem:appendix:fracseries} with $r = \frac{1}{C_0}$ and obtain
\begin{align}\label{thm:id:x-x_*}
\bigg\vert (x-x_*)- \big( \tfrac{y-y_*}{a_{2n+1}}\big)^{\frac{1}{2n+1}} \bigg\vert & < \big\vert \tfrac{y-y_*}{a_{2n+1}}\big\vert^{\frac{2}{2n+1}} && \forall \: |x| \leq \tfrac{1}{C_0}.
\end{align}
Lemma \ref{lem:appendix:fracseries} and the fact that $\frac{1}{C_0} - |x_*| > \frac{1}{2C_0}$ lets us conclude that $\{ y : \eta(x, T_*) = y, |x| \leq \frac{1}{C_0} \}$ contains the ball
\begin{equation*}
|y-y_*| < \tfrac{1}{(2n+2)(2n+1)2^{2n+2}C^{2n+1}_0}.
\end{equation*}

Taylor expanding $w\circ \eta(\cdot, T_*)$ in $x$ about $x_*$ gives us
\begin{equation*}
w\cir \eta(x,T_*) = w\cir \eta(x_*, T_*) + \p_x(w\cir \eta)(x_*, T_*) (x-x_*) + \OO\big( \|\p^2_x(w\cir \eta)(\cdot, T_*) \|_{L^\infty_x} (x-x_*)^2\big).
\end{equation*}
\eqref{id:Eulerian}, \eqref{ineq:x_*}, and our hypotheses on $\bar{w}_0$ imply that if we define $B^w_0 : = w\cir \eta (x_*, T_*)$ then
\begin{align*}
w\cir \eta(x_*,T_*) & = \tfrac{5}{2} + \OO( \tfrac{\bar{C}^{2n-1}_x B_z \eps}{(2n)^3}) + \OO(B_k \eps).
\end{align*}
Since
\begin{align*}
\p_x(w\cir \eta) = w'_0 - \tfrac{1}{2\gamma} \sigma_0 k'_0 + \int^t_0 (\eta_x \W)_t \: ds + \tfrac{1}{2\gamma} \eta_x \Sigma \K,
\end{align*}
our derivative estimates from \S~\ref{sec:HOE} tell us that
\begin{align}\label{thm:approx:w}
\frac{(i+1)^2\|\p^{i+1}_x (w\cir \eta) - \p^{i+1}_x \bar{w}_0\|_{L^\infty_x}}{i! \bar{C}^i_x } & \lesssim B_k \eps
\end{align}
for $t \in [0,T_*]$, $i=0, \hdots, 2n+1$. Therefore, if we define
\begin{equation}
B^w_1 : = \p_x(w\cir \eta)(x_*, T_*) = -1 + x^{2n}_* + \p_x(w\cir \eta)(x_*, T_*)-\bar{w}'_0(x_*)
\end{equation}
then \eqref{ineq:x_*:1} and \eqref{thm:approx:w} give us
\begin{equation}\label{id:Bw1}
B^w_1  = -1 + \OO(B_z \eps).
\end{equation}
Our assumptions on $\bar{C}_x$ and $\eps_0$ imply that $\bar{C}_x B_z \eps_0 < 1$, so \eqref{thm:approx:w} also yields
\begin{equation}\label{thm:bd:w:2}
\|\p^2_x(w\cir \eta)(\cdot, T_*) \|_{L^\infty_x} \leq C_0 + \OO( \bar{C}_x B_k \eps) \lesssim C_0.
\end{equation}
We thus arrive at the expansion
\begin{equation*}
w\cir \eta(x, T_*) = B^w_0 + B^w_1 (x-x_*) + \OO(C_0(x-x_*)^2)
\end{equation*}
for all $x$. Plugging \eqref{thm:id:x-x_*} into this equation gives us
\begin{align*}
w(y,T_*) & = B^w_0 + B^w_1 \big(\tfrac{y-y_*}{a_{2n+1}}\big)^{\frac{1}{2n+1}}[1 + \OO\big(\big\vert\tfrac{y-y_*}{a_{2n+1}}\big\vert^{\frac{1}{2n+1}}\big)] + \OO\big(C_0\big\vert\tfrac{y-y_*}{a_{2n+1}}\big\vert^{\frac{2}{2n+1}}\big) \\
& = B^w_0 + B^w_1 \big(\tfrac{y-y_*}{a_{2n+1}}\big)^{\frac{1}{2n+1}} + \OO\big((1+C_0)\big\vert\tfrac{y-y_*}{a_{2n+1}}\big\vert^{\frac{2}{2n+1}}\big)
\end{align*}
for all $|y-y_*| < \frac{1}{(2n+2)(2n+1)2^{2n+2}C^{2n+1}_0}$. If we define
\begin{equation}
\begin{split}
b_0 & : = B^w_0 = \tfrac{5}{2} + \OO(\bar{C}^{2n}_x B_z \eps) , \\
b_1 & : = B^w_1 a^{-\frac{1}{2n+1}}_{2n+1} = -[1+ \OO(B_z \eps)](\tfrac{1}{2n+1})^{-\frac{1}{2n+1}}[1 + \OO(\bar{C}^{2n}_x B_z \eps)]^{-\frac{1}{2n+1}} \\
& = -(2n+1)^{\frac{1}{2n+1}}[1+ \OO(\bar{C}^{2n}_x B_z \eps)],
\end{split}
\end{equation}
then, using \eqref{a:2n+1:2}, we arrive at our expansion \eqref{thm:expansion:0} for $w(\cdot, T_*)$.

Taylor expanding $\eta_x(\cdot, T_*)$ about $x_*$ and using \eqref{eq:vanish}, \eqref{approx:eta_x:C}, \eqref{ineq:T_*}, and \eqref{thm:constraints:eps_0} gives us
\begin{subequations}\label{eq:Taylor:eta_x}
\begin{align}
\eta_x(x, T_*) & = (2n+1)a_{2n+1}(x-x_*)^{2n} + h_{2n+1}(x)(x-x_*)^{2n+1},  && \\
|h_{2n+1}(x)| & \leq \tfrac{44(1+\eps^{\frac{1}{2}})}{(2n+2)^2} \bar{C}^{2n+1}_x B_z \eps < \tfrac{1}{4} && \forall \: |x| \leq \tfrac{1}{C_0}.
\end{align}
\end{subequations}
We know from  \eqref{thm:bd:w:2} that
\begin{align*}
\p_x(w\cir \eta)(x,T_*) & = B^w_1 + \OO\big( \|\p^2_x(w\cir \eta)(\cdot, T_*) \|_{L^\infty_x} (x-x_*)\big) \\
 & = B^w_1 + \OO\big( C_0(x-x_*)\big)
\end{align*}
for all $x$. Therefore, it follows from \eqref{thm:id:x-x_*} that for all $y$ in the interval $|y-y_*| < \frac{1}{(2n+2)(2n+1)2^{2n+2}C^{2n+1}_0}$ we have that
\begin{align*}
\p_y w(y, T_*) & = \frac{\p_x(w\cir\eta)(x,T_*)}{\eta_x(x, T_*)} \\
& = \frac{\big[ B^w_1+ \OO\big( C_0 (x-x_*)\big) \big] }{(2n+1)a_{2n+1}}(x-x_*)^{-2n} \big[ 1 + \OO(\bar{C}^{2n+1}_x B_z \eps (x - x_*) ) \big] \\
& = \frac{B^w_1 + \OO(C_0|\frac{y-y_*}{a_{2n+1}}|)}{(2n+1)a_{2n+1}}\big(\tfrac{y-y_*}{a_{2n+1}}\big)^{-\frac{2n}{2n+1}}\big[1 + \OO\big(\bar{C}^{2n+1}_x B_z \eps\big\vert\tfrac{y-y_*}{a_{2n+1}}\big\vert^{\frac{1}{2n+1}}\big)] \\
& = -\tfrac{1}{2n+1} b_1(y-y_*)^{-\frac{2n}{2n+1}} + \OO_{n,\alpha, C_0}\big(\big\vert\tfrac{y-y_*}{a_{2n+1}}\big\vert^{-\frac{2n-1}{2n+1}}\big)
\end{align*}
which gives us \eqref{thm:expansion:1}.

To approximate the location of $y_*$, using \eqref{eq:T_*}, \eqref{id:Eulerian}, and our bounds on $|x_*|$ gives us
\begin{equation*}
y_* 
= \eta(x_*,T_*) 
= x_* + \int^{T_*}_0 \lambda_3 \cir \eta (x_*, t) \: dt 
= \tfrac{5}{2} + \tfrac{\OO(\bar{C}^{2n}_x B_z \eps)}{n^3 \bar{C}_x} + \OO((1+\alpha) B_z \eps).
\end{equation*}
This implies \eqref{thm:approx:y_*}.

To prove $C^{1,\frac{1}{2n+1}}_y$ estimates on $z$ and $k$, we will first show that $\eta^{-1}(\cdot, t)$ is uniformly $C^{0, \frac{1}{2n+1}}_y$ up to time $T_*$. In particular, we will show that
\begin{align}\label{ineq:holder:eta-1}
|x_1-x_2| & \leq 2^{2-\frac{1}{2n+1}} C^{2n}_0 |\eta(x_1,t) - \eta(x_2,t)|^{\frac{1}{2n+1}} && \forall \: x_1, x_2 \in \TT, t \in [0, T_*].
\end{align}
We will prove \eqref{ineq:holder:eta-1} by bounding $|x_1-x_2|$ separately on the segments $|x-0| \leq \frac{1}{C_0}$ and $|x-0| \geq \frac{1}{C_0}$. 

We know from  Lemma \ref{lem:LB:0} that 
\begin{equation*}
\eta_x \geq \tfrac{1}{2}C^{-2n}_0
\end{equation*}
for $|x-0| \geq \frac{1}{C_0}$. Therefore, for all $x_1, x_2$ in this segment we have the bound
\begin{align*}
|x_1-x_2| & \leq 2C^{2n}_0 |\eta(x_1,t) - \eta(x_2,t)| \leq 2C^{2n}_0 |\eta(x_1,t) - \eta(x_2,t)|^{\frac{1}{2n+1}} && \forall \: t \in [0,T_*].
\end{align*}

Now consider $x_1,x_2$ with $-\frac{1}{C_0} \leq x_1 < x_2 < \frac{1}{C_0}$. We know from \eqref{approx:eta_xt:C} that $\eta_{xt}(x,t) < 0$ for all $|x-0| \leq \frac{1}{C_0}, t \in [0, T_*]$, and therefore,
\begin{align*}
|\eta(x_2,t)- \eta(x_1,t)|  = \int^{x_2}_{x_1} \eta_x(x, t) \: dx \geq \int^{x_2}_{x_1} \eta_x(x, T_*) \: dx = |\eta(x_2,T_*) -\eta(x_1,T_*)|
\end{align*}
for all $t \in [0,T_*]$. Therefore, it suffices to bound $|\eta(x_2,T_*) -\eta(x_1,T_*)|$ below. In the case where $x_* \leq  x_1$, \eqref{eq:Taylor:eta_x}  and \eqref{a:2n+1:2} imply
\begin{align*}
\eta(x_2,T_*) -\eta(x_1,T_*) & = \int^{x_2}_{x_1} \eta_x(x,T_*) \: dx \\
& \geq a_{2n+1} \int^{x_2}_{x_1} (2n+1)(x-x_*)^{2n} \: dx - \tfrac{1}{4} \int^{x_2}_{x_1} (x-x_*)^{2n+1} \: dx \\
& = a_{2n+1}\big[ (x_2-x_*)^{2n+1}-(x_1-x_*)^{2n+1}\big] - \tfrac{1}{4(2n+2)} \big[ (x_2-x_*)^{2n+2}-(x_1-x_*)^{2n+2}\big] \\
& \geq \big[a_{2n+1} - \tfrac{1}{4(2n+2)} \tfrac{2n+1}{2n+2}\big] (x_2-x_1)^{2n+1} \\
& \geq \tfrac{1}{4(2n+1)}(x_2-x_1)^{2n+1}.
\end{align*}
The same argument works for the case where $x_2 \leq x_*$. In the last case where $x_1 < x_* < x_2$,  \eqref{thm:id:x-x_*} gives us the bound
\begin{align*}
|x_2-x_1| & \leq \big\vert\tfrac{y_2-y_1}{a_{2n+1}}\big\vert^{\frac{1}{2n+1}} + \big\vert\tfrac{y_2-y_*}{a_{2n+1}}\big\vert^{\frac{2}{2n+1}} + \big\vert\tfrac{y_1-y_*}{a_{2n+1}}\big\vert^{\frac{2}{2n+1}} \\
& \leq 3\big\vert\tfrac{y_2-y_1}{a_{2n+1}}\big\vert^{\frac{1}{2n+1}}  \\
& \leq 3(2(2n+1))^{\frac{1}{2n+1}} |y_2-y_1|^{\frac{1}{2n+1}} ,
\end{align*}
where $y_i = \eta(x_i, T_*)$.  Therefore
\begin{equation*}
|x_2-x_1| \leq 6|\eta(x_2,t)- \eta(x_1,t)|^{\frac{1}{2n+1}}
\end{equation*}
for all $x_1,x_2$ with $|x_i| \leq \frac{1}{C_0}$, $t \in [0, T_*]$. Putting this together with our bounds for $|x-0| \geq \tfrac{1}{C_0}$, we conclude that \eqref{ineq:holder:eta-1} holds.

It now follows from our estimates in \S~\ref{sec:HOE} that 
\begin{equation*}
|\p_y k(y_1,t) - \p_y k(y_2, t)| 
= |\K(x_1, t) - \K(x_2,t)|
\leq \|\p_x \K\|_{L^\infty_{x,t}} |x_1-x_2| 
\leq  C^{2n}_0 \bar{C}_x B_k \eps  |y_1-y_2|^{\frac{1}{2n+1}}.
\end{equation*}
Analogous computations also give us uniform $C^{0,\frac{1}{2n+1}}_y$ estimates on $\p_y z$.

It is also straightforward consequence of our estimates from \S~\ref{sec:HOE} that $w \circ \eta, z\circ \eta,$ and $k \circ \eta$ are $C^{2n+1,1}_{x,t}$  on $\TT \times [0,T_*]$ and the inverse map $(x,t) \rightarrow (y,t)$ is locally $C^{2n+1,1}_{x,t}$ away from $(x_*,T_*)$, so $w,z,$ and $k$ are locally $C^{2n+1,1}_{y,t}$ on $\TT \times [0,T_*] \setminus \{(y_*,T_*)\}$. 
\end{proof}

\appendix


\section{Polynomial inversion}
\label{sec:appendix:poly}

This section generalizes the results from the appendix of a previous work of the authors \cite{NeShVi2023}. Let $\CC((z))$ denote the field of formal Laurent series in the variable $z$ with coefficients in $\CC$, i.e. the field of formal power series with coefficients in $\CC$ that also allow for finitely many terms of negative degree. The field of {\it{Puiseux series}} in the variable $x$ with coefficients in $\CC$ is then defined to be the union $ \bigcup_{j > 0} \CC((x^{1/j})) $, which is itself a field. The Puiseux-Newton theorem states that $\bigcup_{j >0} \CC((x^{1/j}))$  is in fact an algebraically closed field. We will now introduce a useful special case of the Puiseux-Newton theorem:  

\begin{theorem}[Analytic Puiseux-Newton]
\label{APN}
If $\CC\{x\}$ denotes the ring of convergent power series in $x$, and $f(x,y) \in \CC\{x\}[y]$ is a polynomial of degree $m> 0$, irreducible in $\CC\{x\}[y]$, then there exists a convergent power series $y \in \CC\{z\}$ such that the roots of $f$ in $\bigcup_{j > 0} \CC(( x^{1/j}))$ are all given by
\begin{equation*}
y(x^{1/m}), y(e^{2\pi i/m} x^{1/m}), \hdots, y(e^{2\pi i\tfrac{m-1}{m}} x^{1/m}).
\end{equation*}
It follows that in general if $f(x,y) \in \CC\{x\}[y]$ then for each Puiseux series solution $y$ of $f(x,y(x)) = 0$ there exists some $\bar{y} \in \CC\{z\}$ and $m \leq \deg f$ such that $y(x) = \bar{y}(x^{1/m})$.
\end{theorem}

\begin{proof}[Proof of Theorem~\ref{APN}]
See~\cite[Section 8.3]{BrKn1986}. 
\end{proof}

Now, fix a positive integer $n$. Recursively define the sequence
\begin{align*}
c^n_0 & : = 1, \\
c^n_m & : = \sum_{\ell_1+\cdots + \ell_{2n+2} = m-1} c^n_{\ell_1} \cdots c^n_{\ell_{2n+2}} - \tfrac{1}{2n+1} \sum_{\substack{j_1 + \cdots j_{2n+1} = m \\ 0 \leq j_i \leq m-1}} c^n_{j_1} \cdots c^n_{j_{2n+1}},
\end{align*}
and then define the formal power series
\begin{equation*}
 \bar{y}_n(x) : = \sum^\infty_{j=0} \tfrac{(-1)^j}{(2n+1)^j} c_j x^{j}. 
\end{equation*}
One can check that
\begin{equation*}
 y_0(x) : = x^{\frac{1}{2n+1}}\bar{y}_n(x^{\frac{1}{2n+1}}) 
 \end{equation*}
is a Puiseux series solution of the algebraic equation
\begin{equation*}
 -x + y_0^{2n+1} + y^{2n+2}_0 = 0. 
 \end{equation*}
It follows from Theorem \ref{APN} that $\bar{y}_n$ is a convergent power series with some positive radius of convergence 
\begin{equation*}
 R_n : = \frac{2n+1}{\limsup_{j \rightarrow \infty} |c^n_j|^{1/j}} > 0. 
 \end{equation*}

If $a_{2n+1} \in \R^\times,a_{2n+2} \in \R$ one can check that 
\begin{align*}
y(x) & : = (\tfrac{x}{a_{2n+1}})^{\frac{1}{2n+1}}\bar{y}_n\big( \tfrac{a_{2n+2}}{a_{2n+1}} (\tfrac{x}{a_{2n+1}})^{\frac{1}{2n+1}}\big) \\
&  = (\tfrac{x}{a_{2n+1}})^{\frac{1}{2n+1}} \sum^\infty_{j=0} \big( \tfrac{-a_{2n+2}}{(2n+1)a_{2n+1}} \big)^j c^n_j (\tfrac{x}{a_{2n+1}})^{\frac{j}{2n+1}} 
\end{align*}
solves
\begin{equation*}
 -x + a_{2n+1} y^{2n+1} + a_{2n+2} y^{2n+2} = 0 
 \end{equation*}
for all $x$ satisfying
\begin{equation*}
 |a_{2n+2}|^{2n+1} |x| < a_{2n+1}^{2n+2} R_n^{2n+1}. 
 \end{equation*}

Since $\bar{y}_n$ is holomorphic, if $0 < r< R_n$ then Cauchy's estimate gives us
\begin{align*}
|\tfrac{c_j}{(2n+1)^j}| & \leq r^{-j} \max_{|z|=r|} |\bar{y}_n(z)|
\end{align*}
for all nonnegative integers $j$. It follows that for $0< r < R_n$ and $N \geq 0$, we have
\begin{align*}
\big\vert\bar{y}_n(z) - \sum^N_{j=0} \tfrac{(-1)^j}{(2n+1)^j} c_j z^{j} \big\vert & \leq \frac{|\tfrac{z}{r}|^{N+1}}{1-|\tfrac{z}{r}|} \max_{|z| = r} |\bar{y}_n(z)| & \forall \: |z| < r.
\end{align*}
Therefore, if we define
\begin{equation}\label{def:M_n}
M_n : =  \max_{|z| = \frac{3R_n}{4}} |\bar{y}_n(z)|,
\end{equation}
then for any $N \geq 0$ we get the bound
\begin{align}\label{ineq:appendix:N}
\big\vert y(x) - (\tfrac{x}{a_{2n+1}})^{\frac{1}{2n+1}} \sum^N_{j=0} \big( \tfrac{-a_{2n+2}}{(2n+1)a_{2n+1}} \big)^j c^n_j (\tfrac{x}{a_{2n+1}})^{\frac{j}{2n+1}} \big\vert & \leq  3M_n (\tfrac{4a_{2n+2}}{3a_{2n+1} R_n})^{N+1} |\tfrac{x}{a_{2n+1}}|^{\frac{N+2}{2n+1}} \\
\forall \: x \text{ s.t. } \: |a_{2n+2}|^{2n+1} |x| &  < a_{2n+1}^{2n+2} (\tfrac{R_n}{2})^{2n+1}. \notag
\end{align}

\begin{lemma}\label{lem:appendix:1}
Let $n$ be a positive integer, let $I \subset \R$ be an interval, and let $y \in C^{2n+1,1}(I)$. Suppose $y$ has a Taylor series expansion about $x_* \in I$ of the form
\begin{equation*}
y(x) = y_* + a_{2n+1}(x-x_*)^{2n+1} + a_{2n+2}(x)(x-x_*)^{2n+2}
\end{equation*}
where $a_{2n+1} > 0$. Define $J \subset I$ to be 
\begin{equation*}
J : = \{ x \in I : |a_{2n+2}(x)|^{2n+1} |y(x)-y_*| < a_{2n+1}^{2n+2} R^{2n+1}_n, \quad \tfrac{(2n+2)a_{2n+2}(x) (x-x_*)}{(2n+1)a_{2n+1}} >-1 \},
\end{equation*}
and define $\tilde{J}$ to be the connected component of $J$ containing $x_*$. Then for all $x \in \tilde{J}$ we have 
\begin{equation}\label{eq:appendix:fracseries}
(x-x_*) = \big(\tfrac{y(x)-y_*}{a_{2n+1}}\big)^{\frac{1}{2n+1}} \sum^\infty_{j=0} \big( \tfrac{-a_{2n+2(x)}}{(2n+1)a_{2n+1}} \big)^j c^n_j \big(\tfrac{y(x)-y_*}{a_{2n+1}}\big)^{\frac{j}{2n+1}} .
\end{equation}
\end{lemma}

\begin{proof} Without loss of generality, $x_* = y_* = 0$ . For $(y,a) \in \R^2$ satisfying
$$ |a|^{2n+1} |y| < a_{2n+1}^{2n+2} R^{2n+1}_n $$
define the function $\tilde{x}$ via the series
$$ \tilde{x}(y,a) : = \big(\tfrac{y}{a_{2n+1}}\big)^{\frac{1}{2n+1}} \sum^\infty_{j=0} \big( \tfrac{-a}{(2n+1)a_{2n+1}} \big)^j c^n_j \big(\tfrac{y}{a_{2n+1}}\big)^{\frac{j}{2n+1}} $$
which we know to converge. For $(y,a)$ in the domain of $\tilde{x}$, $\tilde{x}(y,a)$ solves
\begin{equation*}
\label{eq:appendix:xtilde}
-y + a_{2n+1} \tilde{x}(y,a)^{2n+1} + a\tilde{x}(y,a)^{2n+2} = 0.
\end{equation*}
For all $x \in J$, $(y(x), a_{2n+2}(x))$ is in the domain of $\tilde{x}$ and we have
\begin{align}
\label{eq:appendix:xtilde:J}
y(x) & =  a_{2n+1} \tilde{x}(y(x),a_{2n+2}(x)))^{2n+1} + a_{2n+2}(x)\tilde{x}(y(x),a_{2n+2}(x))^{2n+2} = 0.
\end{align}

For $x \in \tilde{J}$, define the bootstrap hypothesis and conclusion
\begin{align*}
H(x) & : = `` \tfrac{(2n+2)a_{2n+2}(x)\tilde{x}(y(x), a_{2n+2}(x))}{(2n+1)a_{2n+1}}  > -1 " \\
C(x) & : = ``\tilde{x}(y(x), a_{2n+2}(x)) =x. "
\end{align*}
We note that
\begin{enumerate}[leftmargin=*]
\item $H(0)$ is true,
\item the set of $x \in \tilde{J}$ where $C(x)$ is true is a closed subset of $\tilde{J}$ because $x \rightarrow \tilde{x}(y(x), a_{2n+2}(x))$ is continuous, and
\item the set of $x \in \tilde{J}$ where $C(x)$ is true is in the interior of the set where $H(x)$ is true, by virtue of the definition of $J$ and the fact that  $x \rightarrow \tilde{x}(y(x), a_{2n+2}(x))$ is continuous.
\end{enumerate}
Lastly, the set where $H(x)$ is true contains the set where $C(x)$ is true. To see this, for $(x,z) \in I \times \R$ satisfying 
\begin{equation*}
\tfrac{(2n+2)a_{2n+2}(x)z}{(2n+1)|a_{2n+1}|} >-1
\end{equation*}
define the function
\begin{equation*}
f(x,z) : = a_{2n+1} z^{2n+1} + a_{2n+2}(x) z^{2n+2}.
\end{equation*}
For all $(x,z)$ in the domain of $f$ we have
\begin{align*}
\p_zf(x,z) & = (2n+1)a_{2n+1} z^{2n}[1+\tfrac{(2n+2)a_{2n+2}(x)z}{(2n+1)a_{2n+1}}]
\end{align*}
so that $f$ is a strictly increasing function of $z$ everywhere in its domain. If $x \in \tilde{J}$ is such that $H(x)$ is true, then $\tilde{x}(y(x), a_{2n+2}(x))$ is in the domain of $f$ and \eqref{eq:appendix:xtilde:J} implies that
\begin{equation*}
f(x, \tilde{x}(y(x), a_{2n+2}(x))) = y(x) = f(x,x).
\end{equation*}
Since $f$ is a strictly increasing function of its second argument on its domain, we conclude that $C(x)$ is true. This completes our bootstrap argument, and since $\tilde{J}$ is connected we conclude that $C(x)$ is true for all $x \in \tilde{J}$.
\end{proof}

\begin{lemma}\label{lem:appendix:fracseries}
In the context of the previous lemma, if we further assume that  $I = [-r,r]$ and $r$ satisfies
\begin{subequations}\label{hyp:appendix:lem}
\begin{align}\label{hyp:appendix:1}
\tfrac{(2n+2)\|a_{2n+2}\|_{L^\infty_x} }{(2n+1)a_{2n+1}}(r+|x_*|) & < 1, \\
\label{hyp:appendix:2}
\tfrac{\|a_{2n+2}\|_{L^\infty_x}}{a_{2n+1} R_n} (r+|x_*|)  & < \tfrac{1}{4},
\end{align}
\end{subequations}
then for all $x \in I$ we have
\begin{equation}\label{ineq:appendix:fracseries}
\bigg\vert (x-x_*)-\big(\tfrac{y(x)-y_*}{a_{2n+1}}\big)^{\frac{1}{2n+1}} \bigg\vert \leq \tfrac{4M_n \|a_{2n+2}\|_{L^\infty_x}}{a_{2n+1} R_n} |\tfrac{y(x)-y_*}{a_{2n+1}}|^{\frac{2}{2n+1}}.
\end{equation}
Furthermore, the range of $y$ contains the ball
\begin{align}\label{ball:y}
|y-y_*| & \leq  \tfrac{a_{2n+1} (r-|x_*|)^{2n+1}}{2n+2}.
\end{align}
\end{lemma}

\begin{proof}
It is immediate from \eqref{hyp:appendix:1} that
\begin{equation*}
|y(x)-y_*| \geq \tfrac{a_{2n+1}}{2n+2} |x-x_*|^{2n+1}
\end{equation*}
for all $x \in I$. This implies that the range of $y$ contains the ball defined by \eqref{ball:y}. Furthermore, \eqref{hyp:appendix:lem} implies that $J=I$ and therefore $\tilde{J} = I$. Lemma \ref{lem:appendix:1} therefore implies that  \eqref{eq:appendix:fracseries} holds for all $x \in I$. Now our assumption \eqref{hyp:appendix:2} allows us to apply \eqref{ineq:appendix:N} with $N=1$ to obtain \eqref{ineq:appendix:fracseries}.
\end{proof}



\section{Frequently used computations}
\label{sec:appendix:freqcomp}

The following lemmas are straightforward exercises:

\begin{lemma}\label{lem:ODE}
If $A > 0$ and $f \in \RR$ are constants, and $x:[0,T] \rightarrow \RR$ satisfies
\begin{equation*}
\dot{x} \leq -A x + f
\end{equation*}
for all $t \in [0,T]$, then
\begin{align*}
x(t) & \leq x(0) e^{-At} + (1-e^{-At})\tfrac{f}{A} \leq \max(x(0), \tfrac{f}{A}) & \forall \: t \in [0,T].
\end{align*}
\end{lemma}

\begin{lemma}\label{lem:id:multi}
If $\beta$ is a multi-index and $j_1+\hdots+j_m = |\beta|$, then
\begin{equation}
 \sum_{\substack{\gamma_1 + \cdots + \gamma_m = \beta \\ |\gamma_i| = j_i}} {\beta\choose{\gamma_1 \cdots \gamma_m}} = {|\beta|\choose{j_1 \cdots j_m}}.
\end{equation}
\end{lemma}

\begin{lemma}\label{lem:ineq:sum:nk}
For all positive integers $m,\ell$ and all $p > 1$, we have the bound
\begin{align*}
\sum_{j_1+\hdots+j_\ell = m} \tfrac{(m+1)^p}{(j_1+1)^p \cdots (j_\ell+1)^p} & < \ell^{1+p} \bigg( \sum^\infty_{j=1} \tfrac{1}{j^p} \bigg)^{\ell-1}.
\end{align*}
In particular, for $p=2$ and $\ell = 2,3$ we have
\begin{align}\label{ineq:sum:nk}
&\sum^m_{j=0} \tfrac{(m+1)^2}{(j+1)^2(m+1-j)^2} < \tfrac{4\pi^2}{3}, &&\sum_{j_1+j_2+j_3 = m} \tfrac{(m+1)^2}{(j_1+1)^2(j_2+1)^2(j_3+1)^2} < \tfrac{3\pi^4}{4}.
\end{align}
\end{lemma}

Many variants of the inequalities \eqref{ineq:sum:nk} can be proven via similar computations: 
\begin{equation}\label{ineq:appendix:B}
\begin{aligned}
& \sum^m_{j=0} \tfrac{m+2}{(j+1)^2(m+2-j)} < \tfrac{2\pi^2}{3},
&& \sum^m_{j=1} \tfrac{m+2}{j^3(m+2-j)} < 8,
&&& \sum^m_{j=2} \tfrac{m+1}{j^3(m+2-j)} < 6.1,\\
& \sum^m_{j=1} \tfrac{(m+1)^2}{j^3 (m+1-j)^2} < 25, 
&& \sum^{m-1}_{j=0} \tfrac{(m+1)^2}{(j+1)^2(m-j)^2} < \tfrac{16 \pi^2}{3},
&&& \sum_{\substack{j_1 + j_2 + j_3 = m \\ j_1 \geq 1}} \tfrac{(m+1)^2}{j_1^3 (j_2+1)^2 (j_3+1)^2} < 163 \\
& \sum_{\substack{j_1 + j_2 + j_3 = m \\ j_1 \geq 1}} \tfrac{m+2}{j_1^3 (j_2+1)^2 (j_3+2)} < 21, 
&& \sum_{\substack{j_1 + j_2 + j_3 = m \\ j_1, j_2 \geq 1}} \tfrac{(m+1)^2}{j_1^3 j^3_2 (j_3+1)^2} < 66, 
&&& \sum_{\substack{j_1 + j_2 + j_3 + j_4 = m \\ j_1, j_2 \geq 1}} \tfrac{(m+1)^2}{j^3_1 j^3_2 (j_3+1)^2(j_4+1)^2} < 641.
\end{aligned}
\end{equation}
These inequalities \eqref{ineq:sum:nk}--\eqref{ineq:appendix:B} are used throughout \S~\ref{sec:HOE}-\ref{sec:stab}. Note that the righthand side of the inequalities  \eqref{ineq:sum:nk}--\eqref{ineq:appendix:B} is always {\it{independent of $m$}}.


\section{Implicit function theorem}
\label{sec:appendix:IFT}

In this section, suppose that $X$ is a Banach space, $U \subset X$ is an open set, and $V \subset \R^N$ is an open, convex set with $B_r(y_0) \subseteq V \subseteq B_R(y_0)$ for some $y_0 \in \R^N, 0 < r \leq R$. Let $f : U \times V \rightarrow \R^N$.

\begin{lemma}[Existence of unique solutions]\label{lem:IFT:1}
Suppose that $f$ is Lipschitz in $y$ and that there exists a constant $\theta \in [0, 1)$ such that for all $x \in U$ we have
\begin{subequations}
\begin{align}\label{hyp:IFT:1}
\| \Id - Df_y(x, \cdot) \|_{L^\infty_y(V)} & \leq \theta, \\
\label{hyp:IFT:2}
|f(x,y_0)| & \leq r- \theta R.
\end{align}
\end{subequations}
Then there exists a function $g:  U \rightarrow V$ such that 
\begin{align*}
\bigg\{ (x,y) \in U \times V : f(x,y) = 0 \bigg\} = \bigg\{ (x, g(x)) : x \in U \bigg\}.
\end{align*}
\end{lemma}

\begin{proof}
For each $x \in U$, define the map $\Psi(x) : V \rightarrow \R^N$, $\Psi(x)(y) : = y- f(x,y)$. Since $V$ is convex, it follows from \eqref{hyp:IFT:1} that $\Psi(x)$ is $\theta$-Lipschitz for all $x \in U$. Therefore, \eqref{hyp:IFT:2} implies that
\begin{align*}
\Psi(x)(V) \subset B_{\theta R}( \Psi(x)(y_0)) \subset B_r(y_0) \subset V.
\end{align*}
Thus $\Psi(x)$ is a contraction mapping for each $x \in U$. Define $g(x)$ to be the unique fixed point of $\Psi(x)$.
\end{proof}

\begin{lemma}[Lipschitz implicit function theorem]
In the context of the previous lemma, if we additionally assume that $f$ is uniformly Lipschitz in $x$ with
\begin{align}\label{hyp:IFT:3}
|f(x_1,y)-f(x_2,y)| & \leq L|x_1-x_2| && \forall \: x_1, x_2 \in U, y \in V,
\end{align}
then the function $g$ is $\frac{L}{1-\theta}$-Lipschitz.
\end{lemma}

\begin{proof}
Recursively define the sequence of functions
\begin{equation*}
\begin{cases} g_0(x)  : =  x \\
g_{n+1}(x)  :=  \Psi(x)(g_n(x)) \end{cases} .
\end{equation*}
It is immediate that $[g_{n+1}]_{C^{0,1}_x(U)} \leq \theta [g_{n}]_{C^{0,1}_x(U)} + L$ for all $n$. Therefore, we have
\begin{equation*}
[g_{n+1}]_{C^{0,1}_x(U)} \leq L \sum^n_{j=0} \theta^j < \frac{L}{1-\theta}
\end{equation*}
for all $n$. Since $g_n$ converges pointwise to $g$, our result follows.
\end{proof}

\section*{Acknowledgements}
The work of I.N.~was in part supported by the NSF Postdoctoral Fellowship DMS-2401743.
The work of S.S.~was in part supported by the Collaborative NSF grant DMS-2307680. The work of V.V. was in part supported by the Collaborative NSF grant DMS-2307681 and a Simons Investigator Award.


\end{document}